\documentclass[11pt]{amsart}
\usepackage[dvipsnames]{xcolor}
\usepackage[pagebackref,backref=true,colorlinks=true, linkcolor=Sepia, citecolor=Sepia]{hyperref}
\usepackage{amsmath,amsthm,amssymb,fancyhdr,graphicx,bbm,cancel,mathrsfs,todonotes,xypic,pinlabel}
\usepackage{cleveref}
\usepackage{tikz}
\usetikzlibrary{external}
\usepackage{extarrows,tipa}
\usetikzlibrary{matrix}
\usepackage[all]{xy}
\usepackage{mathtools}
\usepackage{bm}
\usepackage{verbatim}
\usepackage{microtype}
\usepackage{subcaption}
\usepackage{tikz}
\usepackage{tikz-cd}

\theoremstyle{theorem}
\newtheorem{theorem}{Theorem}[section]
\newtheorem{theoremalpha}{Theorem}

\newtheorem{lemma}[theorem]{Lemma}
\newtheorem{proposition}[theorem]{Proposition}
\newtheorem{corollary}[theorem]{Corollary}
\theoremstyle{definition} %The commands below have bold title, standard text
\newtheorem{definition}[theorem]{Definition}
\newtheorem{example}[theorem]{Example}

\newtheorem{question}[theorem]{Question}

\newtheorem{conjecture}[theorem]{Conjecture}
\newtheorem{remark}[theorem]{Remark}
\theoremstyle{remark}

\usepackage{enumerate,pdfpages,stmaryrd}
\usepackage[margin=1.11in, marginparwidth=1in]{geometry}
\usepackage{tikz}
\usepackage{dsfont}
\usetikzlibrary{matrix}

\renewcommand{\xto}{\xrightarrow}

\newcommand{\C}{\mathbb{C}}

\newcommand{\calA}{\mathcal{A}}

\newcommand{\calB}{\mathcal{B}}
\newcommand{\calS}{\mathcal{S}}

\newcommand{\calP}{\mathcal{P}}

\newcommand{\Z}{\mathbb{Z}}
\newcommand{\Q}{\mathbb{Q}}
\newcommand{\R}{\mathbb{R}}
\newcommand{\K}{\mathbb{K}}
\newcommand{\F}{\mathbb{F}}

\newcommand{\Hyp}{\mathbb{H}}

\newcommand{\beq}{\begin{equation*}}
\newcommand{\eeq}{\end{equation*}}

\newcommand{\calO}{\mathcal{O}}

\newcommand{\On}{\mathrm{O}}

\newcommand{\GL}{\mathrm{GL}}

\DeclareMathOperator{\Diff}{Diff}

\DeclareMathOperator{\Sp}{Sp}
\DeclareMathOperator{\Aut}{Aut}

\DeclareMathOperator{\PL}{PL}
\DeclareMathOperator{\Isom}{Isom}

\DeclareMathOperator{\SO}{SO}

\newcommand{\dmo}{\DeclareMathOperator}

\newcommand{\al}{\alpha}\newcommand{\ga}{\gamma}\newcommand{\ep}{\epsilon}\newcommand{\si}{\sigma}
\newcommand{\Om}{\Omega}\newcommand{\ka}{\kappa}

\newcommand{\what}{\widehat}\newcommand{\wtil}{\widetilde}

\newcommand{\bs}{\backslash}

\newcommand{\ra}{\rightarrow}

\newcommand{\bb}[1]{\mathbb{#1}}\newcommand{\ov}[1]{\overline{#1}}

\newcommand{\wh}[1]{\widehat{#1}}

\dmo{\sgn}{sign}\dmo{\Span}{span}
\dmo{\we}{\wedge}
\dmo{\ind}{ind}\dmo{\Ind}{Ind}
\dmo{\bop}{\bigoplus}\dmo{\pic}{Pic}
\dmo{\vol}{Vol}\dmo{\gal}{Gal}\dmo{\perm}{Perm}
\dmo{\tor}{Tor}\dmo{\ext}{Ext}\dmo{\Ext}{Ext}
\dmo{\aut}{aut}

\dmo{\inn}{Inn}\dmo{\var}{Var}
\dmo{\ad}{ad}\dmo{\curl}{curl}
\dmo{\hy}{\bb H}\dmo{\Sl}{SL}
\dmo{\psl}{PSL}
\dmo{\iso}{iso}
\dmo{\conf}{Conf}
\dmo{\stab}{Stab}\dmo{\Jac}{Jac }
\dmo{\diam}{diam}\dmo{\fix}{Fixed}\dmo{\Fix}{Fix}
\dmo{\injR}{injRad}\dmo{\Ad}{Ad}
\dmo{\esv}{ess-vol}
\dmo{\nil}{Nil}\dmo{\sol}{Sol}
\dmo{\Div}{div}
\dmo{\SU}{SU}
\dmo{\rk}{rk}
\dmo{\rank}{rank}
\dmo{\psp}{PSp}\dmo{\psu}{PSU}
\dmo{\PU}{PU}\dmo{\pgl}{PGL}
\dmo{\Mod}{Mod}\dmo{\range}{Range}
\dmo{\eu}{eu}\dmo{\mi}{mi}
\dmo{\Log}{Log}\dmo{\supp}{supp}
\dmo{\maps}{Maps}\dmo{\Gr}{Gr}
\dmo{\Pin}{Pin}
\dmo{\Spin}{Spin}\dmo{\Str}{Str}
\dmo{\Sq}{Sq}\dmo{\Symp}{Symp}
\dmo{\pd}{PD}\dmo{\PD}{PD}\dmo{\sig}{Sig}
\dmo{\ev}{ev}\dmo{\St}{St}
\dmo{\Pt}{Pt}\dmo{\pt}{pt}
\newcommand{\msf}{\mathsf}

\dmo{\Pl}{PL}
\dmo{\String}{String}\dmo{\smear}{smear}

\dmo{\dev}{dev}
\dmo{\met}{Met}\dmo{\contact}{Contact}
\dmo{\teich}{Teich}\dmo{\Teich}{Teich}\dmo{\qi}{QI}
\dmo{\der}{Der}
\dmo{\cl}{Cliff}\dmo{\Cl}{Cl}
\dmo{\Pf}{Pf}
\dmo{\ch}{ch}\dmo{\diag}{diag}
\dmo{\grad}{grad}\dmo{\Char}{char}
\dmo{\spec}{Spec}\dmo{\Arg}{Arg}
\dmo{\gl}{GL}
\dmo{\sym}{Sym}\dmo{\Sym}{Sym}
\dmo{\com}{Comm}
\dmo{\Lk}{Lk}
\dmo{\CAT}{CAT}
\dmo{\Rep}{Rep}
\dmo{\Res}{Res}
\dmo{\Conf}{Conf}
\dmo{\PConf}{PConf}
\dmo{\Push}{Push}
\dmo{\Cont}{Cont}
\dmo{\sm}{\setminus}
\dmo{\vn}{\varnothing}
\dmo{\disk}{\mathbb D}
\dmo{\Trd}{Trd}\dmo{\Mat}{Mat}
\dmo{\Riem}{Riem}
\dmo{\Diffn}{\Diff_0}\dmo{\diff}{diff}
\dmo{\homeo}{Homeo}

\dmo{\Ham}{Ham}\dmo{\Met}{Met}
\dmo{\Ein}{Ein}\dmo{\CP}{\co P}
\dmo{\Per}{Per}\dmo{\Ric}{Ric}
\dmo{\Nrd}{Nrd}
\dmo{\Comp}{Comp}\dmo{\PSC}{PSC}
\dmo{\Cent}{Cent}\dmo{\Orb}{Orb}
\dmo{\aind}{a-ind}\dmo{\tind}{t-ind}
\dmo{\constant}{constant}
\dmo{\Td}{Td}
\dmo{\LMod}{LMod}
\dmo{\SMod}{SMod}
\dmo{\SDiff}{SDiff}
\dmo{\Br}{Br}
\dmo{\csch}{csch}
\dmo{\triv}{triv}
\dmo{\genus}{genus}
\dmo{\Homeq}{HomEq}
\dmo{\PP}{\mathbb{P}}
\dmo{\U}{U}
\dmo{\Gal}{Gal}
\dmo{\BDiff}{\wtil{\Diff}}
\dmo{\BAut}{\wtil{\Aut}}
\dmo{\Iso}{Iso}
\dmo{\Cone}{Cone}
\dmo{\codim}{codim}
\dmo{\II}{II}
\dmo{\I}{I}
\dmo{\InjRad}{InjRad}
\dmo{\Inn}{Inn}
\dmo{\sys}{sys}
\dmo{\Comm}{Comm}
\dmo{\PO}{PO}
\dmo{\vertex}{Vert}
\dmo{\POm}{P\Om}
\dmo{\ab}{ab}
\dmo{\PSO}{PSO}
\dmo{\CRS}{CRS}
\dmo{\Diffext}{Diffext}
\dmo{\Diffextad}{Diffextad}
\dmo{\Diffstand}{Diffstand}

\newcommand{\G}{\Gamma}

\newcommand{\gam}{\gamma}
\newcommand{\del}{\delta}
\newcommand{\lam}{\lambda}

\newcommand{\bbF}{\mathbb{F}}

\newcommand{\Hy}{\mathbb{H}}

\newcommand{\bbT}{\mathbb{T}}

\newcommand{\bbZ}{\mathbb{Z}}
\newcommand{\calD}{\mathcal{D}}

\newcommand{\calG}{\mathcal{G}}

\newcommand{\calH}{\mathcal{H}}

\newcommand{\calI}{\mathcal{I}}

\newcommand{\calQ}{\mathcal{Q}}

\newcommand{\calY}{\mathcal{Y}}
\newcommand{\calZ}{\mathcal{Z}}

\newcommand{\wt}[1]{\widetilde{#1}}

\author{Mauricio Bustamante}
\address{Departamento de Matem\'aticas, Pontificia Universidad Cat\'olica de Chile}
\email{mauricio.bustamante@uc.cl}

\author{Eduardo Reyes}
\address{Departamento de Matem\'aticas, Pontificia Universidad Cat\'olica de Chile}
\email{eduardoreyes@uc.cl}

\author{Stefano Riolo}
\address{Dipartimento di Matematica, Università di Bologna}
\email{stefano.riolo@unibo.it}
\begin{document}
\title{Stably tangential strict hyperbolization}
%\author{Mauricio Bustamante, Eduardo Reyes, and Stefano Riolo}

\maketitle
\begin{abstract} We show that the Charney--Davis strict hyperbolization procedure can preserve stable tangent bundles, answering a question of Charney and Davis. The key input is the construction of many hyperbolizing pieces, obtained using separability properties of hyperbolic cubulable groups. Moreover, these pieces may be chosen so that every face is connected, answering a question of Belegradek. We then apply this construction to suitable cubulations of flat manifolds to produce infinitely many commensurability classes of closed hyperbolic manifolds, both arithmetic and non-arithmetic, with diverse topological features. In particular, we obtain the first examples in which all the Stiefel--Whitney classes are non-trivial below the top degree, and the first orientable examples with non-trivial Pontryagin classes. 
We also construct infinite towers of finite covers of closed hyperbolic manifolds in which no cover is stably parallelizable or spin. Our methods further yield new pairs of exotic negatively curved Riemannian manifolds.
\end{abstract}

\section{Introduction}
Closed Riemannian manifolds of negative sectional curvature play a prominent role in manifold topology. They are known to satisfy the Borel conjecture \cite{farrell-jones.borelneg} if their dimension is not 4, so their topological type is completely determined by their fundamental group. Hence, much of their topology is reflected in the algebra and geometry of that group. Moreover, these manifolds exhibit various forms of geometric and topological rigidity, which provide extra tools for understanding their topology. In this way, they form a natural bridge between manifold topology and geometric group theory.

The fundamental examples of negatively curved manifolds are hyperbolic manifolds, characterized by having constant sectional curvature $-1$. Closely related are the rank-one locally symmetric spaces, such as complex and quaternionic hyperbolic manifolds, whose sectional curvatures vary between $-4$ and $-1$. Other examples of variable negative curvature arise from branched covers of (real and complex) hyperbolic manifolds \cite{mostow-siu,gromov-thurston,deraux,stover-toledo} and from suitably altering locally symmetric spaces, for instance by taking connected sums with exotic spheres (see e.g.~\cite{farrell-jones.negatively}).

There is, however, another source of examples, arguably of a different nature: the \textit{strict hyperbolization} procedure of Charney--Davis \cite{charney-davis}. This is a two-step process: first, a simplicial complex $\msf{N}$ is replaced by a non-positively curved cube complex $\calG(\msf{N})$, in a construction introduced by Gromov \cite[Section~3.4]{gromov.hypgroups} and later studied by Davis and Januszkiewicz \cite{davis-januszkiewicz}. The second step consists of turning $\calG(\msf{N})$ into a negatively curved space. In fact, Charney and Davis Charney and Davis \cite{charney-davis} design a procedure to turn a cube complex  $\msf{C}$ into a piecewise hyperbolic space. Moreover if the cube complex is non-positively curved then the resulting space is a locally $\mathrm{CAT}(-1)$ space. 
They build such a space by replacing in a consistent way the cubes of $\msf{C}$  by the faces of a certain hyperbolic manifold with right-angled corners $X$ which has the symmetries of a cube. We call such an $X$ a \textit{hyperbolizing piece} (see Section \ref{sec:strict}). The geometry and topology of the hyperbolizing piece influence that of the resulting negatively curved space, which we will denote by $\calH_X(\msf{C})$ and occasionally call the \textit{hyperbolized complex}. Indeed, a further refinement of this strict hyperbolization was provided by Ontaneda \cite{ontaneda}. He showed that when the simplicial complex $\msf{N}$ is a smooth manifold  and the hyperbolizing piece $X$ is chosen to have all its strata with ``large'' normal injectivity radius, then the singular locally $\CAT(-1)$ metric on the hyperbolized complex $\calH_X(\calG(\msf{N}))$ can be smoothed to a genuine negatively curved Riemannian metric.
This refinement of strict hyperbolization produces a negatively curved Riemannian representative in every cobordism class of smooth manifolds, and hence a vast new supply of examples. It also makes it possible to realize prescribed topological features, such as non-trivial rational Pontryagin classes or arbitrarily large Betti numbers, on manifolds of negative curvature.  Indeed, the strict hyperbolization of a non-positively curved smoothly cubulated manifold $\msf{C}$ yields not only a locally $\mathrm{CAT}(-1)$ manifold  $\calH_X(\msf{C})$, but also a \textit{hyperbolization map} 
\beq
g_{\msf{C}} \colon \calH_X(\msf{C}) \to \msf{C}.
\eeq
This map is known to induce a surjection in singular homology with coefficients in any abelian group and, if $\msf{C}$ is $R$-orientable for a commutative ring $R$, then so is $\calH_X(\msf{C})$ and $g_\msf{C}$ has degree 1. Thus, by Poincar\'e duality, it induces an injective map in cohomology with coefficients in $R$ \cite[Proposition 7.1]{charney-davis}.
Moreover, if the cube complex $\msf{C}$ is \textit{foldable}, that is,
it admits a combinatorial map $\msf{C}\to\square^n$ to the standard cube which is injective on each cell, then the hyperbolized complex admits a smooth structure for which the map $g_{\msf{C}}$ preserves the rational Pontryagin classes \cite[Proposition 7.2]{charney-davis}, \cite[Addendum]{ontaneda.smoothing}.

In this work, we show that there is an abundance of  hyperbolizing pieces $X$ for which the hyperbolization map has better tangential properties. We fix some notation before we state our results. For a smooth manifold $M$, we denote its tangent bundle by $TM$ and the trivial real vector bundle over $M$ of rank $k$ by $\varepsilon^k_M$, or simply $\varepsilon^k$ if the base space is obvious from the context. We say that $M$ is \textit{stably parallelizable} if the Whitney sum $TM\oplus\varepsilon^k$ is trivializable for some $k\geq 0$. A homeomorphism $\msf{C}\to M$ from a foldable cube complex $\msf{C}$ to $M$ which restricts to a smooth embedding on each cube of $\msf{C}$ is called a \textit{foldable smooth cubulation} of $M$. In that case, we may identify $M$ with $\msf{C}$.
\begin{theoremalpha}\label{thm:main-stably-tan} 
Given any integer $n>0$, there exist $n$-dimensional hyperbolizing pieces $X$
with the following property. Let $\msf{C}$ be a compact $n$-dimensional smooth manifold (possibly with boundary) equipped with a foldable smooth cubulation. Then the hyperbolization map $g_\msf{C}:\mathcal{H}_X(\msf{C})\to \msf{C}$ is stably tangential, that is, there exists an isomorphism of vector bundles
\beq
g_{\msf{C}}^*(T\msf{C}\oplus\varepsilon^k)\cong T\calH_X(\msf{C})\oplus\varepsilon^k 
\eeq
for some $k\geq 0$.
\end{theoremalpha}
The smooth structures considered here are Ontaneda's normal smooth structures \cite{ontaneda.cube, ontaneda.smoothing}, see also Lemma \ref{cor:tangential} and Section \ref{sec:smoothstructures}.
The question of whether strict hyperbolization could preserve stable tangent bundles (which, as we will see in Lemma \ref{cor:tangential}, amounts to asking whether there exists a stably parallelizable hyperbolizing piece) was posed as Question 7.4 in \cite{charney-davis}, as well as in \cite[Section~3]{belegradek} and \cite[p.~2]{ontaneda.smoothing}.
Jean Lafont has informed us that he can also prove this theorem independently.

If $\msf{N}$ is a compact $n$-dimensional smoothly triangulated manifold, Gromov and Davis--Januszkiewicz's hyperbolization procedure $\msf{N}\mapsto\calG(\msf{N})$ is known to produce a manifold equipped with a foldable smooth cubulation $\calG(\msf{N})$ \cite[Lemma 7.5]{charney-davis} and a stably tangential map $\calG(\msf{N})\to \msf{N}$ which induces a surjection in homology with arbitrary coefficients \cite[Theorem B, Corollary 1f.6]{davis-januszkiewicz}. In particular, if  $\msf{N}$ is $R$-orientable for some commutative ring $R$, so is $\calG({\msf{N}})$ and the map has degree 1. Therefore, for the hyperbolizing pieces $X$ of Theorem \ref{thm:main-stably-tan}, the composition
\beq
\calH_X(\calG(\msf{N}))
\xto{g_{\calG(\msf{N})}}\calG(\msf{N})\to \msf{N}
\eeq
is also a stably tangential map which induces a surjection in singular homology.
Furthermore, it will be apparent (see Theorem \ref{thmalpha:flexibleCD} below) that the hyperbolizing pieces in Theorem \ref{thm:main-stably-tan} can be chosen large enough so that Ontaneda's Riemannian hyperbolization can be applied. Putting this together, we obtain the following corollary.
\begin{corollary}\label{cor:Riemannian-hyp}
Let $N$ be a closed $n$-dimensional smooth manifold with $n\geq 2$ and let $\epsilon>0$. Then there is a closed $n$-dimensional Riemannian manifold $M$ whose sectional curvatures lie in the interval $[-1-\epsilon,-1]$, and a smooth stably tangential map $f:M\to N$ which induces a surjection in singular homology with coefficients in any abelian group. Moreover, if $N$ is $R$-orientable for some commutative ring $R$, then so is $M$ and the map $f$ has degree 1. By Poincar\'e duality, it induces an injective map in singular cohomology with coefficients in $R$. In particular, the (integral) Pontryagin classes and Stiefel--Whitney classes of $N$ are pulled back injectively to those of $M$.
\end{corollary}
These results also imply the existence of arbitrarily pinched negatively curved manifolds with (or without) all sorts of stable tangential structures, such as spin, spin$^c$, spin$^h$, string, stably almost complex, stably framed, etc.

The following corollary is a consequence of having a degree $1$ stably tangential map.
\begin{corollary}
Let $\msf{C}$ be a smooth cubulation of a closed orientable $n$-manifold and $X$ be a hyperbolizing piece as in Theorem \ref{thm:main-stably-tan}. Then $g_{\msf{C}}:\calH_X(\msf{C})\to\msf{C}$ induces a surjection on any generalized homology theory.
\end{corollary}
Note that the map $\calG({\msf{N}})\to\msf{N}$ also has this property \cite[Theorem B]{davis-januszkiewicz} and so the composition $\calH_X(\calG(\msf{N}))\to\calG(\msf{N})\to\msf{N}$ induces a surjection on any generalized homology theory.

In a different direction, hyperbolization procedures have been used classically to construct manifolds with exotic topological properties. In fact, the Gromov--Davis--Januszkiewicz (non-strict) hyperbolization \cite{davis-januszkiewicz} can be used to construct a cube complex homotopy equivalent to a closed orientable aspherical topological spin $4$-manifold with signature 8, so this manifold cannot support a $\PL$ structure, by Rokhlin's theorem. Upgrading this example, via strict hyperbolization, to a non-$\PL$ aspherical manifold with \textit{Gromov-hyperbolic} fundamental group requires the preservation of the second Stiefel--Whitney class for strict hyperbolization. This was not available before, but is now a consequence of Theorem \ref{thm:main-stably-tan}. In combination with \cite[Theorem 1.1]{lafont-ruffoni}, we obtain the following. 
\begin{corollary}\label{cor:nonPL4manifold}
There exists a closed orientable aspherical topological $4$-manifold whose fundamental group is Gromov-hyperbolic and virtually compact special, and which is not homotopy equivalent to any $\PL$ $4$-manifold.
\end{corollary}
Theorem \ref{thm:main-stably-tan} and Corollary \ref{cor:Riemannian-hyp} 
can also be used in combination with the topological non-invariance of the (integral) Pontryagin classes to show the existence of non-diffeomorphic smooth structures on arbitrarily pinched negatively curved manifolds.
\begin{corollary}\label{cor:exotic-smoothings}
Fix $\epsilon>0$ and $n\geq 9$. Then there exist pairs of closed orientable $n$-dimensional smooth manifolds $M$ and $N$ with the following properties:
\begin{enumerate}
\item $M$ and $N$ are homeomorphic but not diffeomorphic; 
\item $M$ and $N$ admit Riemannian metrics with sectional curvatures in the interval $[-1-\epsilon, -1]$; 
\item $M$ and $N$ are not homeomorphic to any locally symmetric manifold of real rank one, or to any of the Gromov--Thurston 
\cite{gromov-thurston}, or Mostow--Siu \cite{mostow-siu}, or Deraux 
\cite{deraux}, or 
Stover--Toledo \cite{stover-toledo} 
manifolds.
\end{enumerate}
\end{corollary}

 To the authors' knowledge, these examples are new. All previously known examples of ``exotic negatively curved pairs'' are homeomorphic to a locally symmetric space or to a Gromov--Thurston branched cover, and all constructions involved connected sums with exotic spheres or removing neighborhoods of closed totally geodesic submanifolds and regluing them with a twist (see \cite{farrell-jones.negatively,fj-noncpt,fj-complex,fjo-pl1,ardanza,aravinda-farrell-cayley,aravinda-farrell-quaternionic}). 
\subsection{Applications to hyperbolic manifolds}
Hyperbolization procedures have typically been used to obtain aspherical negatively curved manifolds that are ``far'' from being locally symmetric. However, strict hyperbolization actually yields, for the appropriate input, manifolds that are hyperbolic (i.e. constant negatively curved). Surprisingly, this observation has never been used to construct closed hyperbolic manifolds (although strict hyperbolization has been used to construct \emph{infinite volume} hyperbolic manifolds \cite{liu}). Moreover, as a consequence of Theorem \ref{thm:main-stably-tan}, some of these hyperbolic manifolds exhibit characteristic-class behavior not realized by any previously known examples.

To make these ideas more precise, we recall how hyperbolizing pieces are obtained. One starts with a closed, orientable, hyperbolic $n$-manifold $M$ together with a collection $\calY$ of embedded totally geodesic codimension-$1$ submanifolds invariant under an isometric action of the group of symmetries of the $n$-cube, and satisfying certain conditions. 
These conditions ensure that, after cutting $M$ along the submanifolds in $\calY$, one obtains a connected hyperbolic $n$-manifold with corners $X$ which has the symmetries of the $n$-cube. This manifold $X$ is a hyperbolizing piece for the strict hyperbolization procedure. 
We call any manifold $M$ satisfying these conditions a \emph{Charney--Davis manifold}, and refer the reader to Definition \ref{def:CDmanifold} for its precise definition. 

As mentioned before, with a suitable hyperbolizing piece $X$, the $\mathrm{CAT}(-1)$ space $\calH_X(\calG(\msf{N}))$ can be given a smooth Riemannian metric of negative sectional curvature, no matter what the input triangulated manifold $\msf{N}$ is. On the other hand, if the universal cover of a cube complex $\msf{C}$ is isomorphic to the standard cubulation of $\R^n$ 
(in other words, $\msf{C}$ is isometric to a flat manifold), then $\calH_X(\msf{C})$ comes naturally equipped with a hyperbolic metric, no matter what the input hyperbolizing piece $X$ is. 
We call any such $\msf{C}$ a \emph{flat cube complex}; some examples arise as cubulations of flat manifolds with diagonalizable holonomy representation (see Sections \ref{sec:foldableflatmflds} and \ref{sec:hyperbolizingflat}). It is known that some of these flat manifolds have many non-trivial Pontryagin and Stiefel--Whitney classes (see Section \ref{sec:characteristicclassesLS}). Moreover, in this case Ontaneda's normal smooth structure inherits a hyperbolic metric (Lemma \ref{lem:normal coordinates}), which can be combined with Theorem \ref{thm:main-stably-tan} to obtain the next result. 

In the sequel, $w_i(M)\in H^i(M;\F_2)$ and $p_i(M)\in H^{4i}(M;\Z)$ denote the $i$-th Stiefel--Whitney and Pontryagin class of the tangent bundle of the smooth manifold $M$, respectively. 
\begin{theoremalpha}\label{thm:Pontryagin}\hfill
\begin{enumerate} 
\item For all $n\geq 2$ there exist closed hyperbolic $n$-manifolds $M$ such that $w_i(M)\neq 0$ for $1\leq i\leq n-1$. Moreover, if $n=6k+4$, $k\geq 1$, then $p_i(M)\neq 0$ for $i\leq k$.
\item For all $n\geq 8k$, $k\in\{1,2\}$, there exist closed orientable hyperbolic $n$-manifolds $M$ such that $p_i(M)\neq 0$ for $i\leq k$.
\item For all $n\geq 4$ there exist closed orientable hyperbolic $n$-manifolds $M$ such that $w_2(M)\neq 0$. Consequently they do not admit a spin structure.
\item For all $n\geq 5$ there exist closed orientable hyperbolic $n$-manifolds $M$ such that $w_3(M)\neq 0$. Consequently, they do not admit a spin$^{c}$-structure.
\item For all $n\geq 10$ there exist closed orientable hyperbolic $n$-manifolds $M$ such that $w_5(M)\neq 0$. Consequently, they do not admit a spin$^{h}$-structure.
\end{enumerate}
Furthermore, in each of the cases above, there are infinitely many arithmetic commensurability classes and infinitely many non-arithmetic commensurability classes of such manifolds.
\end{theoremalpha}
The hyperbolic manifolds constructed here are strict hyperbolizations of cubulations of flat manifolds. Therefore, they are all
null-bordant. This is because the strict hyperbolization procedure preserves Pontryagin and Stiefel–Whitney numbers, and every flat manifold bounds \cite{hamrick-royster}.

The manifolds in the family (1) above provide the first examples of closed $n$-dimensional hyperbolic manifolds with non-vanishing Stiefel--Whitney classes in all degrees $1\leq i\leq n-1$.  They also provide the first examples of closed hyperbolic manifolds with non-zero Pontryagin classes in arbitrarily high degree. The manifolds in (2) provide the first examples of closed \emph{orientable} hyperbolic manifolds with non-vanishing Pontryagin classes (the existence of non-orientable examples in dimension $4$ is implicit in \cite[Proposition 5.2]{chen-bdry}). Note that these Pontryagin classes are necessarily \textit{torsion} cohomology classes. Indeed, since every hyperbolic manifold is locally conformally flat, its real (and hence rational) Pontryagin classes vanish by the Chern--Weil theory. This also follows from the fact that every closed hyperbolic manifold has a stably parallelizable finite-sheeted cover \cite{sullivan-covers, okun}. The family in (3) recovers the result of Martelli, Slavich and the third-named author \cite{MRS.nonspin}, in which they exhibited, for each $n\geq 4$, a single (arithmetic) commensurability class of closed, orientable, non-spin, hyperbolic $n$-manifolds. The manifolds in (4) recover the main result from \cite{chen.spinc} and answer Questions 1.4 and 1.5 from that paper by producing infinitely many non-arithmetic commensurability classes of non-spin$^c$ examples. The family in (5) answers Question 1.7 from \cite{chen.spinc} by producing closed hyperbolic manifolds without higher spin structures; see also \cite{albanese-milivojevic}.
\begin{remark}
In Corollary \ref{lem:teichnernegcurved}, we use Corollary \ref{cor:Riemannian-hyp} and Ontaneda's work to produce closed orientable negatively curved Riemannian manifolds that are non-spin$^c$ but satisfy $w_3(M)=0$. 
\end{remark}

One extra phenomenon we exhibit in this paper is that closed hyperbolic manifolds with complicated (in terms of their characteristic classes) finite sheeted covers are not sparse. 
Our next result suggests that non-stably parallelizable hyperbolic manifolds are more abundant than one might think. For the next statement, we say that two closed smooth manifolds $M,M'$ are \emph{tangentially related} if there exists a closed smooth manifold $N$ and degree-$1$ stably tangential smooth maps $M\ra N\leftarrow M'$.
\begin{theoremalpha}\label{thmalpha:manymanywithnontrivial}
Let $M$ be any of the manifolds obtained from Theorem \ref{thm:Pontryagin}. Then $M$ has an infinite tower $\cdots \ra M_{i+1}\ra M_i \ra \cdots \ra M_1\ra M_0=M$  of finite covers, such that each cover $M_{i+1}\ra M_i$ is non-trivial, and each $M_i$ is tangentially related to $M$.
\end{theoremalpha}

This result follows from the more general Theorem \ref{thm:manymanywithnontrivial}. In light of this theorem, it is natural to ask:
\begin{question}
    Does there exist an orientable closed hyperbolic manifold $M$ and a tower of finite covers $(M_i)_{i\geq 1}$ of $M$ such that:
    \begin{enumerate}
        \item the injectivity radius of $M_i$ tends to infinity; and
        \item no $M_i$ is stably parallelizable?
    \end{enumerate}
\end{question}
\subsection{Charney--Davis manifolds}
The strict hyperbolization $\calH_X(\msf{C})$ of a smooth manifold equipped with a foldable smooth cubulation $\msf{C}$ embeds in the product $\msf{C}\times X$ with trivial normal bundle (see Section \ref{sec:propertieshyperbolization}). 
Thus, to prove Theorem \ref{thm:main-stably-tan} and Theorem \ref{thm:Pontryagin}, it suffices to find stably parallelizable hyperbolizing pieces $X$ out of (arithmetic and non-arithmetic) Charney--Davis manifolds in distinct commensurability classes. This is achieved by the following theorem.
\begin{theoremalpha}\label{thmalpha:flexibleCD}
    Let $n>0$ and $\K\neq \Q$ be a totally real number field. Then there exist infinitely many (arithmetic and non-arithmetic) commensurability classes of Charney--Davis $n$-manifolds $M$ such that:
    \begin{enumerate}
        \item the adjoint trace field of each $M$ contains $\K$ (and equals $\K$ if $M$ is arithmetic); and
        \item for any finite cover $M'$ of $M$, there exists a further finite cover of $M'$ that is also Charney--Davis.
    \end{enumerate}
\end{theoremalpha}

Since any closed hyperbolic manifold has a stably parallelizable cover \cite{sullivan-covers,okun}, Theorem \ref{thmalpha:flexibleCD} implies the existence of stably parallelizable hyperbolizing pieces $X$. It is worth mentioning that establishing the existence of an (arithmetic) Charney--Davis manifold is one of the main results in \cite{charney-davis}.

The novelty in our Theorem \ref{thmalpha:flexibleCD}
is twofold. First, it gives an ``abundance'' of Charney--Davis manifolds, including non-arithmetic ones. Second, item (2) of the theorem is new even for the original arithmetic Charney--Davis manifolds constructed in \cite{charney-davis}. In fact, Charney and Davis construct manifolds of the form $M=\Hyp^n/\Gamma$ with $\Gamma$ a convenient principal congruence subgroup of a uniform standard arithmetic lattice in $\SO_0(n,1)$. 
One could try to pass to finite index subgroups of $\Gamma$ to obtain more examples, even stably parallelizable ones. Indeed, an argument like this is sketched in \cite[Lemma 6.32]{davis-book}. However, since not every closed hyperbolic manifold has the congruence subgroup property \cite{lubotzky,LLR}, this argument cannot work in general. 

Nevertheless, we can fix the gap in \cite[Lemma 6.32]{davis-book}, prove Theorem \ref{thm:main-stably-tan}, and obtain the arithmetic examples of Theorem \ref{thm:Pontryagin} using the congruence subgroup strategy (see Section \ref{subsec:congruenceapproach}). However, the proof of Theorem \ref{thmalpha:flexibleCD} and hence of Theorem \ref{thm:Pontryagin} in full, requires completely different ideas based on the theory of hyperbolic cubulable groups.

\begin{remark}\label{rmk:smallsystole1}
    By applying a variant of the \emph{inbreeding} construction of Agol \cite{agol.systoles} for $n=4$ and further developed by Belolipetsky--Thompson for arbitrary $n\geq 2$ \cite{belolipetsky-thomson}, we can ensure that the non-arithmetic manifolds $M$ in Theorem \ref{thmalpha:flexibleCD} have arbitrarily small systoles, while still being stably parallelizable. In particular, this allows us to ensure that the non-arithmetic manifolds in Theorem \ref{thm:Pontryagin} have arbitrarily small systoles. See also Remarks \ref{rmk:smallsystole2}, \ref{rmk:smallsystole3}, and \ref{rmk:equalitysystole}. Since the non-arithmetic manifolds that we build are hybrids, they are not quasi-arithmetic, as in Douba's work \cite[Theorem~1]{douba}.
\end{remark}

\subsection{Hyperbolic cubulable groups}\label{intro:proofs}
To prove Theorem \ref{thmalpha:flexibleCD}, we need a robust way to find finite index subgroups satisfying some desired properties. We find such subgroups using the theory of virtually special cubulable hyperbolic groups \cite{haglund-wise.special,AGM.QCERF,agol.haken,wise.QCH,groves-manning.improper}.

There is an intimate connection between strict hyperbolization and cubulable groups. The recent work of Lafont--Ruffoni \cite{lafont-ruffoni} implies the existence of a strict hyperbolization procedure $\calH$, so that $\calH(\msf{C})$ has cubulable (hence virtually compact special \cite{haglund-wise.special}) fundamental group for any compact foldable NPC cube complex $\msf{C}$. (There is a similar result for \emph{relative} strict hyperbolization \cite{LRGM.relative}.) An important ingredient in \cite{lafont-ruffoni} is the fact that the fundamental groups of closed arithmetic manifolds of simplest type are virtually special \cite{BHW}. Another key tool is the work of Groves--Manning \cite{groves-manning.improper}, which relies on \emph{group theoretic Dehn filling} \cite{osin.peripheral,groves-manning.dehnfilling}. Using the machinery of Dehn filling and virtual specialness, we prove the following purely group-theoretic statement, which might be of independent interest, and is used to prove Theorem \ref{thmalpha:flexibleCD}.

\begin{theoremalpha}\label{thm.maingrouptheory}
    Let $\G$ be a hyperbolic cubulable group and consider a finite group $\Phi$ acting on $\G$ by automorphisms. Let $\calQ$ be a finite, $\Phi$-invariant collection of quasiconvex subgroups of $\G$, and let $ \G_0<\G$ be a finite index subgroup. Then there exists a $\Phi$-invariant, finite index normal subgroup $\G'<\G$ such that:
    \begin{enumerate}
        \item $\G'<\G_0$; and
        \item for all $Q_1,Q_2\in \calQ$ we have
        \[\G' \cap Q_1Q_2 \subset (\G'\cap Q_1)(\G'\cap Q_2). \]
    \end{enumerate}
\end{theoremalpha}

This theorem is applied to $\G=\pi_1(M)$ for $M$ a Charney--Davis $n$-manifold, {where} $\Phi=B_n$ is the isometry group of the Euclidean $n$-cube acting on $M$, the collection $\calQ$ is related to the fundamental groups of the submanifolds in the invariant system of submanifolds of $M$ as in Theorem \ref{thm:CD-piece}, and $\Gamma_0$ corresponds to a convenient finite cover of $M$. Condition (2) above (which resembles \cite[Lemma~6.5]{charney-davis}) guarantees that the cover $M'$ associated to $\G'$ is again a Charney--Davis manifold. Condition (1) allows us to choose $M'$ covering any prescribed manifold commensurable with $M$.

The proof of Theorem \ref{thm.maingrouptheory} is based on ideas from the proof of Agol--Groves--Manning of the malnormal special quotient \cite{AGM.MSQT}, and a sketch of the proof is as follows. We iteratively apply Dehn filling to the group $\G$ (for an appropriate choice of peripheral subgroups) to get a finite sequence of hyperbolic cubulable quotients $\G \ra \ov {\G}_1\ra \cdots \ra \ov{\G}_k$. All these quotients carry a natural $\Phi$-action so that consecutive quotient maps are $\Phi$-equivariant. We perform these fillings so that the \emph{height} of the image of collection $\calQ$ (Definition \ref{def:height}) decreases after each iteration, and so that in $\ov{\G}_k$, each group in $\calQ$ gets mapped to a finite group. By residual finiteness, in this case it is easy to find a finite index subgroup $\ov\G'<\ov\G_k$ satisfying condition (2) in Theorem \ref{thm.maingrouptheory} (for the images of $\calQ$). By carefully performing the fillings (see Proposition  \ref{prop.intersectionfillingmulticoset}), we can ensure that the preimage $\G'$ of $\ov\G'$ in $\G$ satisfies conditions (1) and (2).

As an extra consequence of the above result, we answer a question of Belegradek \cite[Section~3]{belegradek} (see also the first Remark on page 343 in \cite{charney-davis}). A key feature of the hyperbolizing piece $X$ induced by a Charney--Davis $n$-manifold is that its poset of faces is isomorphic to the poset of faces of the $n$-cube. These faces are preimages of faces of this cube under a map from $X$ and they are connected in dimensions $0$ and $n$. These faces are not necessarily connected for $0<k<n$, but using Theorem \ref{thm.maingrouptheory} we can find plenty of hyperbolizing pieces, all of whose faces are connected. 

\begin{corollary}\label{cor:connectedfaces}
For any of the Charney--Davis covers obtained from Theorem \ref{thmalpha:flexibleCD}~(2), we can further assume that the induced hyperbolizing piece has connected faces.
\end{corollary}

Combining this corollary with the functoriality of strict hyperbolization with respect to cubical embeddings, we can obtain \emph{connected} totally geodesic subspaces in the hyperbolized complexes out of connected, locally convex subcomplexes in the input cube complex. In particular, hyperbolizing (appropriate cubulations of) flat manifolds of diagonal type with codimension-$1$ connected hypersurfaces yields closed hyperbolic manifolds with connected totally geodesic codimension-1 hypersurfaces. Moreover, if the hypersurface in the flat manifold is non-separating, then the same is true for the hypersurface in the strict hyperbolization.

\subsection{Related work}\label{subsec:previouswork}
Now we describe how our methods and results differ from previous work on characteristic classes of hyperbolic manifolds.

In the closed case, the constructions of hyperbolic manifolds with non-trivial characteristic classes from \cite{MRS.nonspin,chen.spinc} can be described in two steps. In the first step, a certain closed hyperbolic $n$-manifold $M_0$ is constructed, so that some non-trivial characteristic classes are prescribed. In the second step, an infinite sequence of hyperbolic manifolds $M_0\subset M_1\subset M_2\subset \dots$ is constructed, so that each $M_i$ is an embedded totally geodesic codimension-1 submanifold of $M_{i+1}$. This construction is performed so that non-triviality of characteristic classes on $M_i$ implies the same for $M_{i+1}$. The constructions of these totally geodesic embeddings are based on the work of Kolpakov--Reid--Slavich \cite{KReS}, which requires the manifolds to be arithmetic of simplest type. (See also \cite{KRiS} for similar embedding results in the non-arithmetic case, and \cite{chen.spinc} for a useful refinement of \cite{KReS} in this context).

In \cite{MRS.nonspin}, the manifold $M_0$ is $4$-dimensional, orientable, and non-spin, and is built by assembling some right-angled 120-cells. An extension of this construction is given in \cite{MRS.plumbings}, but in all these cases, the resulting manifolds $M_0$ are all commensurable and arithmetic. This also restricts the manifolds $M_1,M_2,\dots$ since the field of definition of each  $M_i$ must contain the field of definition of $M_0$. On the other hand, in \cite{chen.spinc}, $M_0$ is 1-dimensional (i.e. a circle) and arithmetic of simplest type for an arbitrary totally real number field $\K$. Starting from $i\geq 5$, the manifold $M_i$ becomes orientable and non-spin$^c$, but it is still arithmetic.

It is not excluded that some variants of the totally geodesic embedding results from \cite{KRiS} may also give non-spin$^c$ examples that are non-arithmetic. Moreover, one can get from \cite{MRS.nonspin} infinitely many commensurability classes of arithmetic,
non-spin examples in dimensions $n \geq 5$. With some extra work, it may be possible to produce from \cite{MRS.nonspin} infinitely many commensurability classes of non-arithmetic non-spin examples for $n \geq 4$.
However, the methods of \cite{MRS.nonspin,MRS.plumbings} do not yield non-commensurable arithmetic examples in dimension 4. 

On the other hand, our methods allow us to prove stronger versions of these results in a unified and neat way. By applying Theorem \ref{thm:main-stably-tan} and Propositions \ref{thmA:diagonalfoldable} and \ref{thm:flat->hyperbolic}, we translate the problem of obtaining hyperbolic manifolds with non-trivial characteristic classes to the analogous problem for flat manifolds of diagonal type, for which there is a vast literature; see e.g.~\cite{lee-szczarba,im-kim,putrycz-szczepanski,gasior,LPS,gasior-lutowski}. In addition, we have exact control on the commensurability classes of the resulting hyperbolic manifolds, as they coincide with those of the initial Charney--Davis manifold. Furthermore, by Theorem \ref{thmalpha:flexibleCD}~(2) we also have freedom in choosing a Charney--Davis manifold in some fixed commensurability classes. 

We note that flat manifolds have been used to construct cusped, finite volume, hyperbolic manifolds with non-trivial characteristic classes \cite{long-reid.spinning,reid-sell}, using that any $n$-dimensional flat manifold appears as the cusp cross-section of a cusped arithmetic $(n+1)$-hyperbolic manifold \cite{long-reid.flat,mcreynolds}. Using this method and the flat manifolds considered in this paper (see Section \ref{sec:characteristicclassesLS}), one sees that there are infinitely many commensurability classes of \textit{cusped} arithmetic $(n+1)$-manifolds satisfying the conclusions of Theorem \ref{thm:Pontryagin} (but with a shift in the degree of the non-trivial characteristic classes). In this context, see also \cite{riolo-rizzi.cusped}.

\subsection*{Organization of the paper}
In Section \ref{sec:CC+CD}, we discuss some known facts about foldable cube complexes and Charney--Davis's strict hyperbolization. In Section \ref{sec:CDcovers}, assuming Theorem \ref{thm.maingrouptheory} (whose proof is deferred to the end), we prove Theorem \ref{thmalpha:flexibleCD} and Corollary \ref{cor:connectedfaces}, along with Theorem \ref{thm:main-stably-tan}.
We then turn to the study of foldable flat cubulations of flat manifolds. In Section \ref{sec:foldableflatmflds}, we prove Proposition \ref{thmA:diagonalfoldable}, asserting that any flat manifold of diagonal type admits such a cubulation. Next, in Section \ref{sec:hyperbolizingflat}, we prove Proposition \ref{thm:flat->hyperbolic}, which states that hyperbolizing flat cube complexes yields hyperbolic manifolds commensurable with the input Charney--Davis manifold used to construct the hyperbolizing piece.
Section \ref{sec:characteristicclassesLS} is devoted to the study of the Lee--Szczarba manifolds, a particular class of flat manifolds with many non-trivial characteristic classes. In Section \ref{sec:applications}, we apply the machinery developed in the previous sections to construct exotic pairs of negatively curved manifolds, and several hyperbolic manifolds with non-trivial characteristic classes. There we prove Corollary \ref{cor:exotic-smoothings}, Theorem \ref{thm:Pontryagin}, and
Theorem \ref{thmalpha:manymanywithnontrivial}.
Finally, in Section \ref{sec.dehnfilling}, we prove Theorem \ref{thm.maingrouptheory}, a separability theorem about hyperbolic cubulable groups. This is our main technical tool for constructing many Charney--Davis manifolds.
%%%%%%%%%%%%%%%%%%%%%%%%%%%%%%%%%%%%%%%%%%%%%%%%%%%%%%%%%%%%%%%%%%%%%%%%%%%%%%%%%%%%%%%%%%%%%%%%%%%%%%%%%%%%%%%%%%%%%%
\subsection*{Acknowledgments}
We thank Igor Belegradek, Federico Castillo, Jacopo Chen, Sami Douba, Bruno Martelli, Ricardo Menares, Alan Reid, and Leone Slavich for useful comments and conversations. M. Bustamante is supported by ANID Fondecyt Regular grant 1250727. 
E. Reyes is supported by ANID Fondecyt Iniciaci\'on grant 11260637. He also 
acknowledges support from the Simons Laufer Mathematical Sciences Institute in Berkeley, California, during the Spring 2026 semester (National Science Foundation Grant No. DMS-2424139). 
S. Riolo was funded by the European Union – NextGenerationEU under the National Recovery and Resilience Plan (PNRR) – Mission 4 Education and research – Component 2 From research to business – Investment 1.1 Notice Prin 2022 – DD N. 104 del 2/2/2022, with title ``Geometry and topology of manifolds'', proposal code 2022NMPLT8 – CUP J53D23003820001. He acknowledges the hospitality and support of the Pontificia Universidad Cat\'olica de Chile, where this project began.
%%%%%%%%%%%%%%%%%%%%%%%%%%%%%%%%%%%%%%%%%%%%%%%%%%%%%%%%%%%%%%%%%%%%%%%%%
\section{Cube complexes and strict hyperbolization}\label{sec:CC+CD}
In this section we recall some known facts about cube complexes and Charney--Davis's strict hyperbolization procedure. For references, see for instance \cite{charney-davis,belegradek,lafont-ruffoni,davis-book}.

\subsection{Cube complexes and folding}\label{sec:cubecomplexes}

The $n$\emph{-cube} is the space $\square^n=[0,1]^n$ equipped with the Euclidean metric. A \emph{face} $F$ of $\square^n$ is given by considering a (possibly empty) subset $I\subset \{1,\dots,n\}$, a function $\ep: I \ra \{0,1\}$, and defining
\[F=F_{I,\ep}=\{(x_1,\dots,x_n)\in \square^n: x_i=\ep(i)\text{ for }i\in I\}.\]
The \emph{dimension} of the face $F=F_{I,\ep}$ is $n-(\# I)$. Equivalently, we say that the  \emph{codimension} of $F$ is $\# I$. The set of faces of $\square^n$ and the inclusion relation form a poset. This poset is invariant under the group $B_n$ of isometries of $\square^n$. 

A \emph{cube complex} is a  metric polyhedral complex in which all polyhedra are copies of cubes, not necessarily of the same dimension, so that different cubes are glued isometrically along faces (see for instance \cite[Definition~I.7.37]{bridson-haefliger}). A \emph{cell} of the cube complex is a face of a cube used in its construction, and a $k$-\emph{cell} is a cell of dimension $k$. The \emph{dimension} of a cube complex is the supremum of the dimensions of its cells.

The Euclidean metric on each cell induces a length metric on the cube complex. A cube complex is \emph{non-positively curved} NPC if its length metric is locally $\CAT(0)$. By Gromov's criterion \cite[Theorem~II.5.2]{bridson-haefliger}, a cube complex is NPC if and only if the link at each vertex is a flag complex.

A map between cube complexes is \emph{combinatorial} if it is continuous and maps cells to cells. An $n$-dimensional cube complex $\msf{C}$ is \textit{foldable} if it admits a combinatorial map $f:\msf{C}\ra \square^n$ such that its restriction to each cell of $\msf{C}$ injective. Such a map $f$ is called a \emph{folding map}. Note that all the cells of a foldable cube complex are embedded. Any cube complex that covers a foldable cube complex is foldable. Any two folding maps on an $n$-dimensional cube complex differ (up to an isotopy setwise fixing each cell) by composition with an element of $B_n$. 

\begin{example}\label{ex.flats}
Let $\msf{R}_n$ be $\R^n$ equipped with its standard cubical structure induced by the translation action of $\Z^n$. That is, the cube complex structure is $\Z^n$-invariant and $\square^n\subset \R^n$ is a cell. 
The \emph{standard} $n$-torus is the quotient $\bbT^n=\msf{R}_n/\Z^n$. This complex is NPC but is not foldable since cells of dimension greater than 0 are not embedded. On the other hand, the finite cover $\wh\bbT^n=\msf{R}_n/(2\Z)^n$ of $\bbT^n$ is foldable (see Section \ref{subsec:foldingflat}).
\end{example}

\subsection{Strict hyperbolization}\label{sec:strict}
Now we describe in more detail the strict hyperbolization procedure by Charney and Davis \cite{charney-davis}. This procedure takes as input an $n$-dimensional NPC cube complex and returns a locally $\CAT(-1)$ space as output. For simplicity, we restrict to strict hyperbolization of  foldable cube complexes. 
 
The starting point for the hyperbolization is a topological space $X$ equipped with a continuous map $g:X \ra \square^n$. We call the pair $(X,g)$ a \emph{hyperbolizing piece}, although very often we will omit the map $g$ from the notation. A \emph{face} of the hyperbolizing piece is the preimage under $g$ of a face of $\square^n$. Its \emph{dimension} is simply the dimension of the corresponding face of $\square^n$. Note that faces of $X$ are not necessarily connected, see e.g.~\cite[Remark~2.6]{lafont-ruffoni}.

The hyperbolization procedure then takes as input an $n$-dimensional foldable cube complex $\msf{C}$ (with folding map $f:\msf{C}\ra \square^n$) and returns the fibered product
\[\calH_X(\msf{C})=\{(c,x)\in \msf{C}\times X : f(c)=g(x) \}.\]
The space $\calH_X(\msf{C})$ is the \emph{strict hyperbolization} of $\msf{C}$ (also refered to as the \emph{hyperbolized complex}), which is equipped with projection maps $f_X:\calH_X(\msf{C})\ra X$ and $g_\msf{C}:\calH_X(\msf{C})\ra \msf{C}$ that fit into the commutative diagram
\begin{equation}\label{eq.diagramhyp}
    \begin{tikzcd}
\calH_{X}(\msf{C}) \arrow[r,"f_{X}"] \arrow[d,"g_{\msf{C}}"]  & X \arrow[d,"g"]  \\
\msf{C} \arrow[r,"f"] & \square^n.  
\end{tikzcd}
\end{equation}

It is customary to refer to the assignment $\msf{C}\ra \calH_X(\msf{C})$ as the \emph{strict hyperbolization procedure} (with hyperbolizing piece $X=(X,g)$). The key point in Charney--Davis's work \cite{charney-davis} is the construction of a correct hyperbolizing piece $X$, so that if $\msf{C}$ is NPC then $\calH_X(\msf{C})$ is actually negatively curved. 
%%%%%%%%%%%%%%%%%%%%%%%%%%%%%%%
\subsection{Charney--Davis manifolds and the hyperbolizing piece}\label{subsec:hyperbolizingpiece}
In this subsection we describe the hyperbolizing pieces used by Charney--Davis for strict hyperbolization. This piece is a certain $n$-dimensional hyperbolic manifold with corners given by cutting a hyperbolic manifold along totally geodesic submanifolds satisfying some conditions. 

Recall that $B_n$ is the isometry group of the $n$-cube $\square^n$, and it is generated by the coordinate permutations and the reflections $r_i:\square^n \ra \square^n$ so that $r_i(x_1,\dots,x_n)=(\dots,x_{i-1},1-x_{i},x_{i+1},\dots)$ for each $i$. The \emph{standard} orthogonal representation of $B_n$ on $\R^n$ is so that permutations act acoordingly on coordinates of $\R^n$ and the reflections $r_i$ act by sign changes: $r_i(x_1,\dots,x_n)=(\dots,x_{i-1},-x_i,x_{i+1},\dots)$. 

The next theorem summarizes the properties of the manifold used to obtain a hyperbolizing piece. 
\begin{theorem}[{\cite[Theorem~6.1]{charney-davis}}]
\label{thm:CD-piece}
For each $n\ge 1$ there is a closed, connected hyperbolic $n$-manifold $M$, a system $\mathcal Y=\{Y_1,\dots,Y_n\}$ of closed, connected codimension-1 embedded submanifolds of $M$, and an isometric action of $B_n$ on $M$ stabilizing $\mathcal Y$, such that the following hold:
\begin{enumerate}
\item each $Y_i$ is a component of the fixed point set of $r_i$ on $M$;
\item each $Y_i$ is totally geodesic in $M$;
\item the $Y_i$'s intersect orthogonally;
\item $Y_1\cap\cdots\cap Y_n$ is a single point $y$;
\item $B_n$ fixes $y$ and the representation of $B_n$ on $T_yM$ is equivalent to the standard representation; 
\item $M$, and each $Y_i$, is orientable.
\end{enumerate}
\end{theorem}

\begin{definition}\label{def:CDmanifold}
A closed hyperbolic $n$-manifold $M$ satisfying the properties of the previous theorem will be called a \emph{Charney--Davis manifold}, and the collection $\calY=\{Y_1,\dots,Y_n\}$ will be called a \emph{hyperplane system} of $M$. This manifold is implicitely equipped with an isometric action of $B_n$ satisfying the conclusions of the above theorem.
\end{definition}
\begin{example}\label{ex:arithmeticCD}
We proceed to describe the construction of Charney--Davis, which relies on arithmetic data. In \cite[Section~6]{charney-davis}, they consider a totally real quadratic extension of $\Q$, but the same results hold for an arbitrary totally real number field $\K\neq \Q$. Let $\calO_\K$ be the ring of algebraic integers of $\K$, and let $\varepsilon\in \calO_\K$ be a positive number with $\al(\varepsilon)<0$ for any non-trivial Galois conjugation of $\K$. Given $n\geq 1$, consider the quadratic form $q$ on $\calO_\K^{n+1}$ defined by
   \[q(x_0,\dots,x_n)=-\varepsilon x_0^2+x_1^2+\cdots+x_n^2.\]

   We let $\calO(q)$ denote the group of matrices in $\GL_{n+1}(\calO_\K)$ that preserve the form $q$. That is, an invertible matrix $A$ belongs to $\calO(q)$ if and only if $q \circ A=q$. Then $\calO(q)$ is naturally a uniform arithmetic lattice in $\mathrm{O}(n,1)$ \cite[Sections~2.2~and~2.3]{GPS}, hence $\calO(q)/\{\pm \mathrm{I}\}$ acts geometrically on the real hyperbolic space $\Hy^n=\{(x_0,\dots,x_n)\in \R^{n+1}:q(x_0,\dots,x_{n+1})=-1 \text{ and }x_0>0\}$, considered in the hyperboloid model. 
   We embed $B_n$ in $\calO(q)$ as the group generated by the permutations of the last $n$ coordinates and the reflections, so that each $r_i$ maps $(x_0,\dots,x_n)$ to $(\dots,x_{i-1},-x_i,x_{i+1},\dots)$. Note that $B_n$ fixes the vector $\wt y=(1/\sqrt{\varepsilon},0,\dots,0)\in \Hy^n$. 

   For each $i=1,\dots,n$, let $\wt Y_i=\{(x_0,\dots,x_n)\in \Hy^n: x_i=0\}\subset \Hy^n$. Then each $\wt Y_i$ is totally geodesic in $\Hy^n$ and $r_i\in B_n$ acts as a reflection along $\wt Y_i$. Note that $\wt Y_1\cap \cdots \cap \wt Y_n=\{\wt y\}$. The key point in Charney--Davis's construction is to find an appropriate torsion-free finite index normal subgroup $\G<\calO(q)\cap \mathrm{SO}_o(n,1)$, so that $M=\Hy^n/\G$ is a Charney--Davis manifold with hyperplane system given by the projections of the $\wt Y_i$'s in $M$ and the action of $B_n$ on $M$ induced by the embedding of $B_n$ in $\calO(q)$ (see also Section \ref{subsec:congruenceapproach}). 
\end{example}   
Since by varying $\varepsilon$ we obtain infinitely many commensurability classes of lattices $\calO(q)$ \cite{GPS}, this example actually yields the following.
\begin{corollary}\label{cor:infiniteCDmanifolds}
For each $\K$ and quadratic form $q$ as in Example \ref{ex:arithmeticCD}, there exists a Charney--Davis manifold with fundamental group commensurable with $\calO(q)$. In particular, for each $n$ and $\K$ there are infinitely many commensurability classes of arithmetic Charney--Davis $n$-manifolds with field of definition $\K$.
\end{corollary}

If $M$ is a Charney--Davis $n$-manifold with hyperplane system $\calY=\{Y_1,\dots,Y_n\}$, then the hyperbolizing piece $X$ is defined as the metric completion of the space $M \bs (Y_1\cup \cdots \cup Y_n)$, with respect to the length metric on $M \bs (Y_1\cup \cdots \cup Y_n)$ induced by $M$. Equivalently, $X$ is obtained from $M$ after cutting along the submanifolds $Y_1,\dots,Y_n$. Then $X$ is a connected manifold with corners \cite[Section~5]{charney-davis}. In addition, the Pontryagin--Thom construction applied to $M$ with respect to each of the codimension-1 submanifolds $Y_1,\dots,Y_n$ gives a continuous map $\ov{g} : M \ra \bbT^n=\bbT^1\times \cdots \times \bbT^1$. This map induces a continuous map $g:X \ra \square^n$. Recall that a face of $X$ is the (possibly disconnected) preimage under $g$ of a face of $\square^n$. The next lemma summarizes the main properties of the hyperbolizing piece $(X,g)$. 

\begin{lemma}[{\cite[Corollary~6.2]{charney-davis}}]\label{lem:CDpieceX}
Let $X$ be the hyperbolizing piece induced by cutting along the Charney--Davis $n$-manifold $M$. Then $X$ is a connected hyperbolic $n$-manifold with corners, and is equipped with an isometric action of $B_n$ and a continuous, face-preserving and $B_n$-equivariant map $g: X \ra \square^n$ satisfying the following:
\begin{enumerate}
    \item The poset of faces of $X$ is $B_n$–equivariantly isomorphic to that of $\square^n$.
\item Each face of $X$ is totally geodesic.
\item The faces of $X$ intersect orthogonally.
\item Each 0–dimensional face is a single point.
\item The map $g : X\ra \square^n$ and its restriction to each face has degree 1.
\end{enumerate}
\end{lemma}

%%%%%%%%%%%%%%%%%%%%%%%%%%%%%%%%%%%%%%%%%%%%%%%%

\subsection{Properties of strict hyperbolization}\label{sec:propertieshyperbolization}

In this subsection we describe the main properties of the strict hyperbolization procedure. To that end, we fix a hyperbolizing piece $(X,g)$ constructed from the Charney--Davis hyperbolic $n$-manifold $M$. The manifold $M$ and the piece $X$ satisfy the conclusions of Theorem \ref{thm:CD-piece} and Lemma \ref{lem:CDpieceX}. 

If $\msf{C}$ is a foldable cube complex with hyperbolization $\calH_{X}(\msf{C})$ and $\msf{K}\subset \msf{C}$ is a subcomplex, we let $\calH(\msf{K})$ denote $g_{\msf{C}}^{-1}(\msf{K})\subset \calH_{X}(\msf{C})$. Moreover, there is a natural length metric on $\calH_X(\msf{C})$ such that if $\sigma\subset \msf{C}$ is an $n$-dimensional cell, then the map $f_{X}$ from \eqref{eq.diagramhyp} restricts to an isometry $\calH(\sigma) \ra X$, see e.g. \cite[Lemma~2.5~(1)]{lafont-ruffoni}.

The success of Charney--Davis strict hyperbolization in providing interesting negatively curved spaces relies on the following axioms \cite[Section~2]{charney-davis}.

\begin{enumerate}
    \item (Hyperbolicity): if $\msf{C}$ is NPC then $\calH_X(\msf{C})$ is locally $\CAT(-1)$.
    \item (Functoriality): if $\msf{K}\subset \msf{C}$ is a locally convex subcomplex, then $\calH_X(\msf{K})\subset \calH_X(\msf{C})$ is a locally convex subspace. 
\item (Local structure): if $\sigma \subset \msf{C}$ is an $n$-cell, then $\calH_X(\sigma)$ is an $n$-manifold with  corners and the link of $ \calH_X(\sigma)$ in $\calH_X(\msf{C})$ is $\PL$-homeomorphic to the link of $\sigma$ in $\msf{C}$ (for the definition of link, see for instance \cite[Section~2.1.3]{lafont-ruffoni}). 
\item (Homological surjectivity): the map $g_{\msf{C}}:\calH_X(\msf{C})\ra \msf{C}$ induces a surjection on homology.
\end{enumerate}

From properties (3) and (4), we deduce that if $\msf{C}$ homeomorphic to a closed orientable manifold, then $\calH_X(\msf{C})$ is also a closed orientable manifold and $g_\msf{C}:\calH_X(\msf{C})\ra \msf{C}$ is a degree 1 map, which by Poincaré duality induces a monomorphism in cohomology with coefficients in any commutative ring with unit.

To talk about smooth maps we must consider smooth structures in both cube complexes and their hyperbolizations. A \emph{smooth cubulation} of a smooth manifold $N$ is a pair $(\msf{K},\kappa)$ where $\kappa:\msf{C}\to N$ is a homeomorphism from a cube complex $\msf{C}$ that restricts to a smooth embedding on each cube of $\msf{C}$. In this case we will abuse the language and say that $\msf{C}$ is a \textit{smooth cube complex}. Similarly, a folding $f:\msf{C}\ra \square^n$ is \emph{smooth} if it restricts to a smooth embedding on each cell of $\msf{C}$. For a foldable smooth cube complex $\msf{C}$, its hyperbolization $\calH_X(\msf{C})$ has a \emph{normal smooth structure} \cite{ontaneda.cube,ontaneda.smoothing} such that the inclusion $\calH_X(\msf{C})\subset \msf{C}\times X$ is a smooth embedding with trivial normal bundle $\nu(\calH_X(\msf{C}))$. With respect to these smooth structures, the map $g_{\msf{C}}:\calH_X(\msf{C})\ra \msf{C}$ preserves the rational Pontryagin classes.
 
The following lemma gives a sufficient condition for the hyperbolization map $g_{\msf{C}}$ to preserve all stable characteristic classes.
Recall that a smooth manifold $N$ is stably parallelizable if $TN\oplus\varepsilon_N^k$ is trivializable for some $k\geq 0$, where $\varepsilon^k_N$ is the trivial rank-$k$ vector bundle over $N$. 

\begin{lemma}\label{cor:tangential}
Let $M$ be a stably parallelizable Charney--Davis $n$-manifold with corresponding hyperbolizing piece $(X,g)$. 
Then for any foldable, NPC smooth cube complex $\msf{C}$ of dimension $n$, the normal smooth structure on $\calH_X(\msf{C})$ satisfies that the map
\beq
g_{\msf{C}}:\calH_X(\msf{C})\to \msf{C}
\eeq
is stably tangential, that is, it is covered by a map of stable tangent bundles.
\end{lemma}

\begin{proof}
Recall that with the normal smooth structure on $\calH_X(\msf{C})$, the inclusion
$\calH_X(\msf{C})\subset \msf{C}\times X$ is a smooth embedding with trivial normal bundle. 

Considering the commutative diagram \eqref{eq.diagramhyp} and an isomorphism $\nu(\calH_X(\msf{C}))\cong \ep^n_{\calH_X(\msf{C})}$, we obtain
\beq
T\calH_X(\msf{C})\oplus\varepsilon^n_{\calH_X(\msf{C})}\cong
T\calH_X(\msf{C})\oplus\nu(\calH_X(\msf{C}))\cong g_{\msf{C}}^*(T\msf{C})\oplus{f_X^*}(TX)
\cong
g_{\msf{C}}^*(T\msf{C}\oplus\varepsilon^n_{\msf{C}}),
\eeq
where in the last isomorphism we used that $X$ is stably parallelizable because $M$ is, and that for $X$ this implies actual parallelizability (if the classifying map $X\to B\On(n)\to B\On(n+1)$ of the stable tangent bundle is nullhomotopic, then the tangent classifier $X\to B\On(n)$ factors through $S^n$, so it is nullhomotopic).
\end{proof}
%%%%%%%%%%%%%%%%%%%%%%%%%%%%%%%%%%%%%%%%%%%%%%%%%%%%%%%%%%%%%%%%%%%%%%%%%%%%%%%%%%%%%%%%
%%%%%%%%%%%%%%%%%%%%%%%%%%%%%%%%%%%%%%%%%%%%%%%%%%%%%%%%

\section{Flexibility of Charney--Davis manifolds}\label{sec:CDcovers}
In this section we apply Theorem \ref{thm.maingrouptheory} (whose proof is deferred to Section \ref{sec.dehnfilling}) to construct Charney--Davis manifolds, thereby proving Theorem \ref{thmalpha:flexibleCD} and Corollary \ref{cor:connectedfaces}. Theorem \ref{thm:main-stably-tan} will follow directly from Theorem \ref{thmalpha:flexibleCD}.

First we prove Proposition \ref{prop:CDcoversgen}, a general statement about a Charney--Davis-like configuration for a closed manifold with cubulable hyperbolic group. In our hyperbolic setting, this specializes to Corollary \ref{coro:CDcovershyp}. Combining this result with the existence of many arithmetic and non-arithmetic Charney--Davis manifolds (Corollary \ref{cor:infiniteCDmanifolds} and Proposition \ref{prop:nonarithmeticCD} respectively), we deduce Theorem \ref{thmalpha:flexibleCD}. 
For this section we assume familiarity with hyperbolic groups \cite{CDP,ghys-delaharpe,bridson-haefliger}, as well as separability properties of cubulable hyperbolic groups \cite{haglund-wise.special,AGM.QCERF,agol.haken,wise.QCH}.

%%%%%%%%%%%%%%%%%%%

\subsection{Obtaining Charney--Davis covers}\label{subsec:CDcovers} 

Our first step is to prove the following proposition, which is the topological counterpart of Theorem \ref{thm.maingrouptheory}.

\begin{proposition}\label{prop:CDcoversgen}
    Let $M$ be a compact manifold and let $\Phi$ be a finite group acting continuously on $M$. Let $\calY=\{Y_1,\dots, Y_n\}$ be a finite $\Phi$-invariant collection of compact, connected, $\pi_1$-injective submanifolds of $M$. 
    Suppose that:
    \begin{enumerate}
        \item[(a)] The fundamental group $\G=\pi_1(M,y)$ is hyperbolic and cubulable.
        \item[(b)] The intersection $Y_1\cap \cdots \cap Y_n$ contains a point $y$ that is fixed by $\Phi$.
        \item[(c)] The subgroups $H_i=\pi_1(Y_i,y)$ are quasiconvex in $\G$.
        \item[(d)] For the universal cover $\wt M$ of $M$, let $\wt Y_1,\dots, \wt Y_n\subset \wt M$ be lifts of $Y_1,\dots, Y_n$ associated to $H_1,\dots,H_n$ respectively. 
        Then for all non-empty subsets $I\subset \{1,\dots,n\}$, the action of $H_I:=\bigcap_{i\in I}{H_i}$ on $\wt Y_I := \bigcap_{i\in I}{\wt Y_i}$ is cocompact.
    \end{enumerate}
Then for any finite index subgroup $\wh\G<\G$ there exists a finite cover $M' \ra M$, a lifted continuous action of $\Phi$ on $M'$, and a $\Phi$-invariant collection $\calY'=\{Y_1',\dots,Y_n'\}$ with each $Y_i'$ a connected lift of $Y_i$ to $M'$, such that if $q:\wt M \ra M'$ is the associated covering map, then: 
\begin{enumerate}
    \item $\pi_1(M')<\wh \G$; and, 
    \item $q(\wt Y_I)=\bigcap_{i\in I} Y'_i$ for any non-empty subset $I\subset \{1,\dots,n\}$.  
\end{enumerate}
\end{proposition}
\begin{remark}
    A similar statement holds in the setting where $M$ is a compact $CW$, simplicial, or cube complex, $\Phi$ preserves the CW/simplicial/cubical structure, and $Y_1,\dots,Y_n$ are subcomplexes of $M$.
\end{remark}

In the setting of Charney--Davis manifolds, the proposition above implies the following.

\begin{corollary}\label{coro:CDcovershyp}
Let $M$ be a closed connected hyperbolic $n$-manifold with cubulable fundamental group and a hyperplane system $\calY=\{Y_1,\dots,Y_n\}$ satisfying all the conclusions of Theorem \ref{thm:CD-piece}, except that $M$ is not necessarily orientable. If $M_0$ is any finite cover of $M$, then there exists a Charney--Davis manifold $M'$ with hyperplane system $\calY'=\{Y_1',\dots,Y_n'\}$ and such that:
\begin{enumerate}
    \item $M'$ covers $M_0$ and the induced covering $M' \ra M$ is $B_n$-equivariant;
    \item the covering $M' \ra M$ induces a covering $Y'_i \ra Y_i$ for each $i$; and,
    \item for each $I\subset \{1,\dots,n\}$ non-empty, the intersection $Y_I':=\bigcap_{i\in I}{Y_i'}$ is connected.
\end{enumerate} 
\end{corollary}

\begin{proof}
   Let $y$ be intersection point of all the $Y_i$'s and $\G=\pi_1(M,y)$ and $H_i=\pi_1(Y_i,y)$ for $i=1,\dots,n$. Then $\Phi=B_n$ acts on $M$ preserving preserving the totally geodesic codimension-1 system $\calY$. We first verify the assumptions in the statement of Proposition \ref{prop:CDcoversgen}. Indeed, (a) and (b) follow by assumption, and (c) and (d) follow since each $Y_i$ is totally geodesic, so that each $\wt Y_I$ is also totally geodesic in $\wt M=\Hy^n$ for $I\subset \{1,\dots,n\}$ non-empty. 

   Given a finite cover $M_0$ of $M$ with fundamental group $\G_0<\G$, let $\wh \G<\G_0$ be a finite index subgroup consisting only of orientation-preserving elements. We then apply Proposition \ref{prop:CDcoversgen} to find a (connected) finite cover $M'\ra M$, a lifted isometric action of $B_n$ on $M'$, and a $B_n$-invariant system of orientable, totally geodesic codimension-1 submanifolds $\calY'=\{Y_1',\dots,Y_n'\}$ such that
   \begin{enumerate}
       \item[($i$)] $\pi_1(M')<\wh \G<\pi_1(M)$; and,
       \item[($ii$)] if $q:\wt M \ra M'$ is the associated covering map, then
    \[\bigcap_{i\in I} Y'_i=q(\wt Y_I) \quad  \text{ for any non-empty subset }I\subset \{1,\dots,n\}.\] 
   \end{enumerate}
Condition ($i$) implies that $M'$ is an orientable cover of $M_0$, and condition ($ii$) implies that the intersection $\bigcap_{i=1}^n{Y'_i}=q(\bigcap_{i=1}^n \wt Y_i)$ is a single point since $\bigcap_{i=1}^n \wt Y_i$ is a single point (here we use that the submanifolds $\wt Y_i$ intersect transversely). Hence, $M'$ is a Charney--Davis manifold with hyperplane system $\calY'$, and the conclusions (1) and (2) follow from the construction of $M'$. Finally, conclusion (3) follows since each $\wt Y_I$ is connected (indeed convex in $\wt M=\Hy^n$), so that $Y'_I=q(\wt Y_I)$ is also connected.
\end{proof}

We move on to the proof of Proposition \ref{prop:CDcoversgen}, our main tool being Theorem \ref{thm.maingrouptheory}. In the setting of this proposition, the group $\Phi$ naturally acts by automorphisms on $\G=\pi_1(M,y)$. We also fix a finite index subgroup $\wh \G<\G$. Up to intersecting $\wh \G$ with its translates by the action of $\Phi$, we can assume that $\wh \G$ is $\Phi$-invariant. For this convention, we set $\wh H_I:=H_I \cap \wh \G$ for all $I\subset \{1,\dots,n\}$ non-empty. Then the collections $\calH=\{\wh H_1,\dots,\wh H_n\}$ and $\calQ=\{\wh H_I:I \subset \{1,\dots,n\} \text{ non-empty}\}$ are both $\Phi$-invariant and consist of quasiconvex subgroups by (c) and \cite[Theorem~1.2]{hruska}.

We begin with some preliminary lemmas.

\begin{lemma}\label{lem.G_0}
    There exists a normal, $\Phi$-invariant finite index subgroup $\G_0<\wh \G$ satisfying the following. For all $g\in \G_0$ and non-empty subsets $I,J\subset \{1,\dots,n\}$, if $\wt Y_I \cap g \wt Y_J\neq \emptyset$ then $g\in \wh H_I \wh H_J$. 
\end{lemma}

\begin{proof}
    Fix non-empty subsets $I,J \subset \{1,\dots,n\}$ and compact sets $C_I \subset \wt Y_I$ and $C_J \subset \wt Y_J$ such that $\wh H_IC_I=\wt Y_I$ and $\wh H_JC_J=\wt Y_J$. These sets exist by item (d). Let $\calB_{I,J}$ be the set of all $g\in \wh \G \bs \wh H_I\wh H_J$ satisfying $\wt Y_I \cap g \wt Y_J \neq \emptyset$. 

    Clearly $h_Igh_J \in \calB_{I,J}$ whenever $h_I\in \wh H_I,h_J\in \wh H_J$ and $g\in \calB_{I,J}$, so $\calB_{I,J}=\wh H_IF_{I,J}\wh H_J$ for some $F_{I,J}\subset \wh \G \bs \wh H_i\wh H_J$. If $g\in \calB_{I,J}$ and $\wt z \in \wt Y_I \cap g\wt Y_J$ for some $\wt z$, then we can find $h_I\in \wh H_I$ and $h_J\in \wh H_J$ such that $h_I\wt z \in C_I$ and $h_J^{-1}g^{-1}\wt{z}\in C_J$. This implies that 
$h_I\wt z \in C_I \cap (h_Igh_J)C_J$, so $g'=h_Igh_J$ satisfies \begin{equation}\label{eq.CICJ}
    C_I \cap g'C_J \neq \emptyset.
\end{equation} Since there are only finitely many such elements $g'$ satisfying \eqref{eq.CICJ}, we can choose the set $F_{I,J}$ to be finite. 

By item (a), the group $\wh \G$ is hyperbolic and cubulable, hence quasiconvex subgroups are separable by Agol's theorem \cite{agol.haken} and \cite[Theorem~1.3]{haglund-wise.special}. Therefore, $\wh H_I\wh H_J$ is separable in $\wh\G$ by \cite[Theorem~1.1]{minasyan} and there exists a normal finite index subgroup $K_{I,J}<\wh \G$ such that $F_{I,J}\cap \wh H_IK_{I,J}\wh H_{I,J}=\emptyset$. That is, $K_{I,J}$ is disjoint from $\calB_{I,J}$. Therefore, the subgroup
\[\G_0:=\bigcap_{\phi\in \Phi}\bigcap_{I,J}{\phi(K_{I,J})}\]
satisfies all the conclusions of the lemma, where the second intersection runs over all the pairs of non-empty subsets $I,J\subset \{1,\dots,n\}$.     
\end{proof}

\begin{lemma}\label{lem.Rintersectiondoublecoset}
    Let $\G_0<\wh \G$ be given by Lemma \ref{lem.G_0} and suppose $R<\G_0$ satisfies
    \begin{equation}\label{eq.R}
        R\cap \wh H_I\wh H_J \subset (R\cap \wh H_I)(R \cap \wh H_J)
    \end{equation}
    for all non-empty $I,J\subset \{1,\dots,n\}$. Then for all $k\geq 2$ and all non-empty subsets $I_1,\dots, I_k \subset \{1,\dots,n\}$, if $g_1\wt Y_{I_1}\cap \dots \cap g_k \wt Y_{I_k}\neq \emptyset$ for some $g_1,\dots,g_k\in R$ then we can find $a\in R$ and $h_1\in R\cap \wh H_{I_1},\dots, h_k\in R\cap \wh H_{I_k}$ such that $(g_1,\dots,g_k)=(ah_1,\dots,ah_k)$.
\end{lemma}

\begin{proof}
We will prove the assertion by induction on $k$. Suppose that $k=2$ and let $g_1,g_2\in R$ be such that $\wt Y_{I_1}\cap g_1^{-1}g_2 \wt Y_{I_2}\neq \emptyset$. Since $R<\G_0$ by assumption, \eqref{eq.R} and Lemma \ref{lem.G_0} imply that $g_1^{-1}g_2\in (R\cap \wh H_{I_1})(R\cap \wh H_{I_2})$, hence $g_1^{-1}g_2=h_1^{-1}h_2$ for some $h_1\in R\cap \wh H_{I_1}$ and $h_2\in \wh R\cap H_{I_2}$. This gives us $g_1=ah_1$ and $g_2=ah_2$ for $a=g_1h_1^{-1}\in R$, and the assertion follows in this case.

Now suppose that $k\geq 3$ and $g_1\wt Y_{I_1}\cap \dots \cap g_k \wt Y_{I_k}\neq \emptyset$ for some $g_1,\dots,g_k\in R$. Then $g_1\wt Y_{I_1}\cap \dots \cap g_{k-1} \wt Y_{I_{k-1}}\neq \emptyset$ and by induction we have $(g_1,\dots,g_{k-1})=(a'h'_1,\dots,a'h_{k-1}')$ for some $a'\in R,h_1'\in R\cap \wh H_{I_1},\dots,h'_{k-1}\in R\cap \wh H_{I_{k-1}}$. Then $g_1\wt Y_{I_1}\cap \dots \cap g_{k-1} \wt Y_{I_{k-1}}=a'\wt Y_{I'}$ for $I'=I_1\cup \cdots \cup I_{k-1}$, and we have $a'\wt Y_{I'}\cap g_k\wt Y_{I_k}\neq \emptyset$. Since the assertion holds for $k=2$ we have $(a',g_k)=(ah_{I'},ah_k)$ for some $a\in R, h_{I'}\in R\cap \wh H_{I'}=(R\cap \wh H_{I_1})\cap \cdots \cap (R\cap \wh H_{I_{k-1}})$ and $h_k\in R\cap \wh H_{I_k}$. Then we have $(g_1,\dots,g_{k-1},g_k)=(ah_1,\dots,ah_{k-1},ah_k)$ for $h_i=h_{I'}h_i'$ if $1\leq i\leq k-1$, completing the proof by induction.
\end{proof}

\begin{proof}[Proof of Proposition \ref{prop:CDcoversgen}]
We apply Theorem \ref{thm.maingrouptheory} to $\wh \G$, the collection $$\calQ=\{\wh H_I: I\subset \{1,\dots,n\} \text{ non-empty}\}$$ and the finite index subgroup $\G_0<\wh \G$ given by Lemma \ref{lem.G_0}, to find a $\Phi$-invariant, finite index normal subgroup $\G'<\wh \G$ such that $\G'<\G_0$ and for all non-empty subsets $I,J\subset \{1,\dots,n\}$ we have
\begin{equation}\label{eq:G'}\G'\cap \wh H_I \wh H_J \subset (\G' \cap \wh H_I)(\G' \cap \wh H_J).\end{equation}
Let $M'$ be the (connected) finite cover of $M$ associated to $\G'<\G$, which by construction satisfies (1). Let $Y_1',\dots,Y_n'\subset M'$ be the images of $\wt Y_1,\dots,\wt Y_n$ under the universal covering map $q: \wt M \ra M'$. Since each $Y_i$ is $\pi_1$-injective, $\wt Y_i$ is simply connected, so that $Y_1',\dots, Y_n'$ are connected lifts of $Y_1,\dots,Y_n$ respectively.

Given a non-empty subset $I\subset \{1,\dots,n\}$, we claim that $\bigcap_{i\in I} Y'_i=q(\wt Y_I)$. To show this, suppose that $z'\in \bigcap_{i\in I}{Y_i'}$ lifts to $\wt z$ in $\wt M$. Then there exist $g_i\in \G'$ for $i\in I$ such that $\wt z\in \bigcap {g_i\wt Y_i}$. Lemma \ref{lem.Rintersectiondoublecoset}, the fact that $\G'<\G_0$, and \eqref{eq:G'} imply that $(g_i)_{i\in I}=(ah_i)_{i\in I}$ for some $a\in \G'$ and $h_i\in \G'\cap \wh H_i$ for $i\in I$, which gives us 
$\wt z\in  \bigcap {g_i\wt Y_i}=a\left(\bigcap_{i\in I}\wt Y_i\right)$. 
Since $\wt z\in a\wt Y_I$ and $a\in \G'$, we have $z'=q(\wt z)\in q(a\wt Y_I)=q(\wt Y_I)$, as desired. This confirms (2).

To finish the proof, fix a lift $\wt y\in \wt Y_{\{1,\dots,n\}}$ of $y$, and let $y'=q(\wt y)\in M'$. Since $y$ is fixed by $\Phi$, the action of $\Phi$ on $M$ admits a unique lift to an action on $\wt M$ that fixes $\wt y$. By $\Phi$-invariance of $\G'$, this action on $\wt M$ induces an action on $M'$ that fixes $y'$. This action on $M'$ is the desired lift of the action of $\Phi$ on $M$.
\end{proof}
%%%%%%%%%%%%%%%%%%%%%
\subsection{Non-arithmetic Charney--Davis manifolds}\label{subsec:nonarithmetic}
Our next step is to prove the existence of plenty of non-arithmetic Charney--Davis manifolds.
\begin{proposition}\label{prop:nonarithmeticCD}
Let $\K\neq \Q$ be a totally real number field and $n\geq 2$. Then there exist infinitely many commensurability classes of non-arithmetic Charney--Davis $n$-manifolds with adjoint trace field containing $\K$. All these manifolds have cubulable fundamental groups.
\end{proposition}
For the proof of this proposition, we let $M$ be an arithmetic Charney--Davis $n$-manifold of simplest type with hyperplane system $\calY=\{Y_1,\dots,Y_n\}$ with intersection $y$. We let $\G=\pi_1(M,y)$, which acts on $\Hy^n$ so that $B_n$ acts isometrically on $\Hy^n$ normalizing $\G$ and preserving a system $\wt Y_1,\dots, \wt Y_n$ of lifts of $\calY$ intersecting at a single point $\wt y$ that projects to $y$. We let $H_i=\pi_1(Y_i,y)$ for each $i$ and let $\K$ be the field of definition of $\G$.

We also fix a closed hyperbolic $n$-manifold $N$ so that:
\begin{itemize}
    \item $N$ is arithmetic of simplest type with field of definition $\K$, and non-commensurable with $M$.
    \item $N$ contains an immersed totally geodesic hypersurface $V$ that is commensurable to $Y_1$.
\end{itemize}

There are infinitely many commensurability classes of manifolds $N$ satisfying the two conditions above \cite{GPS}.

The idea is to cut an appropriate finite cover of $M$ along certain totally geodesic codimension-1 submanifolds in a $B_n$-equivariant way. Then we glue some hyperbolic manifolds obtained from $N$ along these boundary components, so that the resulting manifold is non-arithmetic. By performing this equivariantly, we can ensure that the constructed manifold is Charney--Davis.

First, we find the appropriate cover of $M$ that we can cut along.

\begin{lemma}\label{lem:firstlift}
    There is a finite cover $M'$ of $M$, so that:
    \begin{itemize}
        \item $M'$ is a Charney--Davis manifold with hyperplane system $\calY'={Y'_1,\dots,Y_n'}$.
        \item The covering $M'\ra M$ is $B_n$-equivariant and restricts to covering maps $Y_i'\ra Y_i$.
    \end{itemize}
    In addition, there exists a $B_n$-equivariant system $\calZ'=\{Z_1',\dots,Z_m'\}$ of totally geodesic codimension-1 embedded submanifolds of $M'$ such that:
    \begin{itemize}
      \item $Z_i'\cap (Y_1'\cup \cdots \cup Y_n')=Z_i'\cap Z_j'=\emptyset$ for all $i\neq j$.
      \item $\calZ'=\{bZ_1': b\in B_n\}$ and the action of $B_n$ on $\calZ'$ is free.
      \item $Z_1'$ is an orientable finite cover of $V$.
      \end{itemize}
\end{lemma}

\begin{proof}
By the arithmeticity of $M$, $\G$ has dense commensurator $G$ in $\Isom^+(\Hy^n)$, so we can find a $G$-translate $\wt Z_1$ of $\wt Y_1$, disjoint from $\bigcup \wt\calY$, and so that $b\wt Z_1\cap \wt Z_1=\emptyset$ for any non-trivial element $b\in B_n$. Let $\wt \calZ=\{\wt Z_1,\dots, \wt Z_m\}$ be the set of $B_n$-translates of $\wt Z_1$ and $L_j$ be the stabilizer of $\wt Z_j$ in $\G$. Note that $L_j$ is quasiconvex and acts cocompactly on $\wt Z_j$ by assumption.

As in Lemma \ref{lem.G_0}, we can use the virtual specialness of $\G$ \cite{BHW} and the separability of the double cosets $L_iL_j$ and $H_iL_j$ to find a finite cover $\what M\ra M$, so that if $\what Y_i$ and $\what Z_j$ are the projections of $\wt Y_i$ and $\wt Z_j$ in $\what M$ respectively, then:
\begin{itemize}
    \item each $\what Z_j$ is embedded and disjoint from each $\what Y_i$;
    \item $\what Z_i$ and $\what Z_j$ are disjoint for $i\neq j$; and,
    \item $\what Z_1$ is an orientable finite cover of $V$.
\end{itemize}

Note that these properties remain valid in any finite cover of $\what M$. Then we obtain the desired Charney--Davis cover $M'$ by applying Corollary \ref{coro:CDcovershyp} so that there is a finite cover $q:M'\to\what M$. The required family $\mathcal{Z}'$ of totally geodesic codimension-1 embedded submanifolds of $M'$ is obtained by choosing
a connected component $Z_1'\subset q^{-1}(\widehat Z_1)$ and defining $Z_j':=b_j Z_1'$, where $b_j\in B_n$ is the unique element such that $b_j\widehat Z_1=\widehat Z_j$.
\end{proof}

Similarly, we promote $N$ to a manifold we can cut along. The next lemma follows immediately from a theorem of Millson \cite{millson} (see also \cite[Theorem~1.2]{BHW}).

\begin{lemma}\label{lem:promoteN}
    There exists a finite cover $N'$ of $N$ so that $V$ lifts to an embedded, non-separating submanifold $V'$ isometric to $Z_1'$. \qed
\end{lemma}

\begin{proof}[Proof of Proposition \ref{prop:nonarithmeticCD}]
Consider the manifolds $M$ and $N$ as above, and let $M'$ and $N'$ be given by Lemmas \ref{lem:firstlift} and  \ref{lem:promoteN} respectively. Let $P$ be the connected component containing $K:=Y'_1\cup \cdots \cup Y'_n$ of the manifold obtained by cutting $M'$ along $Z'_1,\dots,Z_m'$. Note that the action of $B_n$ on $M'$ induces an isometric action of $B_n$ on $P$.

We also let $Q$ be the manifold obtained by cutting $N'$ along $V'$. Then $Q$ has totally geodesic boundary with two components $U$ and $W$, so that there are isometries $a: Z'_1 \ra U$ and $b: Z_1'\ra W$. 

We will glue several copies of $P$ and $Q$ in a $B_n$-equivariant way to obtain a closed hyperbolic manifold with the desired properties.

There are two cases to consider:

\textbf{Case 1:} $Z_1'$ is separating in $M'$. 

Since the action of $B_n$ on $M'$ is isometric and free on $\{Z_1',\dots,Z_m'\}$, in this case we have that the component containing $K$ of $M'$ cut along $Z_1'$ also contains $Z_j'$ for $j\neq 1$. This implies that $P$ has $m$ boundary components, all isometric to $Z_1'$, and so that the action of $B_n$ is simply transitive on this set of boundary components. By abuse of notation, we also denote these boundary components by $Z_1',\dots,Z_m'$. Let $\what P$ be another copy of $P$, with boundary components $\what{ Z}_1',\dots, \what{Z}_m'$, $B_n$-equivariantly  identified with $Z_1',\dots,Z_m'$. 

Let $Q_1,\dots,Q_m$ be isometric copies of $Q$, each with corresponding boundary components $U_j,W_j$ so that there are isometries $a_j:Z_j'\ra U_j$ and $b_j:\what Z_j'\ra W_j$, all consistent with $a$ and $b$ and the action of $B_n$ on $P$ and $\what P$. We glue $P,\what P$, and $Q_1,\dots,Q_m$ using these isometries, obtaining a closed hyperbolic manifold $R$.

By construction, there is a natural isometric action of $B_n$ on $R$, and there is a $B_n$-equivariant isometric embedding of $P$ into $R$. In particular, $R$ contains isometric copies of $Y_1',\dots,Y_n'$, so that this set is preserved by the $B_n$-action. Note that $R$ is non-arithmetic, since $\pi_1(R)$ contains subgroups isomorphic to $\pi_1(P)<\pi_1(M')$ and $\pi_1(Q)<\pi_1(N')$ and $M',N'$ are not commensurable \cite[Lemma~6.5.17]{wittemorris}. 

To show that $\pi_1(R)$ is cubulable, note that it splits as the fundamental group of a graph of groups in which each vertex group is isomorphic to either $\pi_1(P)$ or $\pi_1(Q)$. All these groups are cubulable (hence virtually special) by \cite[Proposition~7.2]{haglund-wise.special} and quasiconvex in $\pi_1(R)$. Then $\pi_1(R)$ is virtually special, hence cubulable by Wise's quasiconvex hierarchy theorem \cite[Theorem~B]{wise.QCH}. 

At this point, we do not know if $R$ is orientable, so we appeal to Corollary \ref{coro:CDcovershyp} to find  to a finite cover of $R$ that is a Charney--Davis manifold. By construction, the adjoint trace field of $R$ contains $\K$, which gives us all the desired properties for $R$ in this case.

\begin{figure}[!h]
    \centering
\tikzset{every picture/.style={line width=0.75pt}} %set default line width to 0.75pt        

\begin{tikzpicture}[x=0.75pt,y=0.75pt,yscale=-0.75,xscale=0.75]
%uncomment if require: \path (0,300); %set diagram left start at 0, and has height of 300

%Shape: Ellipse [id:dp16723521394573648] 
\draw   (40.67,202.73) .. controls (40.67,198.34) and (50.65,194.78) .. (62.96,194.78) .. controls (75.27,194.78) and (85.25,198.34) .. (85.25,202.73) .. controls (85.25,207.12) and (75.27,210.68) .. (62.96,210.68) .. controls (50.65,210.68) and (40.67,207.12) .. (40.67,202.73) -- cycle ;
%Shape: Ellipse [id:dp993802811359551] 
\draw   (154.92,202.73) .. controls (154.92,198.34) and (164.9,194.78) .. (177.22,194.78) .. controls (189.53,194.78) and (199.51,198.34) .. (199.51,202.73) .. controls (199.51,207.12) and (189.53,210.68) .. (177.22,210.68) .. controls (164.9,210.68) and (154.92,207.12) .. (154.92,202.73) -- cycle ;
%Curve Lines [id:da16996832986260046] 
\draw    (40.67,202.73) .. controls (74.42,326.02) and (166.98,326.39) .. (199.51,202.73) ;
%Curve Lines [id:da4371641144244176] 
\draw    (85.25,202.73) .. controls (85.48,235.62) and (108.41,237.01) .. (108.03,202.41) ;
%Curve Lines [id:da2227197522676948] 
\draw    (132.15,203.05) .. controls (132.37,235.94) and (155.31,237.34) .. (154.92,202.73) ;
%Flowchart: Connector [id:dp9800272098514179] 
\draw  [color={rgb, 255:red, 0; green, 0; blue, 0 }  ,draw opacity=1 ][fill={rgb, 255:red, 0; green, 0; blue, 0 }  ,fill opacity=1 ] (112.13,202.99) .. controls (112.13,203.23) and (112.42,203.42) .. (112.79,203.43) .. controls (113.15,203.43) and (113.45,203.24) .. (113.45,203.01) .. controls (113.45,202.77) and (113.16,202.57) .. (112.79,202.57) .. controls (112.43,202.57) and (112.13,202.76) .. (112.13,202.99) -- cycle ;
%Flowchart: Connector [id:dp4986791374982652] 
\draw  [color={rgb, 255:red, 0; green, 0; blue, 0 }  ,draw opacity=1 ][fill={rgb, 255:red, 0; green, 0; blue, 0 }  ,fill opacity=1 ] (124.74,202.99) .. controls (124.74,203.23) and (125.03,203.42) .. (125.39,203.43) .. controls (125.76,203.43) and (126.06,203.24) .. (126.06,203.01) .. controls (126.06,202.77) and (125.77,202.57) .. (125.4,202.57) .. controls (125.04,202.57) and (124.74,202.76) .. (124.74,202.99) -- cycle ;
%Flowchart: Connector [id:dp3721684693602464] 
\draw  [color={rgb, 255:red, 0; green, 0; blue, 0 }  ,draw opacity=1 ][fill={rgb, 255:red, 0; green, 0; blue, 0 }  ,fill opacity=1 ] (119.1,202.99) .. controls (119.1,203.23) and (119.39,203.42) .. (119.75,203.43) .. controls (120.12,203.43) and (120.42,203.24) .. (120.42,203.01) .. controls (120.42,202.77) and (120.13,202.57) .. (119.76,202.57) .. controls (119.4,202.57) and (119.1,202.76) .. (119.1,202.99) -- cycle ;
%Shape: Ellipse [id:dp44650128399322775] 
\draw   (198.83,91.41) .. controls (198.89,95.8) and (188.97,99.51) .. (176.65,99.69) .. controls (164.34,99.88) and (154.31,96.46) .. (154.24,92.07) .. controls (154.18,87.68) and (164.11,83.97) .. (176.42,83.79) .. controls (188.73,83.61) and (198.76,87.02) .. (198.83,91.41) -- cycle ;
%Shape: Ellipse [id:dp027956260020031154] 
\draw   (84.58,93.11) .. controls (84.65,97.5) and (74.72,101.21) .. (62.41,101.39) .. controls (50.1,101.57) and (40.07,98.16) .. (40,93.77) .. controls (39.94,89.38) and (49.86,85.67) .. (62.17,85.49) .. controls (74.49,85.3) and (84.52,88.71) .. (84.58,93.11) -- cycle ;
%Curve Lines [id:da3395193019482784] 
\draw    (198.83,91.41) .. controls (163.25,-31.36) and (70.7,-30.36) .. (40,93.77) ;
%Curve Lines [id:da4233845865820205] 
\draw    (154.24,92.07) .. controls (153.54,59.19) and (130.58,58.14) .. (131.48,92.73) ;
%Curve Lines [id:da8591572147442541] 
\draw    (107.35,92.45) .. controls (106.64,59.57) and (83.68,58.51) .. (84.58,93.11) ;
%Flowchart: Connector [id:dp5862513360502731] 
\draw  [color={rgb, 255:red, 0; green, 0; blue, 0 }  ,draw opacity=1 ][fill={rgb, 255:red, 0; green, 0; blue, 0 }  ,fill opacity=1 ] (127.37,92.21) .. controls (127.36,91.97) and (127.07,91.78) .. (126.7,91.78) .. controls (126.34,91.79) and (126.05,91.98) .. (126.05,92.22) .. controls (126.05,92.45) and (126.35,92.64) .. (126.71,92.64) .. controls (127.07,92.64) and (127.37,92.45) .. (127.37,92.21) -- cycle ;
%Flowchart: Connector [id:dp6863997160188826] 
\draw  [color={rgb, 255:red, 0; green, 0; blue, 0 }  ,draw opacity=1 ][fill={rgb, 255:red, 0; green, 0; blue, 0 }  ,fill opacity=1 ] (114.76,92.4) .. controls (114.76,92.16) and (114.46,91.97) .. (114.1,91.97) .. controls (113.73,91.97) and (113.44,92.17) .. (113.44,92.4) .. controls (113.44,92.64) and (113.74,92.83) .. (114.1,92.83) .. controls (114.47,92.83) and (114.76,92.63) .. (114.76,92.4) -- cycle ;
%Flowchart: Connector [id:dp9444093133920379] 
\draw  [color={rgb, 255:red, 0; green, 0; blue, 0 }  ,draw opacity=1 ][fill={rgb, 255:red, 0; green, 0; blue, 0 }  ,fill opacity=1 ] (120.4,92.31) .. controls (120.4,92.08) and (120.1,91.89) .. (119.74,91.89) .. controls (119.37,91.89) and (119.08,92.08) .. (119.08,92.32) .. controls (119.08,92.56) and (119.38,92.75) .. (119.74,92.74) .. controls (120.11,92.74) and (120.4,92.55) .. (120.4,92.31) -- cycle ;
%Shape: Ellipse [id:dp3293317842804544] 
\draw   (40,133.88) .. controls (40,131.42) and (49.98,129.43) .. (62.29,129.43) .. controls (74.61,129.43) and (84.59,131.42) .. (84.59,133.88) .. controls (84.59,136.34) and (74.61,138.33) .. (62.29,138.33) .. controls (49.98,138.33) and (40,136.34) .. (40,133.88) -- cycle ;
%Straight Lines [id:da09320709543239336] 
\draw    (40,134.43) -- (40,164.43) ;
%Straight Lines [id:da9127248749133238] 
\draw    (84.59,134.43) -- (84.59,164.43) ;
%Shape: Ellipse [id:dp4861047871404046] 
\draw   (40,165.26) .. controls (40,162.64) and (49.98,160.52) .. (62.29,160.52) .. controls (74.61,160.52) and (84.59,162.64) .. (84.59,165.26) .. controls (84.59,167.88) and (74.61,170) .. (62.29,170) .. controls (49.98,170) and (40,167.88) .. (40,165.26) -- cycle ;
%Shape: Ellipse [id:dp33455589317697343] 
\draw   (156.95,133.6) .. controls (156.95,131.14) and (166.93,129.14) .. (179.25,129.14) .. controls (191.56,129.14) and (201.54,131.14) .. (201.54,133.6) .. controls (201.54,136.05) and (191.56,138.05) .. (179.25,138.05) .. controls (166.93,138.05) and (156.95,136.05) .. (156.95,133.6) -- cycle ;
%Straight Lines [id:da08401800885488231] 
\draw    (156.95,134.14) -- (156.95,164.14) ;
%Straight Lines [id:da49519925733509096] 
\draw    (201.54,134.14) -- (201.54,164.14) ;
%Shape: Ellipse [id:dp24418161654958803] 
\draw   (156.95,164.98) .. controls (156.95,162.36) and (166.93,160.24) .. (179.25,160.24) .. controls (191.56,160.24) and (201.54,162.36) .. (201.54,164.98) .. controls (201.54,167.59) and (191.56,169.71) .. (179.25,169.71) .. controls (166.93,169.71) and (156.95,167.59) .. (156.95,164.98) -- cycle ;
%Shape: Ellipse [id:dp39777231182255157] 
\draw   (412,137.4) .. controls (412,133.01) and (421.98,129.44) .. (434.29,129.44) .. controls (446.61,129.44) and (456.59,133.01) .. (456.59,137.4) .. controls (456.59,141.79) and (446.61,145.35) .. (434.29,145.35) .. controls (421.98,145.35) and (412,141.79) .. (412,137.4) -- cycle ;
%Shape: Ellipse [id:dp4605952758599109] 
\draw   (526.26,137.4) .. controls (526.26,133.01) and (536.24,129.44) .. (548.55,129.44) .. controls (560.86,129.44) and (570.84,133.01) .. (570.84,137.4) .. controls (570.84,141.79) and (560.86,145.35) .. (548.55,145.35) .. controls (536.24,145.35) and (526.26,141.79) .. (526.26,137.4) -- cycle ;
%Curve Lines [id:da9965688952477736] 
\draw    (347.41,137.95) .. controls (381.16,261.24) and (604.05,261.61) .. (636.59,137.95) ;
%Curve Lines [id:da36980511195623267] 
\draw    (456.59,137.4) .. controls (456.81,170.28) and (479.75,171.68) .. (479.36,137.07) ;
%Curve Lines [id:da7982288311846758] 
\draw    (503.49,137.72) .. controls (503.71,170.61) and (526.64,172) .. (526.26,137.4) ;
%Flowchart: Connector [id:dp3139417025590566] 
\draw  [color={rgb, 255:red, 0; green, 0; blue, 0 }  ,draw opacity=1 ][fill={rgb, 255:red, 0; green, 0; blue, 0 }  ,fill opacity=1 ] (483.47,137.66) .. controls (483.46,137.9) and (483.76,138.09) .. (484.12,138.1) .. controls (484.49,138.1) and (484.78,137.91) .. (484.78,137.67) .. controls (484.79,137.44) and (484.49,137.24) .. (484.13,137.24) .. controls (483.76,137.23) and (483.47,137.42) .. (483.47,137.66) -- cycle ;
%Flowchart: Connector [id:dp2579485630758446] 
\draw  [color={rgb, 255:red, 0; green, 0; blue, 0 }  ,draw opacity=1 ][fill={rgb, 255:red, 0; green, 0; blue, 0 }  ,fill opacity=1 ] (496.07,137.66) .. controls (496.07,137.9) and (496.36,138.09) .. (496.73,138.1) .. controls (497.09,138.1) and (497.39,137.91) .. (497.39,137.67) .. controls (497.39,137.44) and (497.1,137.24) .. (496.73,137.24) .. controls (496.37,137.23) and (496.07,137.42) .. (496.07,137.66) -- cycle ;
%Flowchart: Connector [id:dp05801738989208238] 
\draw  [color={rgb, 255:red, 0; green, 0; blue, 0 }  ,draw opacity=1 ][fill={rgb, 255:red, 0; green, 0; blue, 0 }  ,fill opacity=1 ] (490.43,137.66) .. controls (490.43,137.9) and (490.72,138.09) .. (491.09,138.1) .. controls (491.45,138.1) and (491.75,137.91) .. (491.75,137.67) .. controls (491.75,137.44) and (491.46,137.24) .. (491.09,137.24) .. controls (490.73,137.23) and (490.43,137.42) .. (490.43,137.66) -- cycle ;
%Shape: Ellipse [id:dp23622931249753898] 
\draw   (347.41,137.95) .. controls (347.41,133.56) and (357.39,130) .. (369.71,130) .. controls (382.02,130) and (392,133.56) .. (392,137.95) .. controls (392,142.35) and (382.02,145.91) .. (369.71,145.91) .. controls (357.39,145.91) and (347.41,142.35) .. (347.41,137.95) -- cycle ;
%Shape: Ellipse [id:dp7674504464533771] 
\draw   (592,137.95) .. controls (592,133.56) and (601.98,130) .. (614.29,130) .. controls (626.61,130) and (636.59,133.56) .. (636.59,137.95) .. controls (636.59,142.35) and (626.61,145.91) .. (614.29,145.91) .. controls (601.98,145.91) and (592,142.35) .. (592,137.95) -- cycle ;
%Curve Lines [id:da9037097854882638] 
\draw    (392,137.95) .. controls (392.29,160.19) and (412,160.76) .. (412,137.4) ;
%Curve Lines [id:da671150495288077] 
\draw    (570.84,137.4) .. controls (571.13,159.63) and (592,161.32) .. (592,137.95) ;
%Shape: Ellipse [id:dp2875576120755481] 
\draw   (350,88.64) .. controls (350,84.05) and (359.98,80.33) .. (372.29,80.33) .. controls (384.61,80.33) and (394.59,84.05) .. (394.59,88.64) .. controls (394.59,93.23) and (384.61,96.95) .. (372.29,96.95) .. controls (359.98,96.95) and (350,93.23) .. (350,88.64) -- cycle ;
%Shape: Ellipse [id:dp6336082680282631] 
\draw   (414.59,88.64) .. controls (414.59,84.05) and (424.57,80.33) .. (436.88,80.33) .. controls (449.19,80.33) and (459.18,84.05) .. (459.18,88.64) .. controls (459.18,93.23) and (449.19,96.95) .. (436.88,96.95) .. controls (424.57,96.95) and (414.59,93.23) .. (414.59,88.64) -- cycle ;
%Curve Lines [id:da479124516692259] 
\draw    (350,88.64) .. controls (349.76,36.75) and (459.01,37.5) .. (459.18,88.64) ;
%Curve Lines [id:da18777592662520082] 
\draw    (394.59,88.64) .. controls (394.51,77.5) and (414.51,77) .. (414.59,88.64) ;
%Shape: Ellipse [id:dp9191301160353123] 
\draw   (528.25,88.64) .. controls (528.25,84.05) and (538.47,80.33) .. (551.07,80.33) .. controls (563.67,80.33) and (573.89,84.05) .. (573.89,88.64) .. controls (573.89,93.23) and (563.67,96.95) .. (551.07,96.95) .. controls (538.47,96.95) and (528.25,93.23) .. (528.25,88.64) -- cycle ;
%Shape: Ellipse [id:dp6204948640557438] 
\draw   (594.36,88.64) .. controls (594.36,84.05) and (604.58,80.33) .. (617.18,80.33) .. controls (629.78,80.33) and (640,84.05) .. (640,88.64) .. controls (640,93.23) and (629.78,96.95) .. (617.18,96.95) .. controls (604.58,96.95) and (594.36,93.23) .. (594.36,88.64) -- cycle ;
%Curve Lines [id:da45946375140426965] 
\draw    (528.25,88.64) .. controls (528,36.75) and (639.82,37.5) .. (640,88.64) ;
%Curve Lines [id:da038752146808627486] 
\draw    (573.89,88.64) .. controls (573.81,77.5) and (594.28,77) .. (594.36,88.64) ;
%Shape: Ellipse [id:dp3968483150042792] 
\draw   (110.75,264.94) .. controls (110.75,261.18) and (115.26,258.12) .. (120.83,258.12) .. controls (126.4,258.12) and (130.92,261.18) .. (130.92,264.94) .. controls (130.92,268.7) and (126.4,271.75) .. (120.83,271.75) .. controls (115.26,271.75) and (110.75,268.7) .. (110.75,264.94) -- cycle ;
%Curve Lines [id:da6276162721768144] 
\draw [color={rgb, 255:red, 208; green, 2; blue, 27 }  ,draw opacity=1 ]   (117.67,295.12) .. controls (113.92,291.87) and (112.17,275.62) .. (119.17,271.37) ;
%Curve Lines [id:da4036846883387758] 
\draw [color={rgb, 255:red, 208; green, 2; blue, 27 }  ,draw opacity=1 ] [dash pattern={on 0.84pt off 2.51pt}]  (119.33,295.5) .. controls (124.92,291.87) and (124.42,275.37) .. (120.83,271.75) ;
%Shape: Ellipse [id:dp9810402976989715] 
\draw  [color={rgb, 255:red, 208; green, 2; blue, 27 }  ,draw opacity=1 ] (98.67,265.9) .. controls (98.67,255.96) and (108.84,247.9) .. (121.38,247.9) .. controls (133.93,247.9) and (144.1,255.96) .. (144.1,265.9) .. controls (144.1,275.85) and (133.93,283.9) .. (121.38,283.9) .. controls (108.84,283.9) and (98.67,275.85) .. (98.67,265.9) -- cycle ;
%Shape: Ellipse [id:dp5262458347698802] 
\draw   (108.64,26.8) .. controls (108.64,23.04) and (113.16,19.99) .. (118.73,19.99) .. controls (124.3,19.99) and (128.81,23.04) .. (128.81,26.8) .. controls (128.81,30.57) and (124.3,33.62) .. (118.73,33.62) .. controls (113.16,33.62) and (108.64,30.57) .. (108.64,26.8) -- cycle ;
%Shape: Ellipse [id:dp379727461074268] 
\draw   (483.25,200.1) .. controls (483.25,196.34) and (487.76,193.29) .. (493.33,193.29) .. controls (498.9,193.29) and (503.42,196.34) .. (503.42,200.1) .. controls (503.42,203.87) and (498.9,206.92) .. (493.33,206.92) .. controls (487.76,206.92) and (483.25,203.87) .. (483.25,200.1) -- cycle ;
%Curve Lines [id:da6708860162057795] 
\draw [color={rgb, 255:red, 208; green, 2; blue, 27 }  ,draw opacity=1 ]   (490.17,230.29) .. controls (486.42,227.04) and (484.67,210.79) .. (491.67,206.54) ;
%Curve Lines [id:da12311536771224174] 
\draw [color={rgb, 255:red, 208; green, 2; blue, 27 }  ,draw opacity=1 ] [dash pattern={on 0.84pt off 2.51pt}]  (491.83,230.67) .. controls (497.42,227.04) and (496.92,210.54) .. (493.33,206.92) ;
%Shape: Ellipse [id:dp1405477397363033] 
\draw  [color={rgb, 255:red, 208; green, 2; blue, 27 }  ,draw opacity=1 ] (471.17,201.07) .. controls (471.17,191.13) and (481.34,183.07) .. (493.88,183.07) .. controls (506.43,183.07) and (516.6,191.13) .. (516.6,201.07) .. controls (516.6,211.01) and (506.43,219.07) .. (493.88,219.07) .. controls (481.34,219.07) and (471.17,211.01) .. (471.17,201.07) -- cycle ;
%Straight Lines [id:da06042861336212235] 
\draw    (178.7,108.22) -- (178.7,120.67) ;
\draw   (174.88,112.25) -- (178.67,108.14) -- (182.47,112.24) ;
\draw   (182.36,117.48) -- (178.6,121.61) -- (174.77,117.55) ;
%Straight Lines [id:da49972774129716735] 
\draw    (62.08,108.58) -- (62.08,121.02) ;
\draw   (58.26,113.01) -- (62.05,108.89) -- (65.85,113) ;
\draw   (65.73,117.83) -- (61.97,121.97) -- (58.14,117.9) ;
%Straight Lines [id:da24570059863555183] 
\draw    (179.36,174.22) -- (179.36,186.67) ;
\draw   (175.55,178.25) -- (179.34,174.14) -- (183.14,178.24) ;
\draw   (183.02,183.48) -- (179.26,187.61) -- (175.43,183.55) ;
%Straight Lines [id:da18617932014222283] 
\draw    (62.74,174.58) -- (62.74,187.02) ;
\draw   (58.93,179.01) -- (62.72,174.89) -- (66.52,179) ;
\draw   (66.4,183.83) -- (62.64,187.97) -- (58.81,183.9) ;
%Straight Lines [id:da6839341700285582] 
\draw    (369.93,108.61) -- (369.93,121.05) ;
\draw   (366.12,112.64) -- (369.91,108.52) -- (373.71,112.63) ;
\draw   (373.59,117.87) -- (369.83,122) -- (366,117.93) ;
%Straight Lines [id:da2560923457102279] 
\draw    (435.63,108.09) -- (435.63,120.53) ;
\draw   (431.82,112.12) -- (435.61,108) -- (439.41,112.11) ;
\draw   (439.29,117.35) -- (435.53,121.48) -- (431.7,117.41) ;
%Straight Lines [id:da620250815401843] 
\draw    (549.03,107.94) -- (549.03,120.38) ;
\draw   (545.22,111.97) -- (549.01,107.85) -- (552.81,111.96) ;
\draw   (552.69,117.2) -- (548.93,121.33) -- (545.1,117.26) ;
%Straight Lines [id:da621296866930635] 
\draw    (614.83,107.54) -- (614.83,119.98) ;
\draw   (611.02,111.57) -- (614.81,107.45) -- (618.61,111.56) ;
\draw   (618.49,116.8) -- (614.73,120.93) -- (610.9,116.86) ;

% Text Node
\draw (146.08,252.34) node [anchor=north west][inner sep=0.75pt]  [font=\scriptsize,color={rgb, 255:red, 208; green, 2; blue, 27 }  ,opacity=1 ,rotate=-0.57]  {$K$};
% Text Node
\draw (521.12,188.27) node [anchor=north west][inner sep=0.75pt]  [font=\scriptsize,color={rgb, 255:red, 208; green, 2; blue, 27 }  ,opacity=1 ,rotate=-0.57]  {$K$};
% Text Node
\draw (27,232.4) node [anchor=north west][inner sep=0.75pt]    {$P$};
% Text Node
\draw (341,172.4) node [anchor=north west][inner sep=0.75pt]    {$P$};
% Text Node
\draw (27,36.4) node [anchor=north west][inner sep=0.75pt]    {$\widehat{P}$};
% Text Node
\draw (201,192.4) node [anchor=north west][inner sep=0.75pt]  [font=\tiny,color={rgb, 255:red, 208; green, 2; blue, 27 }  ,opacity=1 ]  {$Z_{m} '$};
% Text Node
\draw (86,192.4) node [anchor=north west][inner sep=0.75pt]  [font=\tiny,color={rgb, 255:red, 208; green, 2; blue, 27 }  ,opacity=1 ]  {$Z_{1} '$};
% Text Node
\draw (56,145.23) node [anchor=north west][inner sep=0.75pt]  [font=\tiny]  {$Q_{1}$};
% Text Node
\draw (173.52,144.59) node [anchor=north west][inner sep=0.75pt]  [font=\tiny]  {$Q_{m}$};
% Text Node
\draw (578.57,62.97) node [anchor=north west][inner sep=0.75pt]  [font=\tiny]  {$Q_{m}$};
% Text Node
\draw (400,63.01) node [anchor=north west][inner sep=0.75pt]  [font=\tiny]  {$Q_{1}$};

\end{tikzpicture}
\caption{The separating (left) and non-separtaing cases (right).}
\end{figure}

\textbf{Case 2:} $Z_1'$ is non-separating in $M'$.

The proof is very similar to that of Case 1, so we just provide a sketch. We cut $M'$ along each $Z_j'$, obtaining a connected manifold with $2m$ boundary components. In this case the action of $B_m$ has two orbits on the set of these boundary components. As in the first case, we consider $m$ copies $Q_1,\dots,Q_m$ of $Q$, and glue all these manifolds along their boundaries $B_n$-equivariantly, so that boundary components of each $Q_j$ get glued to boundary components obtained after cutting $Z_j'$. 

As in the first case, the resulting manifold $R$ is non-arithmetic, has cubulable fundamental group, contains $\K$ in its adjoint trace field, and has a Charney--Davis cover.

To end the proof of the proposition we note that by varying the input manifolds $M$ and $N$, we obtain infinitely many commensurability classes.  
\end{proof}

\begin{remark}\label{rmk:smallsystole2}
By appropriately incorporating the inbreeding technique \cite{agol.systoles,belolipetsky-thomson} (see also \cite{douba-huang,douba,huang-zevenbergen}) into the proof of the proposition above, we can ensure that the resulting manifold $R$ has systole as small as we wish. Indeed, by commensurating and using subgroup separability as above, for $\ep>0$ we can additionally find two disjoint, embedded totally geodesic hypersurfaces $A,B\subset P$ such that:
\begin{itemize}
    \item $A,B$ are disjoint from $K$ and from all boundary components of $P$;
    \item $A$ is disjoint from all the $B_n$-translates of $A$ and $B$; and,
    \item there is a geodesic segment orthogonal to $A$ and $B$ in the component of $P$ cut along $A\cup B$ containing $K$, which has length less than $\ep$. 
\end{itemize}
Then all the $B_n$-translates of $A$ and $B$ embed in $R$, and if we cut along these hypersurfaces and take the double of the component containing $K$, we obtain another Charney--Davis manifold $R'$ with systole at most $2\ep$. We leave the details to the reader.
\end{remark}

%%%%%%%%%%%%%%%%%%%%%
\subsection{Proof of Theorem \ref{thmalpha:flexibleCD} and Corollary \ref{cor:connectedfaces}}\label{subsec:proofthmF}

We are ready to prove Theorem \ref{thmalpha:flexibleCD} and Corollary \ref{cor:connectedfaces}, whose proofs follow from the next result.

\begin{theorem}\label{thm:flexibleCDmanifoldsmiddle}
    Let $n\geq 2$ and $\K\neq \Q$ be a totally real number field. Then there exist infinitely many pairwise non-commensurable closed hyperbolic $n$-manifolds $(M_i)_{i\geq 1}$ satisfying the following:
    \begin{enumerate} 
        \item each $M_{2i}$ is arithmetic of simplest type and has $\K$ as field of definition;
        \item each $M_{2i+1}$ is non-arithmetic and $\K$ is contained in adjoint trace field; and,
        \item for any manifold $M'_i$ commensurable with $M_i$, there exists a Charney--Davis manifold $M$ that covers $M_i'$. Moreover, we can assume that each face of the hyperbolizing piece obtained from $M$ is connected.
    \end{enumerate}    
\end{theorem}

\begin{proof}
    For $n$ and $\K$ as in the statement, we apply Corollary \ref{cor:infiniteCDmanifolds}, and Proposition \ref{prop:nonarithmeticCD} to find an infinite sequence $(M_i)_{i\geq 1}$ of pairwise non-commensurable Charney--Davis manifolds $n$-manifolds so that
    \begin{itemize}
        \item each $M_{2i}$ is arithmetic of simplest type and with field of definition $\K$; and,
        \item each $M_{2i+1}$ is non-arithmetic, has cubulable fundamental group, and its adjoint trace field contains $\K$.
    \end{itemize}
    This proves items (1) and (2). To prove item (3), let $M_i'$ be a manifold commensurable with $M_i$, and let $M_i''$ be a common finite cover for both $M_i$ and $M_i'$. Then Corollary \ref{coro:CDcovershyp} gives us a Charney--Davis manifold $M$ that covers $M_i''$, and hence $M_i'$. Also, note from Corollary \ref{coro:CDcovershyp}~(3) that for the hyperplane system $\{Y_1,\dots,Y_n\}$ of $M$, we have that $Y_I=\bigcap_{j\in I}{Y_j}$ is connected for all $I\subset \{1,\dots,n\}$ non-empty. Then the conclusion follows since each face of the corresponding hyperbolizing piece is itself a hyperbolizing piece induced by one of the (now Charney--Davis) manifolds $Y_I$ with hyperplane system $\{Y_{I\cup\{i\}}:i\in \{1,\dots,n\}\bs I\}$.
\end{proof}

\begin{remark}\label{rmk:smallsystole3}
  By taking a stably parallelizable Charney--Davis cover of the manifold $R$ as in Remark \ref{rmk:smallsystole2}, after doubling we can ensure that the Charney--Davis manifold $R'$ has arbitrarily small systole and is also stably parallelizable. Hence it follows that the manifolds $M_{2i+1}$ from the theorem above can be chosen with arbitrarily small systoles. Moreover, by construction these systoles are achieved by closed geodesics disjoint from the any submanifolds in the hyperplane system of $R'$.
\end{remark}

\subsection{Proof of Theorem \ref{thm:main-stably-tan}} 
Let $M$ be a Charney--Davis manifold. By \cite{sullivan-covers,okun} there is a finite-sheeted cover of $M$ which is stably parallelizable. By Theorem \ref{thmalpha:flexibleCD} this stably parallelizable manifold is finitely covered by a (necessarily stably parallelizable) Charney--Davis manifold $\widehat{M}$. Cutting $\widehat{M}$ along its hyperplane system yields a stably parallelizable hyperbolizing piece. The result follows from Lemma \ref{cor:tangential}. \qed

%%%%%%%%%%%%%%%%%%%%%%%%%%%%%%%%%%%%%%%%%%%%%%

\subsection{The approach through congruence subgroups}\label{subsec:congruenceapproach}

We end this section with an alternative argument to show the existence of a stably parallelizable Charney--Davis manifold following the original construction of such manifolds from \cite[Section~6]{charney-davis}, which relies on congruence subgroups. Note however that this approach is not enough to deduce Theorem \ref{thmalpha:flexibleCD}, see Remark \ref{rem:congruence} below.

Let $\K\neq \Q$ be a totally real number field and $q$ a quadratic form on $\calO_\K^{n+1}$ as in Section \ref{subsec:hyperbolizingpiece}.
In \cite{charney-davis} it is shown that, to obtain a Charney--Davis manifold $M$ with fundamental group $\G$ commensurable with $\calO(q)$, it suffices for $\G$ to be a torsion-free \emph{principal congruence subgroup} of $\calO(q)$. That is, there exists a non-zero ideal $I$ of $\calO_\K$ such that $\G$ equals the intersection of $\calO(q)$ with
\[\G(I):=\{A\in \GL_{n+1}(\calO_\K): A \equiv \mathrm{Id} \text{ (mod }I)\}\]
and satisfies $\G=\G(I)\cap \calO(q)<\mathrm{SO}_0(n,1)$.

To that end, Charney and Davis invoke a result of Millson–Raghunathan (see \cite[Lemma~6.7]{charney-davis}) to produce an ideal $I\subset\calO_{\K}$ that factors as
\beq
I=I_1\cdots I_s,
\eeq
for $I_1,\dots,I_s$ suitably chosen, pairwise relatively prime ideals. 
In this way, $M:=\Hyp^n/(\Gamma(I)\cap \calO(q))$ is a Charney--Davis manifold.

In order to ensure that $M$ stably parallelizable, we pass to a deeper congruence subgroup.
More precisely, we use that $\mathcal O_\K$ has infinitely many prime ideals to choose two prime ideals $m_1,m_2\subset\mathcal O_\K$ which do not divide $I$ and whose residue fields $\mathcal O_\K/m_1$ and $\mathcal O_\K/m_2$ have distinct characteristics.  Set
\beq
J:=I\cdot m_1\cdot m_2\subset\mathcal O_\K,
\eeq
and consider the principal congruence subgroup $\Gamma(J)\subset\Gamma(I)\subset\GL_{n+1}(\mathcal O_\K)$.

By definition every element in $\Gamma(J)$ reduces to the identity in $\GL_{n+1}(\mathcal O_\K/J)$.  As $I,m_1,m_2$ are pairwise coprime, the Chinese remainder theorem yields an isomorphism
\beq
\GL_{n+1}(\mathcal O_\K/J)\cong \GL_{n+1}(\mathcal O_\K/I)\times\GL_{n+1}(\mathcal O_\K/m_1)\times\GL_{n+1}(\mathcal O_\K/m_2).
\eeq
Therefore, any element in $\Gamma(J)$ projects to the identity in each factor, so in particular their reductions modulo $m_1$ and $m_2$ are trivial. By \cite[Proposition]{deligne-sullivan} and \cite{sullivan-covers}, the finite sheeted cover $\widehat{M}\to M$ corresponding to $\Gamma(J)\cap \calO(q)$ is stably parallelizable. Since $\Gamma(J)\subset\Gamma(I)$, this cover is also a Charney--Davis manifold by \cite[Lemma~6.6]{charney-davis}.

\begin{remark}\label{rem:congruence}
A similar proof of this proposition is sketched in \cite[Lemma~6.32]{davis-book}, but unfortunately there is a gap. To get a stably parallelizable manifold one constructs a finite index subgroup of $\G=\G(I)\cap \calO(q)$ of the form $\G'=\G(J)\cap \calO(q)$ for $J$ an appropriate ideal of $\calO_\K$ contained in $I$. In Davis's argument, the chosen subgroup $\G'$ is only required to be of finite index in $\G$ and such that $\what{M}=\Hy^n/\G'$ is stably parallelizable. But there is no guarantee that such a manifold $\what{M}$ is Charney--Davis (we cannot verify conclusion (4) in Theorem \ref{thm:CD-piece}). This is because $\G'$ is not necessarily a principal congruence subgroup, not even a congruence subgroup (i.e. $\G'$ does not necessarily contain a principal congruence subgroup of $\calO(q)$), so we cannot apply \cite[Lemma~6.6]{charney-davis}. Indeed, it is known that any arithmetic Kleinian group contains a finite index subgroup that is not congruence, see e.g.~\cite{lubotzky} or \cite[Theorem~1.8]{LLR}. This last fact is also the reason why the argument given above is insufficient to prove Theorem \ref{thmalpha:flexibleCD}.
\end{remark}
%%%%%%%%%%%%%%%%%%%%%%%%%%%%%%%%%%%%%%%%%%%%%%%%%%%%%%%%%%%%%%%%%%%
\section{Foldable cubulations of flat manifolds}\label{sec:foldableflatmflds}
In this section we study flat manifolds and flat cube complexes. The main result is Proposition \ref{thmA:diagonalfoldable}, which asserts that flat manifolds of diagonal type admit foldable flat cubulations. This property turns them into a suitable input for strict hyperbolization. 

\subsection{Flat manifolds}
A closed and connected smooth manifold is \emph{flat} if it has zero sectional curvatures. By Bieberbach theorem, an $n$-dimensional flat manifold is the quotient of $\R^n$ by a torsion-free and cocompact discrete group $\G<\mathrm{\Isom}(\R^n)=\mathrm{O}(n)\ltimes \R^n$. Such a group $\G$ fits into the exact sequence
\begin{equation}\label{eq:extension1}
    1\ra \Z^n \ra \G \ra \Phi\ra 1,
\end{equation}
where $\Z^n$ is a maximal abelian normal subgroup of $\G$ and $\Phi$ is a finite group, called the \emph{holonomy} of $\G$. The action of $\Phi$ on $\Z^n$ by conjugation in $\G$ induces a \emph{holonomy representation} $\Phi \ra \GL_n(\Z)$.  A flat manifold is of \emph{diagonal type} if its holonomy representation is diagonalizable. This implies that the group $\Phi$ is isomorphic to $(\Z/2\Z)^k$ for some $k$. If the manifold is $n$-dimensional, then $k\leq n-1$. 
\subsection{Flat cube complexes}\label{subsec:flatcube}
Recall that $\msf{R}_n$ refers to $\R^n$ equipped with its cubical structure induced by the standard translation action of $\Z^n$. 

\begin{definition}\label{def.flat}
 An $n$-dimensional cube complex is \emph{flat} if its universal cover is isomorphic (as a cube complex) to $\msf{R}_n$.    
\end{definition}

By Bieberbach theorem it follows that every finite $n$-dimensional flat cube complex is admits a flat Riemannian manifold. In fact, any such complex is finitely covered by an $n$-torus with cubical structure isomorphic to $\msf{R}_n/\G$ for $\G$ a finite index subgroup of $\Z^n$. Note however that a flat cube complex is not necessarily foldable. The standard $n$-torus $\bbT^n=\msf{R}_n/\Z^n$ is not foldable, but its $2^n$-sheeted cover $\what\bbT^n=\msf{R}_n/(2\Z)^n$ is foldable, as we will explain below.
\subsection{Cubulating flat manifolds of diagonal type}\label{subsec:foldingflat} In this subsection we prove Proposition \ref{thmA:diagonalfoldable}.
First, we give an explicit folding map for the standard cubulation $\msf{R}_n$ of $\R^n$. 

Let $\msf{f}_1:\R \ra [0,1]$ be periodic map with period 2 that restricts to $x \mapsto |x|$ on the interval $[-1,1]$.
This is the unique folding map from $\msf{R}_1=\R$ to $\square^1=[0,1]$ that restricts to the identity on $\square^1$. 

By construction, $\msf{f}_1$ satisfies $\msf{f}_1(-x)=\msf{f}_1(x)$ and $\msf{f}_1(x+2)=\msf{f}_1(x)$ for all $x\in \msf{R}_1$. 

\begin{figure}[h!!]\label{fig:folding}
%\alt{Image of the Earth from outer space}

\tikzset{every picture/.style={line width=0.75pt}} %set default line width to 0.75pt        

\tikzset{every picture/.style={line width=0.75pt}} %set default line width to 0.75pt        

\begin{tikzpicture}
[x=0.75pt,y=0.75pt,yscale=-1,xscale=1]

%uncomment if require: \path (0,300); %set diagram left start at 0, and has height of 300

%Straight Lines [id:da3595728580609324] 
\draw    (220.31,150.25) -- (250.4,180.19) -- (280.15,149.94) -- (310.15,179.94) -- (340.15,150.19) -- (370.15,179.69) -- (400.15,149.94) -- (430.4,180.06) ;
%Straight Lines [id:da24445427377804052] 
\draw    (310.17,120.06) -- (310.17,239.72) ;
%Straight Lines [id:da977004718675722] 
\draw    (180.15,180.31) -- (450.4,180.31) ;
%Straight Lines [id:da7214779536887098] 
\draw    (220.31,150.25) -- (190.15,180.31) ;

% Text Node
\draw (310.9,182.34) node [anchor=north west][inner sep=0.75pt]  [font=\tiny]  {$0$};
% Text Node
\draw (337.73,182.5) node [anchor=north west][inner sep=0.75pt]  [font=\tiny]  {$1$};
% Text Node
\draw (397.56,182.67) node [anchor=north west][inner sep=0.75pt]  [font=\tiny]  {$3$};
% Text Node
\draw (273.65,183) node [anchor=north west][inner sep=0.75pt]  [font=\tiny]  {$-1$};
% Text Node
\draw (243.81,182.34) node [anchor=north west][inner sep=0.75pt]  [font=\tiny]  {$-2$};
% Text Node
\draw (310.65,146.34) node [anchor=north west][inner sep=0.75pt]  [font=\tiny]  {$1$};
% Text Node
\draw (436.67,151.82) node [anchor=north west][inner sep=0.75pt]    {$...$};
% Text Node
\draw (171,151.15) node [anchor=north west][inner sep=0.75pt]    {$...$};
% Text Node
\draw (366.56,182.67) node [anchor=north west][inner sep=0.75pt]  [font=\tiny]  {$2$};
% Text Node
\draw (213.81,182.34) node [anchor=north west][inner sep=0.75pt]  [font=\tiny]  {$-3$};
\end{tikzpicture}

\caption{The folding map $\msf{f}_1:\R\to [0,1]$.}
\end{figure}

In general, the map $\msf{f}_n:\msf{R}_n \ra \square^n$ given by
    $(x_1,\dots,x_n)\mapsto (\msf{f}_1(x_1),\dots,\msf{f}_1(x_n))$
is the unique folding map that restricts to the identity on $\square^n$. From the properties of $\msf{f}_1$ we easily deduce the following.
\begin{lemma}\label{lem:foldinginvariant}
    Let $\G$ be a discrete subgroup of isometries of $\R^n$ such that every element of $\G$ acts according to   \[(x_1,\dots,x_n)\mapsto ((-1)^{\ep_1}x_1+2a_1,\dots,(-1)^{\ep_n}x_n+2a_n)\]
    for some $\ep_1,\dots,\ep_n\in \{0,1\}$ and $a_1,\dots,a_n\in \Z$. Then $\G$ preserves the cubical structure  of $\msf{R}_n$ and the folding map $\msf{f}_n$ is $\G$-invariant. In particular, $\msf{f}_n$ induces a folding map on the quotient cube complex $\msf{R}_n/\G$.
\end{lemma}

Equivalently, if $D_n<\mathrm{O}(n)$ denotes the group of diagonal matrices with diagonal entries in $\{1,-1\}$, then the folding map $\msf{f}_n:\msf{R}_n\ra \square^n$ is invariant under the group $D_n \ltimes (2\Z)^n< \mathrm{O}(n)\ltimes \R^n=\mathrm{Isom}(\R^n)$.

Recall that an $n$-dimensional flat manifold $F$ is of diagonal type if its holonomy group is isomorphic to $(\Z/2\Z)^k$ for some $k$. If $F$ is such a manifold, then $k\leq n-1$ and $\G=\pi_1(F)$ fits into the short exact sequence
\begin{equation}\label{eq:SESdiagonal}
    1 \ra \Z^n \ra \G \ra (\Z/2\Z)^k\ra 1. 
\end{equation}
The main result of this section is the following.
\begin{proposition}\label{thmA:diagonalfoldable}
Let $F$ be an $n$-dimensional flat manifold of diagonal type. Then there exists a flat foldable cube complex $\msf{C}$ and a homeomorphism $\msf{C}\ra F$ that restricts to a smooth embedding on each cube of $\msf{C}$.
\end{proposition}

\begin{proof}
    Let $F$ be an $n$-dimensional flat manifold of diagonal type, so that $\G=\pi_1(F)$ splits as in \eqref{eq:SESdiagonal}. Then the induced holonomy representation $(\Z/2\Z)^k \ra \GL_n(\Z)$ maps $(\Z/2\Z)^k$ into  $D_n<\mathrm{O}(n)$. Moreover, $\G$ injects into $D_n \ltimes (\frac{1}{2}\Z)^n$ \cite[Lemma~1.4]{miatello-rossetti}. 
    
    If $A:\R^n\ra \R^n$ is the homothety $(x_1,\dots,x_n)=(4x_1,\dots,4x_n)$, then by conjugating by $A$ in $\GL_n(\R)\ltimes \R^n$ we obtain
    \[A\left(D_n \ltimes \left(\frac{1}{2}\Z\right)^n\right)A^{-1}\subset D_n\ltimes (2\Z)^n. \]
    This gives us an injection of $\G$ into $D_n\ltimes (2\Z)^n$. If $\G_0$ denotes the image of $\G$ under this injection, then Lemma \ref{lem:foldinginvariant} tells us that $F$ is homeomorphic to the flat and foldable cube complex $\msf{R}_n/\G_0$. It is immediate that this description provides a smooth cubulation of $F$.
\end{proof}

%%%%%%%%%%%%%%%%%%%%%%%%%%%%%%%%%%%%%%%%%%%%%%%%%%%%%%%%%%%%%%%%%%%%%%%%%%%%%%%%%%%%%%%%%%%%%%%%%%%%%%%%

\section{Hyperbolizing flat cube complexes}\label{sec:hyperbolizingflat}
In this section we prove the following proposition, which states that hyperbolizing a flat foldable cube complex yields a hyperbolic manifold commensurable with the original Charney--Davis manifold used to define the hyperbolizing piece. For a closed Riemannian manifold $M$ with non-positive sectional curvatures, recall that its injectivity radius $\InjRad(M)$ can be computed as half of the length of a shortest closed geodesic \cite[p.~258]{petersen}.

\begin{proposition}\label{thm:flat->hyperbolic}
    Let $\calH_X$ be a strict hyperbolization procedure with hyperbolizing piece $X$ obtained from the $n$-dimensional Charney--Davis manifold $M$. If $\msf{C}$ is a compact flat foldable $n$-dimensional cube complex, then $\calH_X(\msf{C})$ is a closed hyperbolic manifold commensurable with $M$. Moreover, there exists $B>0$ depending on $\msf{C}$, but not on $X$, such that 
    \[\InjRad(\calH_X(\msf{C}))\geq B\cdot \InjRad(M).\]
\end{proposition}

Hyperbolicity of these hyperbolized complexes was proved in \cite[Lemma~3.2]{belegradek}, so what is left to show is the commensurability with the Charney--Davis manifold, and the bound on the injectivity radius.

Throughout this section we fix a Charney--Davis $n$-manifold $M$ with corresponding hyperbolizing piece $(X,g)$. For a foldable cube complex $\msf{C}$ of dimension $n$, we let $\calH(\msf{C})=\calH_X(\msf{C})$ be its hyperbolization, and let $f_X$ and $g_\msf{C}$ be the projections from \eqref{eq.diagramhyp}. Given a cell $\sigma$ of $\msf{C}$, recall the notation $\calH(\sigma)=g_{\msf{C}}^{-1}(\sigma)\subset \calH(\msf{C})$. 

We first note that cubical actions on certain foldable cube complexes induce actions on their hyperbolizations.
\begin{lemma}\label{lem:liftaction}
    Let $\Phi$ be a group acting cubically on a foldable cube complex $\msf{C}$ of dimension $n$. If all the maximal cells of $\msf{C}$ are $n$-dimensional, then there exists a continuous action of $\Phi$ on $\calH(\msf{C})$ so that $g_\msf{C}: \calH(\msf{C})\ra \msf{C}$ is $\Phi$-equivariant.
\end{lemma}
\begin{proof}
    Let $f:\msf{C}\ra \square^n$ be a folding map. Given $\gam \in \Phi$ and an $n$-cell $\sigma$ of $\msf{C}$, there exists a unique isometry $\gam_\sigma$ of $\square^n$ that fits into the commutative diagram
\begin{equation}\label{eq.diagram1}
    \begin{tikzcd}
\sigma \arrow[r,"\gam"] \arrow[d,"f"]  & \gam(\sigma) \arrow[d,"f"]  \\
\square^n \arrow[r,"\gam_\sigma"] & \square^n.  
\end{tikzcd}
\end{equation}
 We define $\what{\gam}_\sigma: \calH(\sigma)\ra \calH(\gam(\sigma))$ according to \begin{equation*}
     \what \gam_\sigma(c,x)=(\gam(c),\gam_\sigma(x))
 \end{equation*}
for $(c,x)\in \calH(\sigma)$. This map is well-defined by \eqref{eq.diagram1} and by the $B_n$-equivariance of $g:X \ra \square^n$, see Lemma \ref{lem:CDpieceX}. 

We want to glue the maps $\what \gam_\sigma$ to obtain a global map on $\calH(\msf{C})$. Indeed, if $\sigma_1,\sigma_2$ are $n$-cells of $\msf{C}$ with non-empty intersection $\sigma_1\cap \sigma_2=D$, then by \eqref{eq.diagram1} we have that
\[\gam_{\sigma_1}(f(c))=f(\gam(c))=\gam_{\sigma_2}(f(c))\]
for any $c\in D$. Hence $\gam_{\sigma_2}^{-1}\circ \gam_{\sigma_1}$ pointwise fixes the face $f(D)\subset \square^n$. By the construction of $X$ from Section \ref{subsec:hyperbolizingpiece}, this forces $\gam_{\sigma_2}^{-1}\circ \gam_{\sigma_1}$ to also (pointwise) fix the face $g^{-1}(f(D))$ of $X$. In particular, if $(c,x)\in \calH(D)$, then $c\in D, x\in g^{-1}(f(D))$, and 
\[\what \gam_{\sigma_1}(c,x)=(\gam(c),\gam_{\sigma_1}(x))=(\gam(c),\gam_{\sigma_2}(x))=\what \gam_{\sigma_2}(c,x).\]

In conclusion, for $\gam\in \Phi$ we have a well-defined continuous map $\what \gam: \calH(\msf{C})\ra \calH(\msf{C})$ given by the restriction of $\what \gam_\sigma$ on each $\calH(\sigma)$ with $\sigma$ an $n$-cell of $\msf{C}$ (here we use that each maximal cell of $\msf{C}$ is $n$-dimensional). From the uniqueness of $\gam_\sigma$ in \eqref{eq.diagram1}, it follows that $\gam \mapsto \what \gam$ induces a continuous action of $\Phi$ on $\calH(\msf{C})$. Moreover, the map $g_{\msf{C}}:\calH(\msf{C})\ra \msf{C}$ is $\Phi$-equivariant by construction. 
\end{proof}
We continue by showing that covering maps between foldable cube complexes induce covering maps between the corresponding hyperbolizations.

\begin{lemma}\label{lem:regularcover}
Let $\msf{C}$ be an $n$-dimensional foldable cube complex with fundamental group $\Phi$ and let $q:\wt{\msf{C}}\ra \msf{C}$ be the universal covering map. Then the action of $\Phi$ on $\wt{\msf{C}}$ induces a natural free and properly discontinuous action of $\Phi$ on $\calH(\wt{\msf{C}})$, so that the quotient is naturally identified with $\calH(\msf{C})$.
\end{lemma}

\begin{proof}
    Let $f:\msf{C}\ra \square^n$ be a folding map and use $\wt f=f \circ q$ as folding map to describe $\calH(\wt{\msf{C}})$. 

    Consider the action of $\Phi$ on $\calH(\wt{\msf{C}})$ given by $\gam(c,x)=(\gam(c),x)$ for $(c,x)\in \calH(\wt{\msf{C}})$ and $\gam\in \Phi$. This action is well-defined since $\wt f(\gam(c))=\wt f(c)$ for $c\in \wt{\msf{C}}$ and $\gam\in \Phi$. This action is also free and properly discontinuous, since the action of $\Phi$ on $\wt{\msf{C}}$ is free and properly discontinuous and the map $g_{\wt{\msf{C}}}:\calH(\wt {\msf{C}})\ra \wt{\msf{C}}$ is $\Phi$-equivariant. Hence $\calH(\wt{\msf{C}})\ra \calH(\wt{\msf{C}})/\Phi$ is a regular covering map.
    
    We are left to produce a $\Phi$-equivariant map $\wt{q}: \calH(\wt{\msf{C}})\ra \calH(\msf{C})$ that induces a homeomorphism $\calH(\wt{\msf{C}})/\Phi\ra \calH(\msf{C})$. To this end, we define $\wt q:  \calH(\wt{\msf{C}})\ra \calH(\msf{C})$ according to $\wt q(c,x)=(q(c),x)$. This map is well-defined, continuous and surjective by the definition of $\wt f$. This map is clearly $\Phi$-invariant, and we have $\wt q(c,x)=\wt q(c',x')$ if and only if $(c',x')=\gam(c,x)$ for some $\gam\in \Phi$, as desired.
\end{proof}

\begin{corollary}\label{cor:nonregularcover}
    Let $\msf{C}$ be an $n$-dimensional foldable cube complex and let $\what{\msf{C}}\ra \msf{C}$ be a (non-necessarily regular) cover of $\msf{C}$. Then $\calH(\what{\msf{C}})$ is a cover of $\calH(\msf{C})$. Moreover, if the cover $\what{\msf{C}}\ra \msf{C}$ is regular (resp. finite of degree $d$), then the cover $\calH(\what{\msf{C}})\ra\calH(\msf{C})$ is also regular (resp. finite of degree $d$).
\end{corollary}

\begin{proof}
It follows immediately from the previous lemma since for the universal cover $\wt{\msf{C}}\ra \what{\msf{C}}$, we have $\calH(\msf{C}) \cong\calH(\wt{\msf{C}})/\Phi$ and $\calH(\what{\msf{C}})\cong\calH(\wt{\msf{C}})/\what\Phi$ for $\pi_1(\what{\msf{C}})=\what\Phi<\Phi=\pi_1(\msf{C})$.
\end{proof}
The next step in the proof of Proposition \ref{thm:flat->hyperbolic} is to show that hyperbolizing $\what\bbT^n=\msf{R}_n/(2\Z)^n$ yields a finite cover of $M$, which is the content of the next lemma.
\begin{lemma}\label{lem:T_ncommM}
There exists a covering map $q:\calH(\what\bbT^n)\ra  M$ of degree $2^n$. 
\end{lemma}

\begin{proof}
Let $\Phi=(\bbZ/2)^n$ act on
$\what{\bbT}^n=\msf{R}_n/(2\bbZ)^n$
by translations by elements of $\{0,1\}^n$. This is a free cubical action, and the quotient is $\what{\bbT}^n/\Phi\cong \bbT^n$.
By Lemma \ref{lem:liftaction}, this action lifts to a continuous free action of $\Phi$ on $\calH(\what{\bbT}^n)$.
Therefore the quotient map
\beq
\calH(\what{\bbT}^n)\ra \calH(\what{\bbT}^n)/\Phi
\eeq
is a covering map of degree $|\Phi|=2^n$.

It remains to identify the quotient. Passing from $\what{\bbT}^n$ to $\what{\bbT}^n/\Phi=\bbT^n$ identifies the $2^n$ top-dimensional cubes of the cubulation $\what{\bbT}^n$ to a single cube. Accordingly, after passing to the quotient by $\Phi$, the $2^n$ copies of the hyperbolizing piece $X$ occurring in $\calH(\what{\bbT}^n)$ are identified to a single copy of $X$. The additional face identifications of $X$ induced by this quotient map are precisely the opposite-face identifications coming from the (one-top cube) cubical structure on $\bbT^n$, which are given by the reflections in $B_n$. Thus the quotient $\calH(\what{\bbT}^n)/\Phi$ is obtained from a single copy of $X$ by the same face identifications induced by the quotient $\pi:X\ra M$, concluding that
$\calH(\what{\bbT}^n)/\Phi\cong M$.
\end{proof}
Our last step is to show that any flat and foldable cube complex is commensurable with $\what\bbT^n$.

\begin{lemma}\label{lem:commT_n}
    Let $\msf{C}$ be an $n$-dimensional, flat and foldable cube complex. If $\msf{C}$ is finite, then $\msf{C}$ and $\what\bbT^n$ have isomorphic regular finite covers.
\end{lemma}

\begin{proof}
    Let $\G$ be the fundamental group of $\msf{C}$, and let it act on $\msf{R}_n$ by Deck transformations, so that this action is combinatorial. Since $\msf{C}$ is homeomorphic a flat $n$-manifold, by Bieberbach theorem $\G$ has a finite index subgroup $\what \G$ isomorphic to $\Z^n$. Then $\what\G$ acts combinatorially, properly and cocompactly on $\msf{R}_n$, and hence by the cubical flat torus theorem \cite[Theorem~3.6]{wise-woodhouse} we can assume that $\what\G$ acts on $\msf{R}_n$ by translations. That is, $\what\G$ acts as the subgroup $(m_1\Z)\times \cdots \times (m_n\Z)<\Z^n$ for some $m_1,\dots,m_n\in \Z^+$. By possibly considering a further finite index subgroup, we can assume that each $m_i$ is even and that $\what\G$ is normal in $\G$. Then $\what\G<(2\Z)^n$ and therefore $\msf{R}_n/\G'$ is a finite regular cover for both $\msf{C}$ and $\what\bbT^n$.
 \end{proof}

\begin{proof}[Proof of Proposition \ref{thm:flat->hyperbolic}]
    Let $\msf{C}$ be a finite, $n$-dimensional, flat and foldable cube complex. By Lemma \ref{lem:commT_n}, $\msf{C}$ and $\what\bbT^n$ have a common isomorphic finite regular cover $\what{\msf{C}}$. Let $d_1,d_2$ be the degrees of the covers $\what{\msf{C}}\ra \msf{C}$ and $\what{\msf{C}}\ra \what\bbT^n$ respectively. Then Corollary \ref{cor:nonregularcover} implies that $\what{M}=\calH(\what{\msf{C}})$ is a regular finite cover of both $\calH(\msf{C})$ and $\calH(\what\bbT^n)$, with degrees $d_1$ and $d_2$ respectively. Since $\calH(\what\bbT^n)$ is a degree $2^n$ cover of $M$ by Lemma \ref{lem:T_ncommM}, $\calH(\msf{C})$ is commensurable with $M$ and the degrees of the covers $M'\ra \calH(\msf{C})$ and $M'\ra M$ are $d_1$ and $d_2\cdot 2^n$ respectively. 
    In particular, this implies that $\calH(\msf{C})$ admits a hyperbolic metric and
    \[\InjRad(\calH(\msf{C}))\geq B\cdot  \InjRad(M)\]
    for $B=d_2^{-1}\cdot 2^{-n}$.
\end{proof}

\begin{remark}\label{rmk:equalitysystole}
    It is immediate from the proof above that if the systole of the Charney--Davis manifold $M$ is achieved by a closed geodesic disjoint from any submanifold in its hyperplane system, then $\InjRad(\calH(\msf{C}))=B\cdot \InjRad(M)$.
\end{remark}

%%%%%%%%%%%%%%%%%%%%%%%%%%%%%%%%%%
\subsection{Smooth structures on strict hyperbolization of flat cube complexes}\label{sec:smoothstructures}
We end this section with an observation about the smooth structure on the strict hyperbolization of an $n$-dimensional flat foldable (smooth) cube complex $\msf{C}$. 

By Proposition \ref{thm:flat->hyperbolic} (indeed, by \cite[Lemma~3.2]{belegradek}), $\calH(\msf{C})$ is equipped with a hyperbolic metric $g_{hyp}$. Recall that this metric is so that the hyperbolization $\calH(\sig)$ of a top-dimensional cell $\sig \subset \msf{C}$ is naturally isometric to the hyperbolizing piece $X$. Let $\calS_{hyp}$ denote the smooth structure on $\calH(\msf{C})$ induced by this metric (equivalently, by the atlas of local isometries to $\mathbb H^n$). On the other hand, recall from Section \ref{sec:propertieshyperbolization} that $\calH(\msf{C})$ is also equipped with a normal smooth structure, which we denote by $\calS_\nu$. It turns out that these smooth structures coincide.

\begin{lemma}\label{lem:normal coordinates}
The smooth structures $\calS_{hyp}$ and $\calS_{\nu}$ on $\calH(\msf{C})$ are diffeomorphic. 
\end{lemma}

Before proving this lemma, we give a few more details about the definition of the normal smooth structure \cite{ontaneda.cube,ontaneda.smoothing}.
It relies on Ontaneda's notion of a \emph{normal smoothing}, which is built from link data. More precisely, for each open $i$-cube $\square\subset \msf{C}$ one chooses a link smoothing
\beq
h_\square:S^{n-i-1}\to \mathrm{Link}(\square,\msf{C}),
\eeq
and these choices determine a normal atlas; if all transition maps are smooth, this atlas defines a smooth structure on $\msf{C}$, called a \emph{normal smooth structure} with respect to the cubulation. In the case relevant here, $\msf{C}$ is \emph{flat}, and the canonical smooth structure on $\msf{C}$ is normal with respect to its cubulation. In fact, after lifting to the universal cover $\R^n$ with its standard cubulation $\msf{R}_n$, one may take the standard round link smoothings and then descend them to $\msf{C}$ via the deck group,
obtaining a normal smooth atlas whose transition maps are affine.

The strict hyperbolization $\calH(\msf{C})$ carries a natural hyperbolized cell structure whose top-dimensional cells are isometric copies of the Charney--Davis piece $X$, glued along totally geodesic faces according to the face identifications in $\msf{C}$. For each open $k$-face $\square^k$ in the cube $\square^n$, let $X_{\square^k}\subset \calH(\msf{C})$ denote the corresponding open stratum. Ontaneda defines a geometric link $\mathrm{Link}(X_{\square^k},\calH(\msf{C}))$ and chooses a link smoothing
\beq
h_{\square^k}:S^{n-k-1}\to \mathrm{Link}(X_{\square^k},\calH(\msf{C})).
\eeq
Fixing $r>0$ smaller than the normal injectivity radius of all strata, one obtains a normal chart
\beq
H_{\square^k}:D^{n-k}\times \mathrm{int}(X_{\square^k})\to \calH(\msf{C}),
\eeq
given by $H_{\square^k}(t v,p)=\exp_p(2r t\, h_{\square^k}(v))$. Ontaneda proves that, when the input smooth structure on $\msf{C}$ is normal, the link smoothings can be chosen so that the resulting normal atlas has smooth transition maps; the maximal smooth atlas generated by these normal charts is $S_{\nu}$.

\begin{proof}[Proof of Lemma \ref{lem:normal coordinates}]
Fix a $k$-face $\square^k\subset \square^n$ and write $X_{\square^k}\subset \calH(\msf{C})$ for the corresponding hyperbolized stratum. Since each hyperbolizing piece has totally geodesic faces meeting orthogonally, and the gluing is performed by face isometries, $X_{\square^k}$ is a totally geodesic submanifold of $(\calH(\msf{C}),g_{hyp})$.

Choose $r>0$ smaller than the normal injectivity radius of every stratum. The exponential map gives a tubular neighborhood of $\mathrm{int}(X_{\square^k})$, and after choosing an identification of the unit normal sphere with $S^{n-k-1}$ we obtain a chart of the form
\begin{equation}\label{eq:Fermi}
\Theta_{\square^k}\colon D^{n-k}\times \mathrm{int}(X_{\square^k})\to \calH(\msf{C}),
\qquad
\Theta_{\square^k}(t v,p)=\exp_p(r t\, v),
\end{equation}
where $v\in S^{n-k-1}$ and $t\in [0,1)$. By construction, $\Theta_{\square^k}$ is smooth in the smooth structure $\calS_{hyp}$.

Since $\msf{C}$ is flat and strict hyperbolization preserves links, we may choose the link smoothings $h_{\square^k}$ so that, under this identification, $h_{\square^k}$ is the standard parametrization of the round unit sphere in the normal space. With this choice, $h_{\square^k}(v)$ is exactly the unit normal vector $v$, and hence the normal chart $H_{\square^k}$ coincides (up to the choice of scale) with the ``Fermi'' chart $\Theta_{\square^k}$ in \eqref{eq:Fermi}. In particular, every Ontaneda normal chart is a smooth chart for $\calS_{hyp}$.

Conversely, let $\phi\colon U\to \mathbb H^n$ be a hyperbolic chart for $\calS_{hyp}$, i.e. a local isometry.
Cover $U$ by images of Ontaneda charts $H_{\square^k}$. On an overlap $U\cap \mathrm{Im}(H_{\square^k})$
we can write
\beq
\phi = (\phi\circ H_{\square^k})\circ H_{\square^k}^{-1}.
\eeq
Here $H_{\square^k}^{-1}$ is smooth by definition of $\calS_{\nu}$, and $\phi\circ H_{\square^k}$ is smooth. Therefore $\phi$ is smooth with respect to $\calS_{\nu}$. This shows that the identity map is a diffeomorphism between the two smooth structures.
\end{proof}
\begin{corollary}\label{cor:tangential-hyp}
If $\msf{C}$ is a flat foldable cube complex, then the strict hyperbolization  $\calH_X(\msf{C})$ with the smooth structure $\calS_{hyp}$ embeds smoothly in $\msf{C}\times X$ with trivial normal bundle. Consequently, if $X$ is parallelizable, the map $g_{\msf{C}}:\calH_X(\msf{C})\to\msf{C}$ is stably tangential.
\end{corollary}
\begin{proof}
The result then follows from the Addendum to Main Theorem in \cite{ontaneda.smoothing}. The second part is Lemma \ref{lem:normal coordinates}.
\end{proof}
%%%%%%%%%%%%%%%%%%%%%%%%%%%%%%%%%%%%%%%%%%%%%%%%%%%%%%%%%%%%%%%%%%%%%%%%%%%%%%%%%%%%%%%%%%%%%%%%%%%%
\section{Lee--Szczarba manifolds and their characteristic classes}\label{sec:characteristicclassesLS}

In this section we study the flat manifolds constructed by Lee--Szczarba \cite{lee-szczarba}, which have non-trivial Stiefel--Whitney and Pontryagin classes.  By Corollary \ref{cor:tangential-hyp} and Propositions \ref{thm:flat->hyperbolic} and \ref{thmA:diagonalfoldable}, the strict hyperbolizations of these manifolds yield closed hyperbolic manifolds with non-trivial characteristic classes, proving parts (1) and (2) of Theorem \ref{thm:Pontryagin}. We also extend Lee--Szczarba's computations in two directions: first, by identifying additional degrees in which the Pontryagin classes are non-zero; and second, by computing these characteristic classes for the orientable double covers.
\subsection{Background on characteristic classes}
A basic reference for the theory of characteristic classes of vector bundles is \cite{milnor-stasheff}.

Recall that the Stiefel--Whitney classes of a smooth manifold $M$ are the Stiefel--Whitney classes of its tangent bundle $TM\to M$. We denote them by $w_i(M)\in H^i(M;\F_2)$. They are primary obstructions to orientation and (higher) spin structures.
A smooth manifold $M$ is orientable if and only if $w_1(M)=0$, and an orientable smooth manifold admits a spin structure if and only if $w_2(M)=0$. 
Likewise, an orientable smooth manifold admits a spin$^c$ structure if and only if $w_2(M)$ admits an integral lift, that is, if it lies in the image of the mod $2$ reduction map $H^2(M;\Z)\to H^2(M;\F_2)$. Observe that if $M$ is a spin$^c$ manifold, then $w_3(M)=0$, but the converse does not hold (see Section \ref{subsec:nontrivialSW}). There are also higher analogues of spin structures whose existence is obstructed by Stiefel--Whitney classes \cite{albanese-milivojevic}. For example, if a manifold admits a spin$^h$ structure, then $w_4(M)$ admits an integral lift. Therefore if $M$ supports a spin$^h$ structure, then $w_5(M)=0$.

Similarly, the Pontryagin classes $p_i(M)\in H^{4i}(M;\Z)$ of a smooth manifold $M$ are the Pontryagin classes of its tangent bundle. They satisfy the relation
\beq
w_{2i}(M)^2 \equiv p_i(M)\,  ( \mathrm{mod}\, 2)
\eeq
in $H^{4i}(M;\F_2)$.

Finally, the defining property of characteristic classes is that they are natural with respect to morphisms of vector bundles. A special case of interest to us is the following: if $f:M\to N$ is a smooth map between smooth manifolds which is covered by a map of stable tangent bundles (e.g. a covering map), then
\beq
f^*w_i(N)=w_i(M)\qquad \text{and}\qquad f^*p_i(N)=p_i(M).
\eeq
\subsection{Lee--Szczarba's construction}\label{subsec:LS}
Given $n\geq 2$, let $\G_n$ be the group of isometries of $\R^n$ generated by the elements $\ga_0,\dots,\ga_{n-1}$ given by
\begin{equation}\label{eq:LSequations}
\gamma_i(x_1,\ldots,x_n)=
\begin{cases}
(x_1+1,x_2,\ldots,x_n) & \text{ if  } i=0\\
(x_1,\ldots,-x_i,x_{i+1}+\frac{1}{2},x_{i+2},\ldots,x_n) & \text{ if  } 1\leq i < n.
\end{cases}
\end{equation}
Then $\Gamma_n$ is discrete and torsion-free, and the quotient $\msf{LS}_n:=\R^n/\Gamma_n$ is a (non-orientable) flat manifold of dimension $n$ whose holonomy group is isomorphic to $(\Z/2\Z)^{n-1}$. In \cite{lee-szczarba}, the following was proven. 

\begin{theorem}[{\cite[Theorem]{lee-szczarba}}]\label{thm:LS}
The Lee--Szczarba manifolds satisfy 
\begin{equation*}
    w_i(\msf{LS}_n)\in H^i(\msf{LS}_n;\F_2)\qquad\text{ for }\qquad 1 \leq i \leq n-1,
\end{equation*} 
\begin{equation*}
    w_{2i}(\msf{LS}_n)^2\neq 0\in H^{4i}(\msf{LS}_n;\F_2)\qquad\text{ whenever }\qquad n=6k+4 \text{ for } k\geq 1 \text{ and }i\leq k.
\end{equation*} 
In particular, the Pontryagin classes of $\msf{LS}_n$ satisfy $p_i(\msf{LS}_n)\neq 0$ whenever $n=6k+4$ for $k\geq 1$ and $i\leq k$. 
\end{theorem}
\begin{remark}
Lee--Szczarba seem to include $p_1(\msf{LS}_4)\neq 0$ in their statement \cite[Theorem]{lee-szczarba}. However, this does not follow from their proof. In fact, with their method, the first dimension in which a non-trivial $p_1$ is detected is $n=6$. See Remark \ref{rmk:k=1 LS} below.
\end{remark}

To refine the computations above, we get into Lee--Szczarba's computations from \cite{lee-szczarba} in more detail. First, we note the short exact sequence
\begin{equation}\label{eq:flat}
1\to\Z^{n}\to\G_n\xto{h}\Phi_n\to 1,
\end{equation}
with $\Phi_n$ isomorphic to $(\Z/2\Z)^{n-1}$ and $\Z^n$ corresponding to $\langle \gam_0,\gam_1^2,\dots,\gam_{n-1}^2 \rangle$. After passing to classifying spaces and identifying $B\Z^n\simeq \mathbb{T}^n$ and $B\G_n\simeq \msf{LS}_n=\R^n/\G_n$, we get a fiber sequence
\beq
\mathbb{T}^n\to \msf{LS}_n\to B\Phi_n.
\eeq

To compute the characteristic classes we use that the classifying map of the tangent bundle  $\msf{LS}_n\to B\mathrm{O}(n)$ factors as 
\beq
\msf{LS}_n\xto{Bh} B\Phi_n\to B\mathrm{O}(n),
\eeq
where the last map is induced by the inclusion $\Phi_n\subset \mathrm{O}(n-1)\times \mathrm{O}(1)\subset \mathrm{O}(n)$ as the subgroup of diagonal orthogonal matrices of size $n-1$.

As $B\Phi_n$ is homotopy equivalent to an $(n-1)$-fold product of infinite-dimensional real projective spaces, its cohomology ring with $\F_2$ coefficients is isomorphic to the polynomial ring
$\F_2[x_1,\ldots, x_{n-1}]$, in $n-1$ variables $x_i\in H^1$ of degree $1$. By a theorem of Borel--Hirzebruch \cite{borel-hirzebruch}, the $j$-th Stiefel--Whitney class of $\msf{LS}_n$ can be computed in terms of the $x_i's$ by the formula
\begin{equation}\label{eq:formulaw_k}
w_j(\msf{LS}_n)=(Bh)^*\sigma_j(x_1,\ldots,x_{n-1})
\end{equation}
where $\sigma_j(x_1,\ldots,x_{n-1})$ is the $j$-th elementary symmetric polynomial in $x_1,\ldots, x_{n-1}$.

Consider the ideal
\begin{equation}\label{eq:LSideal}
    \mathrm{I}(n)=(x_1^2+x_1x_2,\dots,x_{n-2}^2+x_{n-2}x_{n-1},x_{n-1}^2)
\end{equation}
in $\F_2[x_1,\dots,x_{n-1}]$. The following is one of the main technical results in Lee--Szczarba's paper.
\begin{proposition}[{\cite[Proposition~2.1]{lee-szczarba}}]\label{prop:ker-im}
The elements $x_{i_1}x_{i_2}\cdots x_{i_r}$,
where $1\leq i_1<\cdots<i_r<n$, form a basis
for the image $\mathrm{Im} (Bh)^*$. Moreover there is an isomorphism
\beq
\F_2[x_1,\ldots, x_{n-1}]/\mathrm{I}(n)\cong
\mathrm{Im} (Bh)^*\subset H^*(\msf{LS}_n;\F_2).
\eeq
\end{proposition}
Based on this proposition, they deduce Theorem \ref{thm:LS}.

%%%%%%%%%%%%%%%%%%%%%%%%%%%%%%%%%%%%%%%%%%%%%%%%%%%%%%%%%%%%%%%

\subsection{Refined characteristic classes computations}\label{subsec:refinedcharclasses}
In this subsection we give a precise description of the pairs $(k,n)$ for which $w_{k}^2(\msf{LS}_n)\neq 0$, improving the range given in Theorem \ref{thm:LS} of pairs $(k,n)$ for which $p_k(\msf{LS}_n)\neq 0$.

Given $k\geq 1$, let $s(k)$ be the least number $s$ such that there exist integers $r_1,\dots,r_s\geq 0$ with 
\begin{equation}\label{eq:def.s(k)}
k=(2^{r_1}-1)+(2^{r_2}-1)+\cdots+(2^{r_s}-1).
\end{equation}

The main result of this subsection is the following.

\begin{proposition}\label{prop:w_k^2(L_n)}
    Given $1\leq k \leq n$, we have \[w_k(\msf{LS}_n)^2\neq 0 \text{ in } H^{2k}(\msf{LS}_n;\bbF_2)\]
    if and only if $n \geq 2k+s(k)$.        
\end{proposition}

Before proving this proposition, we deduce some immediate corollaries. First, from the fact that $s(2k)\geq 2$ for each $k\geq 1$ we obtain the following bound.
\begin{corollary}\label{cor:4k+1}
   For each $k\geq 1$ we have $w_{2k}(\msf{LS}_{4k+1})^2=0$.
\end{corollary}

When $k=2^r+2^s-1$ for some non-negative integers $r,s$, we have $s(2k)=2$, and hence the bound above is sharp for infinitely many dimensions. 

\begin{corollary}\label{cor:4k+2}
If $k=2^r+2^s-1$ for some non-negative integers $r,s$, then
$w_{2k}(\msf{LS}_{4k+2})^2\neq 0$. 
\end{corollary}
From the bound $s(2q)\leq 2q$ we also deduce the following, which for $n=4$ recovers Theorem \ref{thm:LS}.

\begin{corollary}\label{cor:6k-2}
 For all $1\leq q \leq k$ and $n\geq 0$ we have $w_{2q}(\msf{LS}_{6k+n})^2\neq 0$.    
\end{corollary}

\begin{remark}\label{rmk:k=1 LS}
    Note that $w_{2}(\msf{LS}_4)^2=0$ since $4<6=2\cdot 2+s(2)$, so that non-triviality of $p_1(\msf{LS}_4)$ cannot be detected from the Stiefel--Whitney classes. 
\end{remark}

Now we proceed with the proof of Proposition \ref{prop:w_k^2(L_n)}.
For a tuple $(k,n)$ of integers with $1\leq k \leq n$, let $\calI_k(n-1)$ be the set of $k$-tuples $(i_1,\dots,i_k)$ of integers with $1\leq i_1<i_2<\dots <i_k\leq n-1$. We call $k$ the \emph{length} of $I$. 
We note that $\calI_k(n-s)\leq \calI_k(n-1)$ for all $k\geq 2$ and $s\geq 1$. On $\calI_k(n-1)$ we consider the lexicographic order $\preceq$: $(i_1,\dots,i_k)\preceq (j_i,\dots,j_k)$ if $i_{r}\leq j_{r}$ for the minimum $r$ such that $i_r \neq j_r$. 

Given a tuple $I=(i_1,\dots,i_k)\in \calI_k(n-1)$, we denote $x_I=x_{i_1}\cdots x_{i_k}\in \F_2[x_1,\ldots, x_{n-1}]/\mathrm{I}(n)$, where $\mathrm{I}(n)$ is the ideal from \eqref{eq:LSideal}. In $\F_2[x_1,\ldots, x_{n-1}]/\mathrm{I}(n)$ we note the identities \begin{equation}\label{eq:identix_ix_i+1}
x_{j}^2x_{j+1}^2=x_{j}x_{j+1}^3=x_jx_{j+1}x_{j+2}^2=x_j^2x_{j+2}^2
\end{equation}
for all $j\geq 1$.

Two tuples $I,J\in \calI_k(n-1)$ are \emph{equivalent} if $x^2_I=x_J^2$ in $\F_2[x_1,\ldots, x_{n-1}]/\mathrm{I}(n)$, and we let $[I]$ denote the equivalence class of $I$ under this relation. We also let $\calI_k^{max}(n-1)$ be the set of tuples such that \begin{itemize}
\item $x_I^2\neq 0$ in $\F_2[x_1,\ldots, x_{n-1}]/\mathrm{I}(n)$; and,
\item for any $J\in [I]$ we have $J \preceq I$.
\end{itemize}
The next lemma describes the tuples in $\calI_k^{max}(n-1)$.

\begin{lemma}\label{lem:descriptionImax}
  For $I=(i_1,\dots,i_k)\in \calI_k(n-1)$, we have that $I\in \calI_k^{max}(n-1)$ if and only if $i_k \leq n-2$ and $i_{r+1}-i_r\geq 2$ for all $r=1,\dots,k-1$. 
\end{lemma}

\begin{proof}
    If $I=(i_1,\dots,i_k)$ is as in the statement of the lemma and $x_I^2\neq 0$ in $\F_2[x_1,\ldots, x_{n-1}]/\mathrm{I}(n)$, the relation $x_{n-1}^2=0$ in $\mathrm{I}(n)$ implies that $i_k\leq n-2$. In this case we have
    \[x_I^2=(x_{i_1}x_{i_1+1})(x_{i_2}x_{i_2+1})\cdots (x_{i_k}x_{i_k+1}). \]
    If $r$ is the largest number such that $i_{r+1}=i_r+1$, from the identity \eqref{eq:identix_ix_i+1} we find $J=(i_1,\dots,i_r,i_{r+1}+1,i_{r+2},\dots,i_k)\in [I]$ and $I \prec J$. Therefore, $I\in \calI_{k}^{max}(n-2)$ forces $i_{r+1}-i_r\geq 2$ for all $j$. 
    
    On the other hand, if $I=(i_1,\dots,i_k)\in \calI_k(n-1)$ satisfies $i_k\leq n-2$ and $i_{r+1}-i_r\geq 2$ for al $r$, then Proposition \ref{prop:ker-im} implies that $x_I^2\neq 0$ and no $J$ with $I\prec J$ is equivalent to $I$. That is, $I\in \calI_k^{max}(n-1)$. 
    \end{proof}

\begin{remark}\label{rmk:2k-1<2n-2}
    From the lemma above, $\calI_{k}^{max}(n-1)\neq \emptyset$ implies $2k-1\leq n-2$.
\end{remark}

Consider now $m\geq 1$ with $2m-1\leq n-2$ and let $\lam_m$ be the cardinality of the equivalence class $[(1,2,\dots,m)]=[(1,3,5,\dots,2m-1)]$ in $\calI_m(n-1)$. Note that $(1,3,\dots,2m-1)$ belongs to $\calI_{m}^{max}(n-1)$.

\begin{lemma}\label{lem:catalan}
    For each $m\geq 1$ with $2m-1\leq n-2$, $\lam_m$ equals the $m$-th Catalan number $$C_m=\binom{2m}{m}-\binom{2m}{m+1}.$$
     In particular, $\lam_m$ is odd if and only if $m+1$ is a power of two.  
\end{lemma}

\begin{proof}
    A \emph{Dyck path} of length $2m$ is a lattice path in $\Z^2$ from $(0,0)$ to $(m,m)$ with steps $(1,0)$ and $(0,1)$, that never rises above the diagonal $y=x$. It is known that $C_m$ equals the number of Dyck paths of length $2m$ \cite[Section~1.5]{stanley}. 

    To relate each tuple in $[(1,3,\dots,2m-1)]$ with a Dyck path of length $2m$, note that $I=(a_1,\dots,a_m)$ belongs to $[(1,3,\dots,2m-1)]$ if and only if 
    \begin{itemize}
        \item[i)] $a_1<a_2<\cdots <a_m$; and,
        \item[ii)] $j\leq a_j\leq 2j-1$ for each $j=1,\dots,m$. 
    \end{itemize}
    This follows from the identities \eqref{eq:identix_ix_i+1}. From this data we construct the Dyck path $p_I$ of length $2m$ such that:
    \begin{itemize}
        \item it contains the points $(j,a_j-j)$ for $j=1,\dots,m$; and,
        \item it does not contain the points $(j,a_{j}-j-1)$ for $j=1,\dots,m$. 
    \end{itemize}
    Condition i) guarantees that no point in $p_I$ lies above the diagonal $y=x$, and condition ii) guarantees that the steps for $p_I$ are $(0,1)$ and $(1,0)$. 

    Conversely, if $p$ is a Dyck path of length $2m$, we construct a tuple $I\in [(1,3,5,\dots,2m-1)]$ as follows: given $1\leq j\leq m$, let $b_j$ be the minimum number such that $p$ contains the point $(j,b_j)$. Then we set $I=(b_1+1,b_2+2\dots,b_m+m)$. It is easy to see that $I$ belongs to $[(1,3,\dots,2m-1)]$ and that this construction is inverse to the assignment $I \mapsto p_I$. Hence $\# [(1,3,\dots,2m-1)]=C_m$. 

    We are left to show that $C_m=\binom{2m}{m}-\binom{2m}{m+1}$ is odd if and only if $m+1$ is a power of two. Since $\binom{2m}{m}$ is always even, the parity of $C_m$ equals the parity of $\binom{2m}{m+1}$. By Lucas's theorem (see e.g.~\cite{hu-sun}), the parity of this coefficient depends on the behavior of the base-2 expansion of $m$. Namely, suppose 
    $$m+1=a_0 2^0+a_12^1+\cdots +a_r2^r \ \ \text{ with }a_i\in \{0,1\}
    $$
    and
    $$2m=b_0 2^0+b_12^1+\cdots +b_r2^r \ \ \text{ with }b_i\in \{0,1\} \text{ and }b_r=1.
    $$
    Then $\binom{2m}{m+1}$ is even if and only if $a_i>b_i$ for some $i=0,\dots,r$ (note that $a_r$ is allowed to be $0$). 

    If $m+1$ is a power of 2, then $a_0=a_1=\dots=a_{r-2}=0$ and $a_r=1$, whereas $b_0=0$ and $a_1=\dots =a_r=1$. Then $b_i\geq a_i$ for all $i$ and $\binom{2m}{m+1}$ is odd. If $m+1$ is not a power of 2, let $i<j$ be the first indices such that $a_i=a_j=1$. Then we can check that $b_j=0<a_j$, so that $\binom{2m}{m+1}$ is even. This analysis completes the proof of the lemma.
\end{proof}

Given $I=(i_1,\dots,i_k)\in \calI_k^{max}(n-1)$, let $\beta_1<\dots <\beta_{s-1}$ be the set of indices such that $i_{\beta_{j}+1}-i_{\beta_j}\geq 3$ for $j=1,\dots,s-1$. For completeness we set $\beta_0=0$ and $\beta_s=i_k$.
Then for each $j=1,\dots,s$, from Lemma \ref{lem:descriptionImax} we have that $$I_j=(i_{\beta_j+1},i_{\beta_j+2},\dots,i_{\beta_{j+1}})=(i_{\beta_j+1},i_{\beta_j+1}+2,\dots, i_{\beta_j+1}+2(\beta_j-\beta_{j-1}-1)).$$

In this case, we say that $I=I_1\cdots I_s$ is the \emph{canonical decomposition} of $I$. Note that each $I_j$ is a tuple of length $m_j=\beta_j-\beta_{j-1}$. For this description of $I$ we set $\lam_I:=C_{m_1}\cdots C_{m_s}$. 

For the next lemma, we apply Proposition \ref{prop:ker-im} to see $\F_2[x_1,\ldots, x_{n-1}]/\mathrm{I}(n)$ as a subring of $H^\ast(\msf{LS}_n;\bbF_2)$. Under this identification, by \eqref{eq:formulaw_k} we have
\begin{equation*}
 w_k(\msf{LS}_n)=\sum_{I\in \calI_k(n-1)}{x_I}   
\end{equation*}
for all $1\leq k \leq n$.

\begin{lemma}\label{lem:w_k^2}
    For all $1\leq k\leq n$, in $H^{2k}(\msf{LS}_n,\bbF_2)$ we have
    \[w_k(\msf{LS}_n)^2=\sum_{I\in \calI_k^{max}(n-1)}\lam_I x_I^2.\]
    Here the right hand is taken to be zero if $\calI_k^{max}(n-1)$ is empty.
\end{lemma}

\begin{proof}
By Remark \ref{rmk:2k-1<2n-2}, if $\calI_k^{max}(n-1)$ is empty then $2k-1>n-2$, and hence every tuple in $I_k(n-1)$ is equivalent to a tuple $(i_1,\dots,i_k)$ with $i_k=n-1$. Then $w_k^2(\msf{LS}_n)=0$. 

From now on we suppose that $2k-1\leq n-2$. Then by the construction of $\calI_k^{max}(n-1)$ we have
\[w_k(\msf{LS}_n)^2=\sum_{I\in \calI_k^{max}(n-1)}\#[I]\cdot x_I^2.\]
Therefore, we are left to show that $\#[I]=\lam_I$ for each  $I\in \calI_k^{max}(n-1)$. In order to do this, for $I=(i_1,\dots,i_k)\in \calI_k^{max}(n-1)$ consider its canonical decomposition $I=I_1\cdots I_s$. Let $0=\beta_0<\beta_1<\cdots <\beta_{s-1}<\beta_s=i_k$ be such that for $j=1,\dots,s$ we have $$I_j=(i_{\beta_j+1},i_{\beta_j+1}+2,\dots, i_{\beta_j+1}+2(m_j-1))$$
with $m_j=\beta_j-\beta_{j-1}$. 
Since $i_{\beta_j}-i_{\beta_{j-1}}\geq 3$ for $j=1,\dots,s-1$, it is not hard to see that $[I]$ is the set of all the tuples $J'$ that can be written as a concatenation $J'=J'_1\cdots J'_s$ with each $J'_j\in [I_j]\subset \calI_{m_j}(n-1)$. Then Lemma \ref{lem:catalan} implies that $\#[I]$ equals $\lam_I$, as desired.
\end{proof}

\begin{proof}[Proof of Proposition \ref{prop:w_k^2(L_n)}]
      By Lemma \ref{lem:w_k^2}, we have that $w_k^2(\msf{LS}_n)\neq 0$ if and only if there exists $I=\calI_k^{max}(n-1)$ with $\lam_I$ odd. If such an $I$ exists and has canonical decomposition $I=I_1\cdots I_s$ with each $I_j$ a tuple of length $m_j$, we require:
    \begin{itemize}
        \item[i)] $m_1+\cdots+ m_s=k;$
        \item[ii)] each $C_{m_1},\dots,C_{m_s}$ to be odd; and,
        \item[iii)] $(2m_1-1)+2+(2m_2-1)+2+\cdots +2+(2m_s-1)=2(m_1+\cdots +m_s)+s-2\leq n-2$.
    \end{itemize}
    Combining $i)$ and $iii)$ yields $n\geq 2k+s$, and $ii)$ together with Lemma \ref{lem:catalan} gives us integers $r_1,\dots,r_s\geq 0$ with $m_j=2^{r_j}-1$ for each $j=1,\dots,s$. Combining this with $i)$ gives us $s\geq s(k)$, and hence $n\geq 2k+s(k)$.

    For $s=s(k)$ and a description of $k$ as in \eqref{eq:def.s(k)}, it is not hard to produce a tuple $I=\calI_k^{max}(2k+s(k)-1)$ with $\lam_I$ odd, and hence $w_k(\msf{LS}_{2k+s(k)})^2\neq 0$ by Lemma \ref{lem:w_k^2}. Such a tuple also belongs to $\calI_k^{max}(2k+s-1)$ for each $s\geq s(k)$, concluding the proof of the proposition. 
\end{proof}
%%%%%%%%%%%%%%%%%%%%%%%%%%%%%%%%%%%%%%%%%%%%%%%%%%%%%%%%%

\subsection{Orientable double covers of the manifolds $\msf{LS}_n$}\label{subsec:LSmanifolds}
In this subsection we study the orientable double covers of the Lee--Szczarba manifolds, denoted by $\what{\msf{LS}}_n$. This amounts to pulling back the extension \eqref{eq:flat} along the kernel $\widehat{\Phi}_n$ of the determinant map $\mathrm{det}:\Phi_n \subset \mathrm{O}(n)\to\Z/2\Z$. This yields a map of extensions
\begin{equation}\label{eq:pullback}
\xymatrix{
1\ar[r]& \Z^n\ar[r]\ar[d]& \widehat{\G}_n\ar[r]^{\widehat{h}}\ar[d] & \widehat{\Phi}_n\ar[r]\ar[d]^{\iota} & 1\\
1\ar[r]& \Z^n\ar[r]& \G_n\ar[r]^{h} & \Phi_n\ar[r] & 1,
}
\end{equation}
where $\what\G_n$ is the fundamental group of $\what{\msf{LS}}_n$.
With the identifications $\widehat{\Phi}_n\cong (\Z/2\Z)^{n-2}$ and $\Phi_n\cong (\Z/2\Z)^{n-1}$, the map $\iota$ is given by
\beq
\iota(x_1,\ldots,x_{n-2})=(x_1,\ldots,x_{n-2}, x_1+\cdots +x_{n-2}).
\eeq
Therefore the induced map in cohomology 
\beq
\iota^*:H^*(\Phi_n;\F_2)\cong\F_2[x_1,\ldots,x_{n-1}]\to H^*(\widehat{\Phi}_n;\F_2)\cong\F_2[u_1,\ldots,u_{n-2}]
\eeq
is given by 
\beq
\iota^*(x_i)=
\begin{cases} 
u_i & \text{if}\ \  1\leq i\leq n-2.\\
u_1+\cdots+ u_{n-2} &\text{if}\ \  i=n-1.
\end{cases}
\eeq

As in the case of $\msf{LS}_n$, we consider the ideal
\begin{equation}\label{eq:LSidealor}
 \widehat{\mathrm{I}}(n)=(\mathrm{I}(n),x_1+\cdots+x_{n-1})   
\end{equation}
in $\F_2[x_1,\dots,x_{n-1}]$, and set $\widehat{x}_j:=(\iota\circ B\widehat{h})^*(x_j)$. The next result is the orientable analog of Proposition \ref{prop:ker-im}.
\begin{proposition}\label{prop:ker-im-or}
The elements $\widehat{x}_{i_1}\widehat{x}_{i_2}\cdots \widehat{x}_{i_r}$,
where $1\leq i_1<\cdots<i_r<n-1$, form a basis
for the image $\mathrm{Im} (B\widehat{h})^*$. Moreover, there is an isomorphism
\beq
\F_2[x_1,\ldots, x_{n-1}]/\widehat{\mathrm{I}}(n) \cong
\mathrm{Im} (B\widehat{h})^*\subset H^*(\widehat{\msf{LS}}_n;\F_2).
\eeq
\end{proposition}
\begin{proof}
Let $\{\widehat{E}_r^{p,q},\widehat{d}_r\}$ be the $\F_2$-cohomology Serre spectral sequence for the top fibration in \eqref{eq:pullback} (see for example \cite[Ch. 9, Sec. 4, Theorem 6]{spanier}).

Then the differential
\beq
\widehat{d}_2:\widehat{E}_2^{0,1}\to\widehat{E}_2^{2,0}
\eeq
can be computed from the map of fibrations \eqref{eq:pullback}, the naturality of the Serre spectral sequence, and \cite[Lemma~3.1]{lee-szczarba}.  This calculation yields $\widehat{d}_2(y_i)=\iota^*d_2(y_i)$, and therefore
\beq
\mathrm{Im}(\widehat{d}_2)=\iota^*\mathrm{Im}(d_2)=\iota^*\mathrm{I(n)}\subset\ker(B\widehat{h}^*).
\eeq
 
The result will follow if we show the reverse inclusion, because then
\beq
\mathrm{Im} (B\widehat{h})^*\cong\F_2[u_1,\ldots, u_{n-2}]/\iota^*\mathrm{I}(n)\cong\F_2[x_1,\ldots, x_{n-1}]/\widehat{\mathrm{I}}(n).
\eeq
To show the reverse inclusion, we follow  \cite[p.6]{lee-szczarba} very closely. First note that by construction the holonomy map $\widehat{h}:\widehat{\G}_n\to\widehat{\Phi}_n$ factors through $\widehat{\Phi}_{n-1}\times\Z$.

This gives a commutative diagram
\beq
\xymatrix{
H^*(B\widehat{\Phi}_{n-1}\times \bbT^1)\ar[r]^-{f^*}&H^*(\widehat{\msf{LS}}_n)\\
&H^*(B\widehat{\Phi}_{n})\ar[ul]^{g^*}\ar[u]_{B\widehat{h}^*}
}
\eeq
in which the diagonal map $g^*$ is a surjection with kernel generated by $\langle u_{n-2}^2\rangle$ and that maps $\iota^*\mathrm{I}(n)$ onto the ideal $\widehat{\mathrm{J}}(n)$ generated by
\beq
u_{n-3}^2+u_{n-3}y,\qquad  u_i^2+u_iu_{i+1},\ 1\leq i<n-3,\qquad u_1^2+\cdots u_{n-3}^2.
\eeq
Therefore, the kernel of $B\widehat{h}^*$ is $\iota^*\mathrm{I}(n)$ if and only if the kernel of $f^*$ is $\widehat{\mathrm{J}}(n)$.
This last assertion follows by induction on $n$, using the following  commutative diagram of exact sequences 
\beq
\xymatrix{
0 \ar[r] & \ker(B\widehat{h}^*) \ar[r] & \ker(f^*) \ar[r] & \ker(B\widehat{h}^*) &  \\
0 \ar[r] & \widehat{\mathrm{I}}(n-1) \ar[r] \ar[u] & \widehat{\mathrm{J}}(n) \ar[r] \ar[u] & \widehat{\mathrm{I}}(n-1) \ar[r] \ar[u] & 0.
}
\eeq
This diagram arises from the map of (split) Wang sequences associated to the fibrations over the circle 
\beq
\xymatrix{
\widehat{\msf{LS}}_{n-1} \ar[r] \ar[d] & \widehat{\msf{LS}}_n \ar[r] \ar[d]_{f} & \mathbb{T}^1 \ar[d] \\
B\widehat{\Phi}_{n-1} \ar[r] & B\widehat{\Phi}_{n-1} \times \mathbb{T}^1 \ar[r] & \mathbb{T}^1.
}
\eeq
The base case of the induction
is the equality $\iota^*\mathrm{I}(2)=\ker(B\widehat{h}^*)$, which holds trivially as both terms are zero.
\end{proof}
Proposition \ref{prop:ker-im-or} enables us to compute some characteristic classes of the orientable flat manifolds $\what{\msf{LS}}_n$. Indeed, using this result, \eqref{eq:formulaw_k}, and the map $\iota^*$, we can describe the $j$-th Stiefel--Whitney of $\widehat{\msf{LS}}_n$ as the element
\beq
\sigma_j(x_1,x_2,\ldots,x_{n-2},x_{n-1}) \in \F_2[x_1,\dots,x_{n-1}]/\wh{\mathrm{I}}(n),
\eeq
where $\si_j$ is the $j$-th symmetric polynomial in $n-1$ variables and $\wh{\mathrm{I}}(n)$ is the ideal from \eqref{eq:LSidealor}.

From this description, we can conclude (for example, using SAGE) the non-triviality of the following characteristic classes: $w_2(\wh{\msf{LS}}_5), w_3(\wh{\msf{LS}}_6), w_2(\wh{\msf{LS}}_8)^2, w_5(\wh{\msf{LS}}_{10})$, and $w_4(\wh{\msf{LS}}_{16})^2$.
This implies the following.
\begin{corollary}\label{cor:char-classes-oriented} For the orientable double covers $\what{\msf{LS}}_n$ the following holds.
\begin{enumerate}
    \item If $n\geq 5$ then $w_2(\wh{\msf{LS}}_n)\neq 0$, and hence $\wh{\msf{LS}}_n$ is not spin.
    \item If $n\geq 6$ then $w_3(\wh{\msf{LS}}_n)\neq 0$, and hence $\wh{\msf{LS}}_n$ is not spin$^c$.
    \item If $n\geq 8$ then $w_{2}(\what{\msf{LS}}_n)^2\neq 0$, and hence $p_1(\what{\msf{LS}}_n)\neq 0$.
    \item If $n\geq 10$ then $w_5(\wh{\msf{LS}}_{n})\neq 0$, and hence $\wh{\msf{LS}}_{n}$ is not spin$^h$.
    \item If $n\geq 16$ then $w_{4}(\what{\msf{LS}}_n)^2\neq 0$, and hence $p_2(\what{\msf{LS}}_n)\neq 0$.
\end{enumerate}
\end{corollary}

\begin{proof} We only prove item (1) as the others follow by the exact same argument using the non-triviality of the characteristic classes above. We prove that $w_2(\wh{\msf{LS}}_n)\neq 0$ for $n\geq 5$ by induction on $n$, noting that there is a finite covering map $\what{f}:\wh{\msf{LS}}_{n}\times \bbT^1\ra \wh{\msf{LS}}_{n+1}$. Indeed, this map comes by lifting the covering map $f:\msf{LS}_{n}\times \bbT^1 \ra \msf{LS}_{n+1}$ constructed as follows. By \cite[Proposition~1.3]{lee-szczarba}, the manifold $\msf{LS}_{n+1}$ is diffeomorphic to the mapping torus of the diffeomorphism $g:\msf{LS}_{n}\to \msf{LS}_{n}$ which comes from the reflection in $\R^{n-1}$
\beq
(x_1,\ldots,x_{n})\to (x_1,\ldots,x_{n-1},-x_{n}).
\eeq
Since $g$ has order $2$, it corresponds to a $2$-sheeted cover $f:\msf{LS}_{n}\times \bbT^1\to \msf{LS}_{n+1}$.

Note that $\wh{\msf{LS}}_{n}$ sits in $\wh{\msf{LS}}_{n}\times \bbT^1$ as a codimension-$1$ submanifold with trivial normal bundle. Therefore, if $w_2(\wh{\msf{LS}}_n)\neq 0$ then $w_2(\wh{\msf{LS}}_n\times \bbT^1)\neq 0$. Hence
\beq 
\wh{f}^*(w_2(\wh{\msf{LS}}_{n+1}))=w_2(\wh{\msf{LS}}_n\times \bbT^1)\neq 0,
\eeq
implying $w_2(\wh{\msf{LS}}_{n+1})\neq 0$. 
\end{proof}
\begin{conjecture}\label{conj:nonzero-pi-orientable}
For the orientable flat manifolds $\widehat{\msf{LS}}_n$, the mod $2$ reduction of the $i$-th Pontryagin class $p_i(\widehat{\msf{LS}}_n)$ is non-zero, provided $n\geq 8i$. 
\end{conjecture}

\section{Applications}\label{sec:applications}

In this section we give some applications of our results. In particular, we deduce  Corollary \ref{cor:exotic-smoothings}, Theorem \ref{thm:Pontryagin}, and Theorem \ref{thmalpha:manymanywithnontrivial}.
\subsection{Exotic negatively curved manifolds}

\begin{proof}[Proof of Corollary \ref{cor:exotic-smoothings}.]
Fix $n\geq 9$. Let $Q$ and $R$ be two closed orientable smooth $n$-manifolds with the same underlying $\PL$ manifold $\msf{P}$ satisfying:
\begin{enumerate}[(i)]
\item for some $j>0$, their $j$-th Pontryagin classes satisfy  $p_j(Q)=0$ and $p_j(R)\neq 0$; and,
\item some \emph{rational} Pontryagin class of $Q$ (and hence of $R$, by the topological invariance of rational Pontryagin classes) is non-trivial.
\end{enumerate}
The existence of closed $9$-manifolds that satisfy (i) for $j=2$ follows, for instance, from the proof of Theorem 3.1 in
\cite{crowley-p2}
(see also \cite[Section 4.4]{kreck-lueck} for examples in dimensions $\geq 21$). Connecting sum these manifolds with $\C P^2\times S^5$ yields $9$-manifolds $Q_0$ and $R_0$ which now satisfy (i) and (ii). To get examples $Q$ and $R$ in dimensions $\geq 9$ it suffices to take products of $Q_0$ and $R_0$ with spheres.

Now choose smooth triangulations $\msf{Q}$ and $\msf{R}$ of $Q$ and $R$ compatible with the common $\PL$ structure $\msf{P}$. Since $Q$ and $R$ have the same underlying $\PL$ manifold, we may assume (possibly after subdividing $\msf{Q}$ and $\msf{R}$) that there is a simplicial isomorphism $\msf{Q}\to \msf{R}$.

Let
\beq
M:=\calH_X(\calG(\msf{Q}))\qquad\text{and}\qquad N:=\calH_X(\calG(\msf{R})),
\eeq
where the hyperbolizing piece $X$ is chosen as in Theorem \ref{thm:main-stably-tan} and sufficiently large so that Ontaneda's Riemannian hyperbolization applies. Then $M$ and $N$ are closed orientable Riemannian manifolds all of whose sectional curvatures lie in the interval $[-1-\epsilon,-1]$. Moreover, by the functoriality of $\calG$ and $\calH_X$ \cite[p.349]{charney-davis}, the simplicial isomorphism $\msf{Q}\to \msf{R}$ induces a homeomorphism between $M$ and $N$.

By Corollary \ref{cor:Riemannian-hyp}, we have $p_j(M)=0$ and $p_j(N)\neq 0$, so $M$ and $N$ are not diffeomorphic. In consequence, this proves parts (1) and (2).

It remains to prove part (3). Since strict hyperbolization preserves rational Pontryagin classes, $M$ has a non-zero rational Pontryagin class by (ii). Then $M$ is not homeomorphic to a real hyperbolic manifold  or to a Gromov--Thurston branched cover of a real hyperbolic manifold because these have rational vanishing Pontryagin classes: the former by Chern--Weil theory or \cite{sullivan-covers, okun}, the latter by \cite{ardanza}.

Next, $\pi_1(M)$ does not have Kazhdan's property $(T)$ \cite[Theorem 1.2]{lafont-ruffoni}, whereas cocompact lattices in $\Sp(m,1)$ and in $F_{4(-20)}$ do \cite{kostant}. Hence $M$ is not homeomorphic to a quaternionic or octonionic hyperbolic manifold.
Finally, $\pi_1(M)$ is not a K\"ahler group \cite[Theorem 1.8]{belegradek}. Therefore $M$ is not homeomorphic to a complex hyperbolic manifold, nor to any of the Mostow--Siu, Deraux, or Stover--Toledo examples, since all of these are compact K\"ahler manifolds. 
\end{proof}
\subsection{Hyperbolic manifolds with non-trivial characteristic classes}\label{subsec:nontrivialSW}

\begin{proof}[Proof of Theorem \ref{thm:Pontryagin}]
Let $n\geq 2$ and $\K\neq \Q$ be a totally positive number field. Let $M$ be a Charney--Davis $n$-manifold obtained from Theorem \ref{thmalpha:flexibleCD}, and use item (2) of that theorem together with \cite{sullivan-covers} to further assume that $M$ is stably parallelizable. Let $X$ be the hyperbolizing piece obtained from $M$.

We first deal with item (1), which is the template for all the other items. Let $\msf{C}_n$ be a flat foldable cube complex homeomorphic to the Lee--Szczarba manifold $\msf{LS}_n$, which exists by Proposition \ref{thmA:diagonalfoldable}. Let $N=\calH_{X}(\msf{C}_n)$ be the hyperbolized complex. Then $N$ is hyperbolic and commensurable with $M$ by Proposition \ref{thm:flat->hyperbolic}.
By Corollary \ref{cor:tangential-hyp} and Theorem \ref{thm:LS}, we have that $w_i(N)\neq 0$ if $1\leq i\leq n-1$, and that $p_i(N)\neq 0$ if $n=6k+4$ with $k\geq 1$ and $i\leq k$. By ranging over all possible (arithmetic and non-arithmetic) commensurability classes of $M$ given by Theorem \ref{thmalpha:flexibleCD}, we produce infinitely many commensurability classes of manifolds $N$. This proves item (1).

For item (2) we use the same argument. In this case, we hyperbolize flat foldable cube complexes homeomorphic to the orientable Lee--Szczarba manifolds $\wh{\msf{LS}}_n$, which are the appropriate input by Corollary \ref{cor:char-classes-oriented}~(3) \& (5).

For item (3) and $n=4$, we use as input either of the two orientable, non-spin flat $4$-manifolds of diagonal type \cite{putrycz-szczepanski}. Then we hyperbolize a flat foldable cubulation $\msf{C}_4$ of that manifold. For $n\geq 5$, we hyperbolize the (orientable and non-spin) flat foldable cube complexes $\msf{C}_n=\msf{C}_4 \times \bbT^{n-4}$ (or we can also hyperbolize the orientable Lee-Szczarba manifolds $\what{\msf{LS}}_n$ and apply Corollary \ref{cor:char-classes-oriented}~(1)). 

For item (4) and $n=5$ we hyperbolize a \emph{generalized Hantzsche-Wendt} $5$-manifold \cite{rossetti-szczepanski}. These are orientable flat manifolds with holonomy $(\Z/2\Z)^{n-1}$ and with non-vanishing third Stiefel--Whitney class \cite{LPS}.
For $n>5$ we hyperbolize the product of either of these manifolds with a torus, or alternatively
we hyperbolize the orientable Lee--Szczarba manifold $\what{\msf{LS}}_n$, which has $w_3(M)\neq 0$ by Corollary \ref{cor:char-classes-oriented}~(2). 

Finally, for item (5), we hyperbolize $\what{\msf{LS}}_n$, for $n\geq 10$, which has $w_5(\what{\msf{LS}}_n)\neq 0$ by Corollary \ref{cor:char-classes-oriented}~(4).
\end{proof}
It would be interesting to know whether there is a non-spin$^{c}$ foldable flat orientable manifold $F$ with $w_3(F)=0$. Applying stably tangential strict hyperbolization to such manifold would yield hyperbolic manifolds with the same property, answering  \cite[Question~1.6]{chen.spinc}. Amusingly,
applying Corollary \ref{cor:Riemannian-hyp}, we can find such examples with variable negative curvature.
\begin{corollary}\label{lem:teichnernegcurved}
    For any $n\geq 6$ and $\ep>0$ there exists a closed, orientable, smooth Riemannian $n$-manifold $M$ with sectional curvatures on $[-1-\ep,-1]$ and such that $M$ is not spin$^c$ and $w_3(M)=0$.
\end{corollary}

\begin{proof}
It follows from the work of Crowley--Grant \cite[Theorem~1.3, Proposition~5.9]{crowley-grant} that certain $6$-manifolds $N$ arising as sphere bundles of the vector bundles constructed by Teichner \cite[Lemma 2]{teichner} are orientable, non-spin$^c$ and satisfy $w_3(N)=0$. 
Applying Corollary \ref{cor:Riemannian-hyp} to $N$ we obtain the result in dimension $n=6$. 

To get examples in all dimensions greater than 6, it suffices to hyperbolize (triangulations of) the product of a Teichner manifold $N$ with a sphere.
\end{proof}

%%%%%%%%%%%%%%%%%%%%%%%%%%%%%%%%

\subsection{Charney--Davis manifolds in a given commensurability class}\label{subsec:CDingivencclas}

We proceed with the proof of Theorem \ref{thmalpha:manymanywithnontrivial}, which follows from the next result. It shows that there are plenty of commensurability classes of closed hyperbolic manifolds, each containing infinitely many distinct manifolds with non-trivial characteristic classes. Recall that two closed smooth manifolds $M,M'$ are tangentially related if there exists a closed smooth manifold $N$ and stably tangential, smooth degree 1 maps $M,M'\ra N$. In particular, a certain characteristic class for $M$ is non-trivial if and only if it is non-trivial for $M'$.

\begin{theorem}\label{thm:manymanywithnontrivial}
Let $M$ be any of the manifolds obtained from either Theorem \ref{thm:Pontryagin}. Then there exist sequences $(M_{1,j})_{j\geq 1},(M_{2,j})_{j\geq 1}$ of closed hyperbolic manifolds in the commensurability class of $M$ with $M_{2,1}=M$, and such that:
\begin{enumerate}
    \item the injectivity radius of $M_{1,j}$ tends to infinity as $j$ tends to infinity; 
    \item each $M_{2,j+1}$ is a non-trivial cover of $M_{2,j}$; and,
    \item each $M_{1,j}$ and $M_{2,j}$ is stable tangentially equivalent to $M$.
\end{enumerate}
\end{theorem}

\begin{proof}
For $M$ as in the statement, we have $M=\calH_X(\msf{C})$ for $X$ a hyperbolizing piece obtained from a Charney--Davis manifold $N$ and $\msf{C}$ a flat foldable cube complex. To construct the manifolds $M_{1,j}$, we apply Theorem \ref{thmalpha:flexibleCD} to find a sequence $N_{j}$ of stable parallelizable Charney--Davis manifolds commensurable to $N$, and so that $\InjRad(N_{j}) \to \infty$ as $j\to \infty$. If $X_{j}$ is the hyperbolizing piece associated to $N_{j}$, then Theorem \ref{thmalpha:flexibleCD}, Corollary \ref{cor:tangential-hyp} and Proposition \ref{thm:flat->hyperbolic} imply that $M_{1,j}:=\calH_{X_{j}}(\msf{C})$ is tangentially related to $M$ and $\InjRad(M_{1,j})\to \infty$ as $j\to \infty$.

To construct the sequence $(M_{2,j})_{j\geq 1}$, we use \cite[Corollary]{epstein-shub} to find an infinite tower $(\msf{C}_{j})_j$ of non-trivial self-coverings of $\msf{C}$. Then the hyperbolic manifolds $M_{2,j}=\calH_{X}(\msf{C}_{j})$ are non-trivial finite coverings of $M$ by Corollary \ref{cor:nonregularcover} and are stably equivalent to $M$.
\end{proof}

%%%%%%%%%%%%%%%%%%%%%%%%%%%%%%%%%%%%%%%%%%%%%%%%%%
\section{Dehn filling}\label{sec.dehnfilling}
This final section is devoted to the proof of Theorem \ref{thm.maingrouptheory} from the introduction. This is the technical heart of our paper and was used to prove Theorem \ref{thmalpha:flexibleCD}. 
We recall its statement here for the reader's convenience.
\begin{theorem}\label{thm:maingrouptheorymiddle}
  Let $\G$ be a hyperbolic cubulable group and consider a finite group $\Phi$ acting on $\G$ by automorphisms. Let $\calQ$ be a finite, $\Phi$-invariant collection of quasiconvex subgroups of $\G$, and let $ \G_0<\G$ be a finite index subgroup. Then there exists a $\Phi$-invariant, finite index normal subgroup $\G'<\G$ such that:
    \begin{enumerate}
        \item $\G'<\G_0$; and,
        \item for all $Q_1,Q_2\in \calQ$ we have
        \[\G' \cap Q_1Q_2 \subset (\G'\cap Q_1)(\G'\cap Q_2). \]
    \end{enumerate}
\end{theorem}

For further references about relatively hyperbolic groups, see e.g. \cite{Farb1998RH,hruska}. We will use the machinery of relatively hyperbolic groups and (group theoretic) Dehn filling, introduced independently by Groves--Manning \cite{groves-manning.dehnfilling} and Osin \cite{osin.peripheral}.

\subsection{Dehn fillings}
Let $\G$ be a group and $\calP$ a family of subgroups of $\G$. We call $(\G,\calP)$ a \emph{group pair}. A choice $N_P<P$ of normal subgroups for each $P\in \calP$ determines a \emph{(Dehn) filling} $(\ov\G,\ov\calP)$ of $(\G,\calP)$, where $\ov\G=\G/K$ for $K$ the normal closure of $\bigcup_{P\in \calP}{N_P}$ and $\ov\calP$ is the collection of images in $\ov\G$ of elements of $\calP$. The groups $\{N_P:P\in \calP\}$ are called \emph{filling kernels}. Sometimes we write $\G(\{N_P\}_P)$ for $\ov\G$. 
When we omit mention of the particular filling kernels, we simply write the filling as $(\G,\calP)\ra (\ov\G,\ov\calP)$.

If $Q<\G$ and $\G \ra \G(\{N_P\}_P)$ is a filling such that $gN_Pg^{-1}<Q$ whenever $g\in \G$ and $Q\cap gPg^{-1}$ is infinite, we say that the filling is a $Q$-\emph{filling}. If $\calQ$ is a family of subgroups of $\G$, a filling is a $\calQ$-\emph{filling} if it is a $Q$-filling for every $Q\in \calQ$. A property $\mathsf{P}$ holds \emph{for all sufficiently long fillings} (resp $Q$-fillings, $\calQ$-fillings) of $(\G,\calP)$ if there is a finite set $S \subset (\bigcup \calP)\bs \{1\}$ such that $\mathsf{P}$ holds whenever $\G \ra \G(\{N_P\}_P)$ is a filling (resp. $Q$-filling, $\calQ$-filling) such that $S\cap (\bigcup_P{N_P})=\emptyset$. 

If $(\G,\calP)$ is a group pair and $\Phi$ is a group acting on $\G$ by automorphisms, we say that $(\G,\calP)$ is $\Phi$-\emph{invariant} if $\phi(P)$ is conjugate to a member of $\calP$ for any $\phi\in \Phi$ and $P\in \calP$. In addition, the filling kernels $\{N_P\}_P$ are $\Phi$-\emph{invariant} if $\phi(N_P)=gN_Qg^{-1}$ whenever $\phi\in \Phi, g\in \G$, $P,Q\in \calP$ and $\phi(P)=gQg^{-1}$. If $(\G,\calP)$ and the filling kernel $\{N_P\}_P$ are $\Phi$-invariant, then the kernel of $\G\ra \G(\{N_P\}_P)$ is $\Phi$-invariant, and hence the action of $\Phi$ on $\G$ descends to an automorphism action on $\ov\G(\{N_P\}_P)$.

\subsection{Relative hyperbolicity}

For background on the different equivalent characterizations of relatively hyperbolic pairs, we refer the reader to \cite{hruska}. For our purposes, and particularly in Section \ref{subsection.doublecoset}, we will rely on the notion of \emph{cusped space}, for which we mostly follow \cite{groves-manning.dehnfilling} and \cite[Appendix A]{agol.haken}. 

Given a group pair $(\G,\calP=\{P_1,\dots,P_m\})$ with $\G$ and each $P_i$ being finitely generated, and an appropriate finite generating subset the $S\subset \G$, the \emph{cusped space} $X=X(\G,\calP,S)$ is constructed as follows. Starting from the Cayley graph $\mathrm{Cay}(\G,S)$, the cusped space is built by attaching \emph{combinatorial horoballs}. Each combinatorial $A$ is a graph with vertex set $gP \times \Z_{\geq 0}$ for $P\in \calP$ and $g\in \G$, so that each $A$ is hyperbolic. Such horoball $A$ is attached to $\mathrm{Cay}(\G,P)$ via the identification $gP=gP\times \{0\}$. A vertex $(g,k)$ in a horoball is said to have \emph{depth} $k$, and the depth 0 vertices of the cusped space are the vertices of the Cayley graph. Any edge of the cusped space connects vertices whose depths differ by at most one. See \cite[Section~3]{groves-manning.dehnfilling} for more details about the construction of the cusped space. The key property is that the cusped space $X$ is hyperbolic if and only if $(\G,\calP)$ is a relatively hyperbolic pair. In that case we say that $\calP$ is a \emph{peripheral structure} on $\G$.

When $\G$ is hyperbolic, the following characterization of relative hyperbolicity was essentially proven in \cite[Theorem~7.11]{bowditch.RH}. Recall that a finite collection $\calP=\{P_1,\dots,P_m\}$ of subgroups of $\G$ (with no two distinct members of $\calP$ being conjugate in $\G$) is \emph{almost malnormal} in $\G$ if for all $1 \leq i,j\leq m$ and $g\in \G$ such that $P_i\cap gP_jg^{-1}$ is infinite, we have $i=j$ and $g\in P_i$.

\begin{theorem}\label{thm.charRHforH}
Let $\G$ be a hyperbolic group and $\calP$ a finite collection of subgroups of $\G$. Then $(\G, \calP)$ is relatively hyperbolic if and only if $\calP$ is an almost malnormal family of quasiconvex subgroups. 
\end{theorem}

\subsection{Relative quasiconvexity}

Let $(\G,\calP)$ be a relatively hyperbolic pair and $H<\G$ a subgroup. The \emph{induced peripheral structure} on $H$ with respect to $\calP$ is a collection $\calD_H$ consisting of representatives of
$H$–conjugacy classes of infinite groups of the form $H\cap gPg^{-1}$ for $g\in \G$ and $P\in \calP$. 

Suppose that $\calD=\calD_H$ is finite and consists of finitely generated groups. Then we can form the cusped space $X(H,\calD,S')$ for $S'$ an appropriate generating subset of $H$. In this setting there is a depth-invariant and $H$-equivariant Lipschitz map $\iota:X(H,\calD,S')\ra X(\G,\calP,S)$ \cite[Lemma~]{AGM.QCERF}. The group $H$ is \emph{relatively quasiconvex} in $(\G,\calP)$ if the image of $X(H,\calD,S')$ under $\iota$ is quasiconvex in $X(\G,\calP,S)$. In particular, $(H,\calD)$ is itself a relatively hyperbolic pair. By convention, all relatively quasiconvex subgroups of a relatively hyperbolic group pair are considered with the induced peripheral structure.

If $(\G,\calP)$ is a relatively hyperbolic pair, a
relatively quasiconvex subgroup $H$ of $(\G,\calP)$ is \emph{full} if whenever $P\in \calP$ and $g\in \G$ are such that $H \cap gPg^{-1}$ is infinite we have $H \cap gPg^{-1}$ has finite index in $gPg^{-1}$.

When $\G$ is a hyperbolic group, quasiconvex subgroups are relatively hyperbolic with respect to any peripheral structure on $\G$ \cite[Theorem~1.5]{hruska}. 

\subsection{Height of collections of quasiconvex subgroups}

If $\G$ is a hyperbolic group, then finite collections of quasiconvex subgroups induce peripheral structures \cite[Definition~6.2]{groves-manning.improper}. More precisely, let $\calH$ be a finite collection of
quasiconvex subgroups of $\G$, with at least one member of $\calH$ being infinite. The \emph{peripheral structure} $\calP_\calH$ on $\G$ \emph{induced by} $\calH$ is obtained
as follows. 

First, we consider the collection of all the minimal infinite subgroups of the form
$H_1 \cap g_2H_2g_2^{-1}\cap \dots \cap g_kH_kg_k^{-1}$, where $H_1,\dots,H_n\in \calH$ and the cosets $H_1, g_2H_2,\dots,g_kH_k$ are all distinct. Then we replace each element in this collection by its commensurator in $\G$, and then we choose one representative for each $\G$-conjucacy class. The resulting collection $\calP_\calH$ is the induced peripheral structure. There is an upper bound on the number $k$ of the cosets $(g_iH_i)_i$ as above. This is encoded in the notion of height. 

\begin{definition}\label{def:height}
    Let $\G$ be a group and $\calH$ a collection of subgroups of $\G$. The \emph{height} of $\calH$ in $\G$ is the minimum number $k$ such that for every tuple of distinct cosets $(g_0H_0,\dots,g_kH_k)$ with $H_1,\dots,H_k\in \calH$ and $g_0,\dots,g_k\in \G$, the intersection
$\bigcap_{i=0}^k{g_iH_ig_i^{-1}}$ is finite.  If no such $k$ exists, we say the height of $\calH$ in $\G$ is infinite.
\end{definition}

Note that the family has height $0$ if and only if every member of $\calH$ is finite, and that malnormal families of subgroups have height at most one. 

If $\G$ is hyperbolic and $\calH$ is a finite collection quasiconvex subgroups, then the height is finite as proven in \cite[Theorem~1.2~(3)]{tran} in the more general setting of strongly quasiconvex subgroups of finitely generated groups (see also \cite[Proposition~3.29]{groves-manning.improper}). The induced peripheral structure $\calP_\calH$ is then a finite almost malnormal collection of quasiconvex subgroups of $\G$ (see \cite[Lemma~6.4]{groves-manning.improper} or \cite[Proposition~3.12]{AGM.QCERF}) and $(\G,\calP_\calH)$ is relatively hyperbolic pair. 
We summarize these results in the next proposition, for which item (4) is immediate.

\begin{proposition}\label{prop.propertiesheight}
Let $\G$ be a hyperbolic group and $\calH$ a finite collection of quasiconvex subgroups of $\G$. Then:
\begin{enumerate}
    \item The height of $\calH$ in $\G$ is finite.
    \item The peripheral structure $\calP_\calH$ on $\G$ induced by $\calH$ is finite and the pair $(\G,\calP_\calH)$ is relatively hyperbolic. 
    \item Any $H\in \calH$ is full relatively quasiconvex in $(\G, \calP_\calH)$.
    \item If $\Phi$ is a group acting on $\G$ by automorphisms and the group pair $(\G,\calH)$ is $\Phi$-invariant, then the peripheral structure $\calP_\calH$ is $\Phi$-invariant.
\end{enumerate}
\end{proposition}

An extra property that we will need is that the height is monotone under Dehn filling. More precisely, we have the following result, which was implicitely used in the proof of \cite[Theorem~6.9~(5)]{groves-manning.improper} and whose proof follows by an entirely analogous argument to that of \cite[Theorem~A.47]{agol.haken}.

\begin{theorem}\label{thm.heightdecreases}
Let $\G$ be a hyperbolic group, $\calH$ a finite collection of quasiconvex subgroups of $\G$. Let $\calP=\calP_\calH=\{P_1,\dots,P_m\}$ be the peripheral structure on $\G$ induced by $\calH$, and let $\pi:\G \ra \ov\G=\G(N_1,\dots,N_m)$ be a sufficiently long $\calH$-filling. If at least one member of $\calH$ is infinite and each filling kernel $N_i$ has finite index in $P_i$, then the height of $\ov\calH=\{\pi(H):H\in \calH\}$ in $\ov\G$ is strictly less than that of $\calH$ in $\G$.
\end{theorem}

\subsection{Double cosets intersecting filling kernels}\label{subsection.doublecoset}

In this subsection we prove the next result, which asserts that for sufficiently long fillings, there is a control on the intersection of the kernel of the filling with double cosets of full relatively quasiconvex subgroups.  

\begin{proposition}\label{prop.intersectionfillingmulticoset}
Let $(\G,\calP)$ be a relatively hyperbolic group and let $\calH$ be a finite collection of full relatively quasiconvex subgroups. Then for all sufficiently long $\calH$–fillings $\ov \G =\G/K$ and $H_1,H_2\in \calH$ we have
    \begin{equation*}
        K \cap H_1H_2 \subset (K\cap H_1)(K \cap H_2).
    \end{equation*}
\end{proposition}

The proof this result almost the same as that of \cite[Theorem~6.5]{groves-manning.improper}. However, our result does not follow from \cite[Theorem~6.5]{groves-manning.improper}, so we provide a complete proof. This is the only result of the paper on which we require properties about the geometry of the cusped space. 

First, we require the following ``Greendlinger Lemma'' \cite[Theorem~6.7]{groves-manning.improper}.

\begin{theorem}\label{thm.greenlindger}
    Let $C_1,C_2 > 0$. Let $(\G,\calP)$ be a relatively hyperbolic group with cusped space $X$. Then for all sufficiently long fillings $\G \ra \G/K$, and any geodesic $\gam$ in $X$ joining $1$ to $g \in K \bs \{1\}$, there is a horoball $A$ such that:
\begin{enumerate}
\item $\gam$ contains a depth $C_1$ vertex of $A$; and,
\item there is an element $k\in K$ stabilizing $A$ and two points $a, a' \in A$ and lying on $\gam$  at depth at least $C_1$ such that $d_X(a,kb)< d_X(a,b)-C_2$ (in particular, $d_X(1, kg)< d_X(1,g)-C_2$).
\end{enumerate}
\end{theorem}

We also need the following lemma, whose proof follows immediately from \cite[A.6]{manning-martinezpedroza} (cf.~\cite[Lemma~6.8]{groves-manning.improper}).

\begin{lemma}\label{lem.deepintersection}
    Let $(\G, \calP)$ be a relatively hyperbolic group with cusped space $X$ and $H<\G$ a full relatively quasiconvex subgroup. There exists a constant $\ka$ satisfying the following:

Suppose that $g \in \G$ and that $x_1, x_2 \in gH$. Suppose that $\gam$ is a geodesic in $X$ joining $x_1$ and $x_2$. Further, suppose that $uP$ (for $u \in \G$ and $P \in \calP$) is a coset such that $\gam$ intersects the horoball corresponding to $uP$ to depth at least $\ka$. Then $P$ is infinite and $uPu^{-1}\cap gHg^{-1}$ has finite index in $uPu^{-1}$.
\end{lemma}

\begin{proof}[Proof of Proposition \ref{prop.intersectionfillingmulticoset}]
Let $X$ be the cusped space for $(\G, \calP)$, which we assume to be $\del$–hyperbolic. Let $C_2>0$ be any number, and let $C_1 = \ka+ \del+1$, where $\ka$ is a constant such that Lemma \ref{lem.deepintersection} holds for any subgroup belonging to $\calH$. Let $K<\G$ be the kernel of an $\calH$-filling which is long enough to satisfy the conclusion of Theorem \ref{thm.greenlindger} with the constants $C_1,C_2$ as defined above.

Suppose for the sake of contradiction that $(K\cap H_1H_2)\bs [(K\cap H_1)(K\cap H_2)]$ is non-empty for some $H_1,H_2\in \calH$, and choose $g$ in this set minimizing $d_X(1, g)$. Note that $g\neq 1$ and let $\gam$ be a geodesic in $X$ joining $1$ and $g$. By Theorem \ref{thm.greenlindger} there exists a horoball $A$ in $X$, an element $k\in K$ stabilizing $A$ and $a,b \in \gam\cap A$ at depth at least $C_1$ such that $d_X(a,kb)<d_X(a,b)-C_2$. Then $d_X(1,kg)<d_X(1,g)$ and the desired contradiction will be obtained after showing that $kg \in (K\cap H_1H_2)\bs [(K\cap H_1)(K\cap H_2)]$.

In order to obtain this contradiction, write $g=h_1h_2$ for $h_1\in H_1$ and $h_2\in H_2$. Let $\al_1$ be a geodesic joining $1$ and $h_1$, and $\al_2$ be a geodesic in $X$ joining $h_1$ and $h_1h_2$. Then $\gam,\al_1,\al_2$ form a geodesic triangle in $X$. By $\del$-hyperbolicity, $b$ is within distance $\del$ of a point $b'\in \al_i$ for some $i=1,2$. Our choice of $C_1$ implies that $b'$ lies at depth at least $\ka$ in $A$. Suppose that $A=uP$ for some $u\in \G$ and $P \in \calP$.

Suppose first that $i=1$. Since the geodesic $\al_1$ joins the points $1$ and $h_1$ in $H_1$, Lemma \ref{lem.deepintersection} implies that $P$ is infinite and that $uPu^{-1} \cap H_1$ has finite index in $uPu^{-1}$. Since the filling is a $\calQ$–filling, we have that $uPu^{-1}<K \cap H_1$ and that $k$ (which stabilizes $A$) belongs to $K\cap H_1$. Then we have that $kg=(kh_1)h_2$ belongs to $(K\cap H_1H_2)\bs [(K\cap H_1)(K\cap H_2)]$ since $g \notin (K\cap H_1)(K\cap H_2)$. 

Similarly, if $i=2$, and after applying Lemma \ref{lem.deepintersection} to the geodesic $\al_2$ as in the case above, we conclude that $k\in K\cap h_1H_2h_1^{-1}$ and that $kg=h_1(h_1^{-1}kh_1)h_2=h_1h_2k'$ for some $k'\in (K\cap H_2)$, since $h_1^{-1}kh_1\in K\cap H_2$ and $K\cap H_2$ is normal in $H_2$. Again, $kg\in (K\cap H_1H_2)\bs [(K\cap H_1)(K\cap H_2)]$ and the conclusion follows.
\end{proof}

\subsection{The malnormal special quotient theorem}

One of the main features of Dehn filling is that it behaves well with respect to virtual specialness. The following is celebrated Wise's malnormal special quotient theorem \cite[Theorem~12.2]{wise.QCH} (see also \cite[Theorem~2.7]{AGM.MSQT} for an alternate proof and \cite[Theorem~2]{einstein} for a relative version). Recall that by Agol's theorem \cite{agol.haken}, a hyperbolic group is virtually special if and only if it is cubulable.

\begin{theorem}[Malnormal special quotient theorem]\label{thm.MSQT}
Let $\G$ be a hyperbolic cubulable group and let $\calP=\{P_1,\dots,P_m\}$ be a finite, almost malnormal collection of quasiconvex subgroups. Then there exist normal finite index subgroups $P_i'<P_i$ such that for any Dehn filling $\G \ra \ov \G =\G(N_1,\dots,N_m)$ with $N_i<P_i'$ of finite index, the group $\ov \G$ is hyperbolic and cubulable. 
\end{theorem}

Combining this theorem with the previous results of the section, we obtain the next result, which is our main tool to prove Theorem \ref{thm.maingrouptheory}.

\begin{proposition}\label{prop.combinedstatements}
    Let $\G$ be a hyperbolic and cubulable group and $\Phi$ a finite group acting on $\G$ by automorphisms. Suppose $\calH$ is a $\Phi$-invariant finite collection of quasiconvex subgroups of $\G$, and let $\G_0<\G$ be a finite index subgroup. If at least one group in $\calH$ is infinite, then there exists a normal, $\Phi$-invariant subgroup $K_1<\G$ with associated quotient map $\pi_1:\G \ra \G/K_1$ such that:
    \begin{enumerate}
        \item The quotient $\ov\G=\G/K_1$ is hyperbolic and cubulable. 
        \item $K_1< \G_0$.
        \item For any $H\in \calH$ the image $\ov{H}=\pi_1(H)$ is quasiconvex in $\ov\G$. 
        \item The height of $\ov\calH:=\{\ov H: H\in \calH\}$ in $\ov\G$ is strictly less than that of $\calH$ in $\G$.
        \item For all $H_1,H_2\in \calH$ we have
        \[K_1 \cap H_1H_2 \subset (K_1\cap H_1)(K_1 \cap H_2).\]
        \item The action of $\Phi$ on $\G$ descends to an action on $\ov\G$ such that the collection $\ov\calH$ is $\Phi$-invariant.
    \end{enumerate}
\end{proposition}

\begin{proof}
    Let $\calP=\{P_1,\dots,P_m\}$ be the peripheral structure on $\G$ induced by $\calH$, and let $\pi_1:\G \ra \ov\G=\G(N_1,\dots,N_m)=\G/K_1$ be a sufficiently long $\calH$-filling with each $N_i$ finite index in $P_i$. By \cite[Theorem~1.1]{osin.peripheral} (or equivalently, by \cite[Corollary~1.2 \& Corollary~9.7]{groves-manning.dehnfilling}) we have that $\ov\G$ is hyperbolic. Moreover, from \cite[Proposition~4.3]{AGM.QCERF},  Theorem \ref{thm.heightdecreases}, and Proposition \ref{prop.intersectionfillingmulticoset}, we have that properties (3)-(5) also hold for the filling $\pi_1$. Let $S\subset (\bigcup_{i=1}^mP_i)\bs\{1\}$ be a finite set such that all these properties hold as long as $S\cap (\bigcup_i N_i)=\emptyset$. 

    To prove properties (1) and (2), let $P_1',\dots,P_m'$ be the set of subgroups of $P_1,\dots,P_m$ given by Theorem \ref{thm.MSQT}, and let $\G'_0$ be the intersection of all the conjugates of $\G_0$ in $\G$. Note that $\G'_0$ is a finite index normal subgroup of $\G$ that is contained in $\G_0$. Since $\G$ is hyperbolic and cubulable (hence virtually special by \cite{agol.haken}), it is residually finite, as well as all the subgroups $P_1,\dots,P_m$. Therefore we can find filling kernels $N_1,\dots,N_m$ such that \begin{itemize}
        \item[($i$)] $N_i$ has finite index in $P_i'\cap \G_0'$ for all $i$; and,
        \item[($ii$)] $S\cap (\bigcup_i N_i)=\emptyset$. 
    \end{itemize}
    The filling $\ov\G=\G(N_1,\dots,N_m)=\G/K_1$ then satisfies property (1) by ($i$) and (3)-(5) by ($ii$). To check it also satisfies property (2), recall that $K_1$ is the normal closure of $\bigcup_{i}{N_i}$ in $\G$. Then $K_1$ is contained in $\G_0$ since $\G'_0$ is normal in $\G$, and by ($ii$) we have that each $N_i$ is contained in $\G'_0<\G_0$.
     
     We are left to show that we can further ensure $\Phi$-invariance of $K_1$ and property (6). To this end, by Proposition \ref{prop.propertiesheight} we first note that $\calP$ is a $\Phi$-invariant peripheral structure. Then, for each $1\leq i\leq m$ we let $\calA_i$ be the set of all triplets $(\phi,j,gP_j)$ such that $\phi\in \Phi$, $j\in \{1,\dots,m\}$ and $g\in \G$ are such that $P_i=\phi(gP_jg^{-1})$. Note that each $\calA_i$ is finite (here we use that $\Phi$ is finite). 
     We define new filling kernels $\what N_1,\dots,\what N_m$ according to
     $\what N_i
     :=\bigcap_{(\phi,j,gP_j)\in \calA_i}\phi(gN_jg^{-1})$. Then each $\what N_i$ is a finite index normal subgroup of $P_i$ contained in $N_i$ (since $(1,i,P_i)\in \calA_i$). It is not hard to show that these filling kernels are $\Phi$-invariant, and hence the kernel $K_1$ associated to $\G\ra \G(\what N_1,\dots,\what N_m)$ is $\Phi$-invariant and all properties (1)-(6) above hold. This concludes the proof of the proposition.
\end{proof}

%%%%%%%%%%%%%%%%%%%%%%%%%%%%%%%%%%%%%%%%%%%%%%%%%%%%%%%%%

\subsection{Proof of Theorem \ref{thm.maingrouptheory}}

We end this section with the proof of Theorem \ref{thm.maingrouptheory}, for which we need a few preliminary lemmas.

\begin{lemma}\label{lem.compositionhomos}
    Let $\pi_1:L_0 \ra L_1$ and $\pi_2: L_1 \ra L_2$ be group homomorphisms and let $H_1,H_2<L_0$ be two subgroups such that
    \begin{enumerate}
        \item $\ker\pi_1 \cap H_1H_2 \subset (\ker \pi_1 \cap H_1)(\ker \pi_1 \cap H_2)$.
        \item $\ker\pi_2 \cap \pi_1(H_1)\pi_1(H_2) \subset (\ker \pi_2 \cap \pi_1(H_1))(\ker \pi_2 \cap \pi_1(H_2))$.
    \end{enumerate}

    Then $$\ker(\pi_2 \circ \pi_1) \cap H_1H_2 \subset (\ker(\pi_2 \circ \pi_1) \cap H_1)(\ker(\pi_2 \circ \pi_1) \cap H_2).$$
\end{lemma}

\begin{proof}
    Let $K_i=\ker\pi_i$ for $i=1,2$, $\psi=\pi_2\circ \pi_1$ and $\what K=\ker \psi$. Consider $k=h_1h_2 \in \what K \cap H_1H_2$ with $h_1\in H_1$ and $h_2\in H_2$. Then $\pi_1(k)=\pi_1(h_1)\pi_1(h_2)\in K_2 \cap \pi(H_1)\pi(H_2) \subset (K_2\cap \pi_1(H_1))(K_2\cap \pi_1(H_2))$. But $K_2\cap \pi_1(H_1)=\pi_1(\what K \cap H_1)$ and $K_2\cap \pi_1(H_2)=\pi_1(\what K \cap H_2)$, and hence  $\pi_1(h_1)\pi_1(h_2)=\pi_1(h_1')\pi_1(h_2')$ for $h_1'\in \what K \cap H_1$ and $h_2'\in \what K \cap H_2$. Then $h_1h_2=h_1'uh_2'$ for some $u\in K_1$. But $u=(h_1')^{-1}h_1h_2(h_2')^{-1}\in K_1 \cap H_1H_2\subset (K_1\cap H_1)(K_1 \cap H_2)$, so $u=\what h_1 \what h_2$ for $\what h_1\in K_1\cap H_1$ and $\what h_2\in K_1 \cap H_2$, concluding $k=h_1h_2=h_1'\what h_1 \what h_2 h_2'\in (\what K \cap H_1)(K_1 \cap H_1)(K_1 \cap H_2)(\what K \cap H_2)=(\what K \cap H_1)(\what K \cap H_2)$. 
    \end{proof}

The next lemma is immediate.
\begin{lemma}\label{lem.finitesubgroups}
    Let $\G$ be a residually finite group and let $\calH$ be a finite collection of finite subgroups of $\G$. Then for any finite  index subgroup $\G_0 <\G$ there exists a normal, finite index subgroup $K<\G$ such that
    \begin{itemize}
        \item $K< \G_0$; and,
        \item $K \cap H_1H_2 =\{1\}$ for all $H_1,H_2\in \calH$.
    \end{itemize}
  Moreover, if $\Phi$ is a finite subgroup acting on $\G$ by automorphisms and the collection $\calH$ is $\Phi$-invariant, then $K$ can be chosen to be $\Phi$-invariant. 
\end{lemma}

\begin{proof}[Proof of Theorem \ref{thm.maingrouptheory}]
Let $\G,\Phi,\calH$ and $\G_0$ be as in the statement of the theorem. We will prove the result by induction on the height $k$ of $\calH$ in $\G$. 

If $k=0$ then each $H\in \calH$ is finite, and the result follows from Lemma \ref{lem.finitesubgroups} since $\G$ is residually finite by virtual specialness. Suppose now that $k$ is positive and let $K_1<\G$ be the $\Phi$-invariant normal subgroup given by Proposition \ref{prop.combinedstatements} with associated quotient $\pi: \G \ra \G/K_1=:\ov\G$. Let $\ov H=\pi_1(H)$ for each $H\in \calH$, and let $\ov\calH=\{\ov H:H\in \calH\}$. 
Since $K_1$ is $\Phi$-invariant, there exists a natural automorphism action of $\Phi$ on $\ov\G$. Moreover, we have:
\begin{enumerate}
    \item $ \ov{\G}_0:=\pi_1( \G_0)$ is a finite index subgroup of $\ov \G$;
    \item $\ov\calH$ is a $\Phi$-invariant collection of quasiconvex subgroups, so that the height of $\ov\calH$ is less than $k$; and,
    \item $K_1 \cap H_1H_2 \subset (K_1 \cap H_1)(K_1 \cap H_2)$ for all $H_1,H_2 \in \calH$.
\end{enumerate}  
By our inductive assumption, there exists a $\Phi$-invariant, finite index normal subgroup $\ov K_2 <\ov \G$ such that:
\begin{enumerate}
    \item $\ov K_2 < \ov \G_0$; and,
    \item $\ov K_2 \cap \ov H_1\ov H_2 \subset (\ov K_2\cap \ov H_1)(\ov K_2 \cap \ov H_2)$ for all $H_1,H_2\in \calH$. 
\end{enumerate}
If $\pi_2:\ov\G \ra \ov\G/\ov K_2$ is the quotient homomorphism and $K:=\ker(\pi_2 \circ \pi_1)<\G$, then $K$ is finite index in $\G$ and $\Phi$-invariant. Moreover, we have $K <K_1\G_0=\G_0$, and Lemma \ref{lem.compositionhomos} implies that
\[K\cap H_1H_2 \subset (K\cap H_1)(K\cap H_2)\]
for all $H_1,H_2\in \calH$. This concludes the proof by induction and hence of the theorem.    
\end{proof}

%%%%%%%%%%%%%%%%%%%%%%%%%%%%%%%%%%%%%%%%%%%%%%%%%%%%%%%%%%%%%%%%%%%%%%
    
\bibliographystyle{amsalpha}
\bibliography{refs1}

\end{document}